\documentclass[11pt,reqno,UTF8,twoside]{amsart}
\usepackage{mathrsfs}
\usepackage{amsfonts,amssymb,amsmath,amsthm}
\usepackage{cite}
\usepackage[colorlinks=true,citecolor=red]{hyperref}
\usepackage{titletoc}
\usepackage{geometry}
\usepackage{bm}
\usepackage{indentfirst}
\usepackage{graphicx}
\usepackage{float} 
\usepackage{booktabs}
\usepackage{longtable}
\usepackage{subfigure}
\usepackage{stmaryrd}
\usepackage{enumerate}
\usepackage{tikz}
\usepackage[normalem]{ulem}
\usepackage{color,soul}
\usepackage{setspace}
\usepackage{stmaryrd}
\usepackage[pagewise]{lineno} 
\allowdisplaybreaks[2]
\usepackage{appendix}
\usepackage{cases}
\usepackage{todonotes}
\usepackage{enumitem}

\numberwithin{equation}{section}
\newtheorem{theorem}{Theorem}[section]
\newtheorem{lemma}[theorem]{Lemma}

\newtheorem{proposition}[theorem]{Proposition}
\newtheorem{maintheorem}{Theorem}

\newtheorem*{motivation}{Motivation}

\newtheorem{remark}[theorem]{Remark}

\newcommand{\N}{\mathbb{N}}
\newcommand{\R}{\mathbb{R}}

\newcommand{\E}{\mathbb{E}}

\newcommand{\Pb}{\mathbb{P}}

\newcommand{\dvg}{\operatorname{div}}
\newcommand{\curl}{\operatorname{curl}}

\newcommand{\law}{\operatorname{Law}}
\newcommand{\dist}{\operatorname{dist}}
\newcommand{\runum}[1]{\romannumeral #1}
\hypersetup{linkcolor=blue}

    \makeatletter
    \@namedef{subjclassname@2020}{\textup{2020} Mathematics Subject
	Classification}
    \makeatother

\begin{document}

     \title[ Mixing for Navier--Stokes system ]
	{{\Large E{\MakeLowercase{xponential mixing for the stochastic} N{\MakeLowercase{avier--}S{\MakeLowercase{tokes equation}\\
    \vspace{1mm}
    \MakeLowercase{ with localized noise}}}}
    }}

    \author[Z. Liu, S. Xiang, Z. Zhang]{ {\small Z\MakeLowercase{iyu} L\MakeLowercase{iu}, S\MakeLowercase{hengquan} X\MakeLowercase{iang}, Z\MakeLowercase{hifei} Z\MakeLowercase{hang}}}
    
    \address[Ziyu Liu]{School of Mathematics and Physics, University of Science and Technology Beijing, 100083, Beijing, China.}
    \email{ziyu@ustb.edu.cn}

    \address[Shengquan  Xiang]{School of Mathematical Sciences, Peking University, 100871, Beijing, China.}
    \email{shengquan.xiang@math.pku.edu.cn}

    \address[Zhifei  Zhang]{School of Mathematical Sciences, Peking University, 100871, Beijing, China.}
    \email{zfzhang@math.pku.edu.cn}
 
    \begin{abstract}
     This paper studies the 2D Navier--Stokes equation on a bounded domain with Navier-slip boundary conditions, driven by spatially localized stationary forcing generated by a stochastic heat equation.  We prove exponential mixing for the associated Navier--Stokes--heat system, and consequently exponential convergence of the velocity law under stationary forcing.
     
     The main difficulty is that the white noise is confined to a subdomain and reaches the velocity indirectly through the heat component. To address the degeneracy, the proof combines Malliavin calculus with PDE control theory. The key ingredient is a stabilization scheme that converts localized Navier--Stokes controls into time-regular controls compatible with the heat dynamics.

    \end{abstract}

	

    \maketitle
    \setcounter{tocdepth}{1}
    \tableofcontents

    \section{Introduction}\label{Sec 1}

    \subsection{Background and motivation}\label{Sec 1.1}

    Ergodicity and mixing for SPDEs, central to the statistical description of physical systems, have been extensively studied in recent decades. They are by now well understood when all determining modes are directly forced, but much less so when the noise reaches them only indirectly.

    In \cite{HM-06,HM-08,HM-11b}, Hairer and Mattingly studied ergodicity and mixing for the 2D stochastic Navier--Stokes equations driven by highly degenerate white noise in Fourier space, including configurations with only four directly forced modes. In the study of mixing for Navier--Stokes equations driven by random forcing, Shirikyan considered another direction, where the noise is localized in physical space and bounded in time \cite{Shi-15,Shi-21}. Along this line, the authors and their collaborators recently established mixing results for dispersive equations driven by spatially localized bounded noise \cite{LWXZZ-24,CXZZ-25,CXZZ-26}.

    \vspace{0.6em}

    In contrast, much less is known for SPDEs driven by spatially localized white noise, where the white-in-time forcing acts only on a proper subdomain. In particular, establishing ergodicity and mixing for the Navier–Stokes system under such forcing is quite challenging. Motivated by this problem, the authors recently took a first step in this direction by proving exponential mixing for the 1D stochastic Allen–Cahn equation with localized white noise \cite{LXZ-26}. 

    \begin{motivation}
    Turbulence and statistical properties of the Navier--Stokes equation with localized white noise.
    \end{motivation}

    \vspace{0.3em}

    In the present paper, we study the 2D Navier--Stokes equation on a bounded domain with Lions-type Navier-slip boundary conditions, driven by a spatially localized stationary force generated by a stochastic heat equation. We establish exponential mixing for the resulting Markovian Navier--Stokes--heat system and, consequently, exponential convergence of the velocity law. Our approach combines stochastic analysis and PDE control theory, with quantitative stabilization of the coupled linearized system providing the key link between localized forcing and asymptotic regularization.

    \subsection{Main results}\label{Sec 1.2}
    Let $D\subset\R^2$ be a bounded simply connected domain with smooth boundary, and let $\omega\Subset D$ be a non-empty smooth subdomain. We consider the stochastic Navier--Stokes equation on $D$ with Navier-slip boundary conditions and a localized random force, which reads
	\begin{equation}\label{eq NS}
	\begin{cases}
	\partial_tu-\nu\Delta u+(u\cdot\nabla)u+\nabla\pi=\overline{\xi}(t,x),\quad x\in D,\;t>0,\\
	\dvg u=0,\\
	u\cdot n=0,\quad\curl u=0,\quad x\in\partial D,\\
	u(0,\cdot)=u_0(\cdot).
	\end{cases}
	\end{equation}
    Here $\nu>0$ is the viscosity, $n$ is the outward unit normal on $\partial D$, and $\overline{\xi}$ denotes the zero extension to $D$ of a random field $\xi$ on $\omega$ given by
	\begin{equation}\label{eq xi}
	\overline{\xi}(t,x)=\begin{cases}
	\xi(t,x),\quad&x\in\omega,\\
	0,\quad&x\in D\setminus\omega,
	\end{cases}
	\end{equation}
    where $\xi(t,x)$ is a random process, as specified below.
    
    \vspace{0.6em}

     We consider the Navier--Stokes system \eqref{eq NS} on $\mathcal H^1$-space, which is given by
	\begin{equation*}
	\mathcal H:=\left\{u\in L^2(D;\R^2):\dvg u=0\ \text{in }D,\;u\cdot n=0\ \text{on }\partial D\right\},\quad \mathcal H^1:=H^1(D;\R^2)\cap \mathcal H,
	\end{equation*}
    endowed with the usual $L^2$-norm $\|\cdot\|$ and $H^1(D;\R^2)$-norm, respectively.

    \vspace{0.6em}

    To formulate the localized noise structure, let $\{\psi_j\}_{j\in\N^+}\subset H_0^1(\omega;\R^2)$ be an orthonormal basis of $L^2(\omega;\R^2)$ consisting of eigenvectors of the Dirichlet Laplacian on $\omega$, namely
	\begin{equation*}
	-\Delta_{\omega}\psi_j=\lambda_j\psi_j,\quad\psi_j|_{\partial\omega}=0,\quad j\in\N^+,
	\end{equation*}
    where $\{\lambda_j\}_{j\in\N^+}$ are the corresponding eigenvalues satisfying $0<\lambda_1\leq\lambda_2\leq\cdots$ and  $\lambda_j\rightarrow\infty$. The Hilbert space $L^2(\omega;\R^2)$ is identified with a subspace of $L^2(D;\R^2)$ by extending functions by zero outside $\omega$.

    \vspace{0.3em}
    The localized stationary noise is then defined as follows:
	\begin{equation}\label{eq eta_s}
	\eta^{s}(t,x):=\sum_{j\in\N^+}b_j\int_{-\infty}^{t}e^{-\lambda_j(t-s)}d\hat\beta_j(s)\psi_j(x),\quad t\in\R,
	\end{equation}
    where $b_j$ are real numbers satisfying $\sum_{j\in\N^+}b_j^2<\infty$, and $\{\hat\beta_j(t)\}_{j\in\N^+}$ is a sequence of independent two-sided standard Brownian motions defined on a filtered probability space.

    \vspace{0.6em}

    Our main result concerning exponential mixing under localized stationary noise is as follows.    

    \begin{maintheorem}\label{thm 1}
    For the equation \eqref{eq NS},\eqref{eq xi} with $\xi=\eta^{s}$ given by \eqref{eq eta_s}, and any $B_0>0$, there exists an integer $N=N(B_0)\in\N^+$ such that the following holds. If $\{b_j\}_{j\in\N^+}$ satisfies
	\begin{equation*}
	\sum_{j\in\N^+}b_j^2\leq B_0\quad\text{and}\quad b_j\neq0\quad\text{for }1\leq j\leq N,
	\end{equation*}
    then there exists a probability measure $\mu$ on $\mathcal H^1$ and constants $C,\alpha,\gamma>0$ such that
	\begin{equation}\label{eq mixing1}
	\left|\E\varphi(u(t,u_0))-\int_{\mathcal H^1}\varphi d\mu\right|\leq Ce^{-\alpha t}\exp\left(C\gamma\|u_0\|_{\mathcal H^1}^2\right)\|\varphi\|_{\mathcal H^1,\gamma}
	\end{equation}
    for any $\varphi\in\mathcal{O}_{\mathcal H^1,\gamma}$, $u_0\in \mathcal H^1$ and $t\geq0$, where
	\begin{equation*}
	\mathcal{O}_{\mathcal H^1,\gamma}=\left\{\varphi\in C^1(\mathcal H^1):\|\varphi\|_{\mathcal H^1,\gamma}<\infty\right\},\quad\|\varphi\|_{\mathcal H^1,\gamma}=\sup_{u\in \mathcal H^1}\left(e^{-\gamma\|u\|_{\mathcal H^1}^2}(|\varphi(u)|+\|\nabla\varphi(u)\|_{\mathcal H^1})\right).
	\end{equation*}
    Moreover, there exists a  stationary solution $u^s$ of equation \eqref{eq NS} such that the pair $(u^s,\eta^s)$ is jointly stationary and
	\begin{equation}\label{eq stationary}
	\law(u^s(t))=\mu\quad\forall\,t\in\R.
	\end{equation}
    \end{maintheorem}

    The forcing $\eta^s$ in Theorem~\ref{thm 1} is stationary but temporally correlated, and hence the velocity component alone does not form a Markov process. It is therefore natural to include the dynamics generating the forcing as an additional component of the state space. Such Markovian extensions provide a useful way to study systems driven by colored  noise; see, e.g. \cite{HO-07,KS-25}.

    \vspace{0.3em}
    More precisely, $\eta^s$ is the stationary realization of the stochastic heat equation on $\omega$ given by
	\begin{equation}\label{eq eta}
	\begin{cases}
	d\eta-\Delta_{\omega}\eta dt=\displaystyle\sum_{j\in\N^+}b_j\psi_j(x)d\beta_j(t),\quad x\in\omega,\;t>0,\\
	\eta|_{\partial\omega}=0,\\
	\eta(0,\cdot)=\eta_0(\cdot),
	\end{cases}
	\end{equation}
    where $b_j$ are real numbers, and $\{\beta_j(t)\}_{j\in\N^+}$ is a sequence of independent standard Brownian motions defined on a filtered probability space $(\Omega,\mathcal{F},\{\mathcal{F}_t\}_{t\geq0},\Pb)$.

    \vspace{0.3em}
    We therefore consider the Markovian lift $U=(u,\eta)$ of \eqref{eq NS},\eqref{eq eta} on phase space 
	\begin{equation*}
	\mathcal X:=\mathcal H^1\times L^2(\omega;\R^2),\quad\|U\|_{\mathcal X}^2=\|(u,\eta)\|_{\mathcal X}^2=\|u\|_{\mathcal H^1}^2+\|\eta\|^2,
	\end{equation*}
    where the $L^2$-norm on $\omega$ is denoted by $\|\cdot\|$.

    \vspace{0.6em}
    Theorem \ref{thm 1} follows from the following stronger mixing result for the Markovian lift $(u,\eta)$.

    \begin{maintheorem}\label{thm 2}

    For the equation given by \eqref{eq NS},\eqref{eq xi} with $\xi=\eta$ given by \eqref{eq eta}, and any $B_0>0$, there exists an integer $N=N(B_0)\in\N^+$ such that the following holds. If $\{b_j\}_{j\in\N^+}$ satisfies
	\begin{equation*}
	\sum_{j\in\N^+}b_j^2\leq B_0\quad\text{and}\quad b_j\neq0\quad\text{for }1\leq j\leq N,
	\end{equation*}
    then the Markov process associated to the coupled system \eqref{eq NS},\eqref{eq eta} admits a unique invariant measure $\mu_*$ on $\mathcal X$.  Moreover, there exist constants $C,\alpha,\gamma>0$ such that
	\begin{equation}\label{eq mixing2}
	\left|\E\varphi(U(t,U_0))-\int_{\mathcal X}\varphi d\mu_*\right|\leq Ce^{-\alpha t}\exp\left(C\gamma\|U_0\|_{\mathcal X}^2\right)\|\varphi\|_{\mathcal X,\gamma}
	\end{equation}
    for any $\varphi\in\mathcal{O}_{\mathcal X,\gamma}$, $U_0\in \mathcal X$ and $t\geq0$, where
	\begin{equation*}
	\mathcal{O}_{\mathcal X,\gamma}=\left\{\varphi\in C^1(\mathcal X):\|\varphi\|_{\mathcal X,\gamma}<\infty\right\},\quad\|\varphi\|_{\mathcal X,\gamma}=\sup_{U_0\in \mathcal X}\left(e^{-\gamma\|U_0\|_{\mathcal X}^2}(|\varphi(U_0)|+\|\nabla\varphi(U_0)\|_{\mathcal X})\right).
	\end{equation*}
    \end{maintheorem}

   Let us mention that, to the best of our knowledge, Theorem~\ref{thm 2} provides a first exponential mixing result for a Navier--Stokes system in which white-in-time perturbation is injected only through a localized physical region. In our setting, the noise acts on finitely many modes of the local heat component and is transmitted into the velocity indirectly. The result can be viewed as our first step toward the mixing problem for Navier--Stokes equations driven directly by spatially localized white noise.  We recently learned that Hernández-Santamaría, Le Balc'h and Peralta are independently investigating related questions from a different perspective.

    Additionally, as this paper focuses on the 2D Navier--Stokes equation on a bounded domain with Navier-slip boundary conditions, we mention that the approach could be adapted to other systems with localized white noise, as well as to other noise models such as boundary or Lévy-type noise. The mixing results and underlying methods may also be useful for studying other statistical properties, e.g. large deviations.

    \subsection{Obstructions and ingredients}\label{Sec 1.3}
    
    Building on the approach in our earlier work \cite{LXZ-26,LX-26}, our strategy for proving Theorem~\ref{thm 2} combines a general probabilistic criterion for exponential mixing with a quantitative stabilization result for the coupled linearized system. 
    \vspace{0.6em}

    \noindent{1.3.0. \it Probabilistic framework.}

    \vspace{0.3em}

    The general criterion for exponential mixing that we apply is due to Hairer and Mattingly \cite{HM-06,HM-08,HM-11b}, i.e. Theorem~\ref{thm HM}. This framework reduces mixing to verifying three conditions:
    \begin{itemize}
    \item[(\runum{1})] the existence of a suitable Lyapunov function;
    \item[(\runum{2})] a weak form of irreducibility;
    \item[(\runum{3})] the asymptotic strong Feller property.
    \end{itemize}
    Among these, the last property contributes to the main challenges in the present setting, due to both the localized nature of the noise and the nonlinear turbulence in the Navier--Stokes dynamics.

    \vspace{0.6em}

    \noindent{1.3.1. \it Two obstructions to mixing.}

\vspace{0.3em}

\noindent {\bf Asymptotic strong Feller.}
The first difficulty is to establish the asymptotic strong Feller property. Indeed, such conditions are in many cases necessary for mixing; see e.g. \cite{GLLL-24,GLLL-25}.  Roughly speaking, this amounts to proving a gradient estimate of the form
\begin{equation*}
    \|\nabla P_t\varphi(U_0)\|_{\mathcal X}
    \leq C(\|U_0\|_{\mathcal X})
    \left(\|\varphi\|_{\infty}
    +\delta_t\|\nabla\varphi\|_{\infty}\right),
    \qquad \delta_t\to0,
\end{equation*}
for $C^1$-smooth observables $\varphi$. 

In classical settings, several mechanisms are available.  For systems driven by non-degenerate
white noise, one may establish the strong Feller property; see, e.g.  \cite{FM-95}.  When  the determining modes are directly forced, mixing can often be proved by combining Girsanov’s theorem with the
contractive effect of the dynamics; see, e.g. \cite{EMS-01,BKL-02,Mattingly-02}. For highly degenerate forcing, Hairer and Mattingly \cite{HM-06,HM-11b} developed a Malliavin-calculus approach exploiting the structure of the Navier--Stokes nonlinearity. 

However, in the present setting, the white noise acts only on the localized heat component and reaches the velocity indirectly, so these mechanisms cannot be applied directly. Following the approach established in our earlier work \cite{LXZ-26,LX-26}, we combine Malliavin calculus with PDE control theory.   This reduces asymptotic regularization to a stabilization problem for the coupled linearized system, which eventually leads to a weak observability inequality.

\vspace{0.6em}

\noindent {\bf Weak observability inequality.}
Let $Z=(w,q)$ denote  the solution of the adjoint coupled linearized system. The weak observability estimate required in our argument has the form
\begin{equation*}
    \|Z(0)\|_{\mathcal X}
    \leq C_N
    \|\mathsf P_N q\|_{L^2(0,1;L^2(\omega))}
    +\varepsilon_N\|Z(1)\|_{\mathcal X},
    \qquad \varepsilon_N\longrightarrow0 .
\end{equation*}
For linearized Navier--Stokes equations, local observability is usually derived from Carleman estimates; see e.g. \cite{Guerrero-06,FGIP-04}.  However, the situation here is substantially more involved. Indeed, the observation is available only through the heat component $q$ on the localized region $\omega$. Moreover, only finitely many modes of $q$ are observed. Thus a standard Carleman estimate for the velocity equation does not directly provide the observation appearing above.

The main difficulty is to transfer the partial observation into control of the full coupled state. In particular, the unobserved high-frequency modes must be absorbed into the vanishing remainder $\varepsilon_N\|Z(1)\|_{\mathcal X}$. Moreover, the observability constant must retain ``regular enough'' quantitative dependence on the non-autonomous reference trajectory, which is essential for the subsequent stochastic argument. This requirement forms one of the major difficulties of the present work.

\vspace{0.6em}

\noindent{1.3.2. \it Ingredients of the proof.}

\vspace{0.6em}

\noindent {\bf Navier--Stokes with Navier-slip boundary conditions.}
A first feature of our setting is that the Navier--Stokes equation is posed on a {\it general smooth bounded domain} with Navier-slip boundary conditions. These conditions retain the geometry of a physical boundary while allowing tangential slip, and provide a natural setting for the analysis of viscous flows on bounded domains. In particular, closely related Navier-type boundary conditions have played an important role in the study of vanishing-viscosity limits and boundary layers problems; see, e.g. \cite{Kelliher-06, IS-11}.

From the control viewpoint, bounded domains bring additional difficulties since the boundary behavior has to be incorporated into the controllability analysis. Navier slip conditions provide a useful setting for treating these effects and have been considered in global controllability problems for the Navier--Stokes equation; see, e.g. \cite{CMS-20}. Related work toward the Lions problem with no-slip boundary conditions also highlights the difficulties caused by viscous boundary layers \cite{CMSZ-19}.

The bounded-domain setting thus provides a natural model for studying localized stochastic forcing in the presence of physical boundaries. A further direction is to treat the no-slip Dirichlet condition, where stronger boundary-layer effects make both the control and mixing problems more delicate.

\vspace{0.6em}

    \noindent {\bf Regularization via a coupled structure.}
    Our original motivation was to treat the Navier--Stokes equation driven directly by spatially localized white noise. Following our series of work \cite{LXZ-26,LX-26}, the main issue of asymptotic strong Feller is transformed into a weak observability for the linearized system, which is therefore approached by Carleman estimates. The main obstruction is the low temporal regularity of white noise, which makes the required control construction difficult.

    On the other hand, we observe that the temporal-regularity obstruction can be overcome by inserting an auxiliary dissipative dynamics between the control input and the velocity equation. Indeed, a closely related idea appears in closed-loop stabilization, where additional dynamics are introduced to obtain controls compatible with the evolution; see, e.g. \cite{CP-91,Xiang-18}.   Motivated by this connection, we incorporate a heat component into the mixing problem. A prescribed localized path $\zeta$ can then be realized through  $h=(\partial_t-\Delta_\omega)\zeta$,  which provides the temporal regularity required in the control construction.

    Meanwhile, the same coupled structure also has a natural probabilistic interpretation. The auxiliary heat equation generates a spatially localized stationary forcing for the Navier--Stokes equation. Such random forcing is widely used to produce statistically stationary turbulent flows; see, e.g. \cite{EP-88}. Thus the coupled formulation {\it naturally connects the control regularization with a physically relevant stationary-noise model}.
 
\vspace{0.3em}

\noindent {\bf Regularized stabilization along trajectory with localized force.}
A further ingredient is the {\it construction of regularized controls} for the coupled linearized system. The first regularization concerns the PDE control itself. Starting from a Carleman estimate for the linearized Navier--Stokes equation under Navier-slip boundary conditions \cite{Guerrero-06}, we use Fenchel--Rockafellar duality
    \begin{equation*}
	\inf_{h\in H_{0,1}}J(h)=\sup_{Z\in \mathcal X}J^*(Z)
	\end{equation*}
to construct localized controls with sufficient spatial and temporal regularity for the finite-mode truncation and subsequent lifting through the heat component.

The second regularization is variational and is tailored to the Malliavin argument. After deriving a weak observability inequality for the coupled system, we introduce a regularized inverse of the controllability Gramian and obtain the control
\begin{equation*}
    h=-\mathcal A_{N,g}^*
    (\mathcal G_{N,g}+\beta)^{-1}\mathcal J_gV_0 .
\end{equation*}
Besides providing quantitative stabilization, this explicit structure permits sufficiently precise estimates on the control and its Malliavin derivative, which are needed in the stochastic analysis.

The proof therefore uses controllability--observability duality in both directions. A Carleman estimate first yields observability and, via Fenchel--Rockafellar duality, a regular localized control. After finite-dimensional lifting, approximate controllability is converted back into weak observability and then into the regularized control above, providing the regularity and structure required by the Malliavin argument.

    \subsection{Review of the literature}\label{Sec 1.4}

    The study of ergodicity and mixing for randomly forced PDEs has been a central motivation for the development of ergodic theory for Markov processes. Significant progress has been made, particularly for the 2D Navier--Stokes equations. Early works addressing cases where all the determining modes are perturbed include e.g. \cite{EMS-01, KS-01, BKL-02, FM-95, Mattingly-02, MY-02,NZ-24}.   Hairer and Mattingly \cite{HM-06, HM-08} applied the Malliavin calculus to provide the first results for systems with a highly degenerate white noise, where they introduced the asymptotic strong Feller property. More recently, Kuksin, Nersesyan, and Shirikyan \cite{KNS-20} used a controllability approach to establish similar results for systems perturbed by bounded degenerate colored noise. Further developments in the study of extremely degenerate noise can be found in \cite{HM-11b, FGRT-15, KNS-20-1,PZZ-24}. In \cite{Shi-15, Shi-21}, Shirikyan proposed a controllability approach to investigate systems perturbed by physically bounded localized random forces.

    \vspace{0.6em}

    Another important issue addressed in this paper is the controllability of the Navier--Stokes equations, a topic that has been extensively studied over the past few decades. Specifically, we shall focus on global controllability and stability along trajectories.   Let us mention, e.g. \cite{EGG-16,FGIP-04,FI-96, CL-14,Guerrero-06,EGGP-12} for references on local controllability. In \cite{FGIP-04}, based on this linearization approach and the global Carleman estimate method introduced by Fursikov--Imanuvilov \cite{FI-96}, Fernández-Cara--Guerrero--Imanuvilov--Puel proved the local exact controllability of the Navier--Stokes equations. More recently, the second author introduced a constructive method to quantitative null control in \cite{Xiang-23, Xiang-24},  invoking spectral inequalities and frequency Lyapunov functions.  
    
    The problem of global controllability is less well-understood. In recent years, this method has been applied to make significant contributions to the Lions' problem \cite{CMS-20,CMSZ-19}. In \cite{AS-06}, Agrachev and Sarychev used the geometric control theory to derive global controllability on the torus. This result is further improved by Nersesyan and Rissel in \cite{NR-25} by combining the two above-mentioned methods.

    The study of local exponential stabilization of the Navier--Stokes equations has also yielded fruitful results. Notably, based on the Riccati method, various results have been achieved; see e.g. \cite{BT-04,BT-11} for local exponential stabilization with  internal controls, to \cite{BLT-06,Raymond-07}  with boundary controls.  We also refer to e.g. \cite{CX-21} for quantitative stabilization results for the Burgers equation. Finally, based on the frequency Lyapunov method \cite{Xiang-24}, the second author established quantitative rapid stabilization for the Navier--Stokes system \cite{Xiang-23}. Related stability questions for the Vlasov--Navier--Stokes system were studied in \cite{GHM-18}.

     \subsection*{Organization of the paper} 
     This paper is organized as follows. Section~\ref{Sec 2} collects probabilistic preliminaries and relevant notions from Malliavin calculus. Section~\ref{Sec 3} establishes quantitative stabilization estimates for the coupled linearized system. In Section~\ref{Sec 4}, we prove the asymptotic strong Feller property and irreducibility, and then apply the abstract mixing criterion to prove Theorem~\ref{thm 2} and derive Theorem~\ref{thm 1}. The Appendix contains auxiliary estimates and proofs.

    \section{Probabilistic setup and Malliavin calculus}\label{Sec 2}

    In this section, we first summarize the mathematical setup of the mixing problem for the equation \eqref{eq NS},\eqref{eq eta}. Additionally, we introduce the Malliavin calculus setting, which will play a crucial role in the subsequent analysis.

    \subsection{Well-posedness and Markovian framework}\label{Sec 2.1} To begin with, let us formulate system \eqref{eq NS},\eqref{eq eta} as an abstract evolution equation and define its associated Markovian framework.

    \vspace{0.3em}

    Recall that $\mathcal H$ and $\mathcal H^1$ are given by Section \ref{Sec 1.2}. Set
	\begin{equation*}
	\mathcal H^2=\{u\in H^2(D;\R^2)\cap \mathcal H^1:\curl u=0\ {\rm on}\ \partial D\}.
	\end{equation*}
    On $\mathcal H^2$ we use the standard Sobolev norm $H^2(D;\R^2)$. Let $\Pi$ be the Leray projection onto $\mathcal H$, and define
	\begin{equation*}
	A\colon D(A)=\mathcal H^2\subset \mathcal H\rightarrow \mathcal H,\quad Au=-\nu\Pi\Delta u,\quad B(u,v)=\Pi((u\cdot\nabla)v).
	\end{equation*}
    For the localized heat component, the zero extension from $\omega$ to $D$ will be used without further comment. The corresponding norms are given by
	\begin{equation*}
	\|\eta\|^2=\int_{\omega}|\eta(x)|^2dx,\quad\|\eta\|_{H^1(\omega)}^2=\int_{\omega}|\eta(x)|^2dx+\int_{\omega}|\nabla\eta(x)|^2dx.
	\end{equation*}

    With these notations, the coupled system \eqref{eq NS},\eqref{eq eta} can be written as an abstract evolution equation $U=(u,\eta)$ on $\mathcal X$:
    \begin{equation}\label{eq NS2}
	\begin{cases}
	\partial_tu+Au+B(u,u)=\Pi\,\overline{\eta},\\
	d\eta-\Delta_{\omega}\eta dt=\sum_{j\in\N^+}b_j\psi_jd\beta_j(t),\\
	U(0)=U_0=(u_0,\eta_0)\in \mathcal X.
	\end{cases}
	\end{equation}
    We say that $U=U(t,U_0)$ is a solution of \eqref{eq NS2} if it is $\mathcal F_t$-adapted,
	\begin{equation}\label{eq NS3}
	U\in C([0,\infty);\mathcal X),\quad u\in L^2_{loc}([0,\infty);\mathcal H^2),\quad\eta\in L^2_{loc}([0,\infty);H_0^1(\omega;\R^2))\quad a.s.,
	\end{equation}
    and satisfies \eqref{eq NS2} in the mild sense, that is,
	\begin{equation*}
	\begin{cases}
	u(t)=e^{-At}u_0-\displaystyle\int_0^te^{-A(t-r)}B(u(r),u(r))dr+\displaystyle\int_0^te^{-A(t-r)}\Pi\overline{\eta}(r)dr,\\
	\eta(t)=e^{t\Delta_{\omega}}\eta_0+\displaystyle\sum_{j\in\N^+}b_j\displaystyle\int_0^te^{(t-r)\Delta_{\omega}}\psi_jd\beta_j(r).
	\end{cases}
	\end{equation*}

    \vspace{0.3em}

    Using $\{\beta_j(t)\}_{j\in\N^+}$, let the cylindrical Wiener process $W$ on $L^2(\omega;\R^2)$ be given by 
	\begin{equation*}
	W(t)=\sum_{j\in\N^+}\beta_j(t)\psi_j.
	\end{equation*}
    Let $\mathsf{P}_N$ be the projection in $L^2(\omega;\R^2)$ onto the finite-dimensional subspace spanned by $\{\psi_j\}_{1\leq j\leq N}$ with $N\in\N^+$. We also set
	\begin{equation*}
	W_N(t)=\mathsf{P}_NW(t)=\sum_{1\leq j\leq N}\beta_j(t)\psi_j.
	\end{equation*}
    Note that $W_N\in C([0,\infty);\mathsf P_NL^2(\omega;\R^2))$ a.s., where
	\begin{equation*}
	\mathsf{P}_NL^2(\omega;\R^2):=\operatorname{span}\{\psi_j:1\leq j\leq N\}.
	\end{equation*}
    Identifying $\mathsf{P}_NL^2(\omega;\R^2)$ with $\R^N$, we may view $W_N$ as an element of $C([0,\infty);\R^N)$.

    \vspace{0.3em}

    The following proposition summarizes the basic well-posedness, regularity, and smoothness with respect to data and noise for the coupled system.

    \begin{proposition}\label{prop well-posed}
    For the coupled system given by \eqref{eq NS},\eqref{eq xi},\eqref{eq eta} with $\xi=\eta$, assume that the sequence $\{b_j\}_{j\in\N^+}$ satisfies $\sum_{j\in\N^+}b_j^2<\infty$. Then for any $U_0\in \mathcal X$, there exists a unique solution $U=(u,\eta)\colon[0,\infty)\times\Omega\rightarrow \mathcal X$ of the coupled system \eqref{eq NS},\eqref{eq eta}, which is $\mathcal{F}_t$-adapted and satisfies \eqref{eq NS3}. Moreover, for any $t\geq0$ and almost every realization of the noise $W(\cdot,\omega)$, the map $U_0\mapsto U(t,U_0)$ is Fr\'echet differentiable on $\mathcal X$. Meanwhile, for any $U_0\in \mathcal X$, $t\geq0$ and $N\in\N^+$, the map $W_N\mapsto U(t,U_0,W)$ is Fr\'echet differentiable from $C([0,t];\R^N)$ to $\mathcal X$.
    \end{proposition}

    Throughout this paper, we use the notations $U(t)=U(t,U_0)=U_t=U(t,U_0,W)=(u_t,\overline{\eta}_t)$ to denote the unique solution of equation \eqref{eq NS2}, depending on the context. The meaning will be clear from the circumstances and should not cause any confusion.

    \vspace{0.3em}

    Using Proposition \ref{prop well-posed}, the coupled system \eqref{eq NS},\eqref{eq eta} defines a Feller family of Markov processes in $\mathcal X$. The transition function is given by
	\begin{equation*}
	P_t(U_0,A)=\Pb(U(t,U_0)\in A)\quad\text{for }A\in\mathcal B(\mathcal X),\;U_0\in \mathcal X,\;t\geq0.
	\end{equation*}
    The corresponding Markov semigroup $P_t\colon B_b(\mathcal X)\rightarrow B_b(\mathcal X)$ and its dual $P^*_t\colon\mathcal{P}(\mathcal X)\rightarrow\mathcal{P}(\mathcal X)$ are defined by
	\begin{equation*}
    P_t\varphi(U_0)=\int_{\mathcal X}\varphi(\hat{U})P_t(U_0,d\hat{U}),\quad P_t^*\mu(\Gamma)=\int_{\mathcal X}P_t(U_0,\Gamma)\mu(dU_0)
	\end{equation*}
    for $\varphi\in B_b(\mathcal X)$, $\mu\in\mathcal{P}(\mathcal X)$, $U_0\in \mathcal X$ and $\Gamma\in\mathcal B(\mathcal X)$. Recall that a probability measure $\mu\in\mathcal{P}(\mathcal X)$ is called \textit{invariant} for $\{P_t^*\}_{t\geq0}$ if $P_t^*\mu=\mu$ for any $t\geq0$. Our goal is to investigate exponential mixing for the Markov process $\{U(t,U_0)\}_{t\geq0,\;U_0\in \mathcal X}$ on $\mathcal X$, i.e. the Main Theorem.

\subsection{Malliavin calculus and the coupled system}\label{Sec 2.2}

To establish the asymptotic strong Feller property, we follow the approach developed in \cite{HM-06,HM-11b,FGRT-15}, which relies on the Malliavin calculus techniques. To this end, we introduce in this subsection the associated definitions and summarize some preliminaries for the coupled system.

\vspace{0.6em}

In the following notations, let $U_t=U(t,U_0,W)=(u_t,\eta_t)$ be the solution of \eqref{eq NS},\eqref{eq eta} with initial value $U_0=(u_0,\eta_0)$ and noise input $W$. We introduce a separable Hilbert space
	\begin{equation*}
	H=L^2(\R^+;L^2(\omega;\R^2)).
	\end{equation*}
The cylindrical Wiener process $W$ thus gives rise to an isonormal Gaussian process on $H$.

\vspace{0.3em}

The Malliavin derivative $\mathcal{D}$ associated with the coupled system \eqref{eq NS},\eqref{eq eta} is therefore defined by, for each $h\in L^2(\Omega;H)$, 
	\begin{equation*}
	\mathcal{D}\colon\mathbb D^{1,2}(\mathcal X)\rightarrow L^2(\Omega;H\otimes \mathcal X),\quad\langle\mathcal{D}U_t,h\rangle_{H}=\lim\limits_{\varepsilon\rightarrow0}\frac{U(t,U_0,W+\varepsilon\int_0^\cdot h(r)dr)-U(t,U_0,W)}{\varepsilon},
	\end{equation*}
where $\mathbb D^{1,2}(\mathcal X)$ denotes the space of $\mathcal X$-valued random variables that are Malliavin differentiable with both $F$ and $\mathcal DF$ square integrable, and the limit is understood in $L^2(\Omega;\mathcal X)$.

\vspace{0.6em}

For any $V_s=(v_s,p_s)\in \mathcal X$ and $s\geq0$, let $\mathcal J_{s,t}V_s=V_t=(v_t,p_t)$ be the unique solution of the linearized equation, which reads
	\begin{equation}\label{eq J}
	\begin{cases}
	\partial_tv-\nu\Delta v+(u\cdot\nabla)v+(v\cdot\nabla)u+\nabla\pi=\overline{p}(t,x),\quad x\in D,\;t>s,\\
	\dvg v=0,\\
	v\cdot n=0,\quad\curl v=0,\quad x\in\partial D,\\
	\partial_tp-\Delta_\omega p=0,\quad x\in\omega,\\
	p|_{\partial\omega}=0,\\
	v(s,\cdot)=v_s(\cdot),\quad p(s,\cdot)=p_s(\cdot).
	\end{cases}
	\end{equation}
Here $\pi$ is the pressure and $\overline{p}$ denotes the zero extension of $p$ from $\omega$ to $D$.

It is useful to introduce the diagonal operator
	\begin{equation*}
	\Lambda\colon L^2(\omega;\R^2)\rightarrow L^2(\omega;\R^2),\quad\Lambda h=\sum_{j\in\N^+}b_j\langle h,\psi_j\rangle_{L^2(\omega)}\psi_j.
	\end{equation*}
Note that the assumption $\sum_{j\in\N^+}b_j^2<\infty$ implies that $\Lambda$ is Hilbert--Schmidt on $L^2(\omega;\R^2)$. Then for any $h\in H$, it follows that
	\begin{equation}\label{eq D def}
	\langle\mathcal{D}U_t,h\rangle_{H}=\int_0^t\mathcal J_{r,t}\mathcal Bh(r)dr,
	\end{equation}
where the operator $\mathcal B$ is defined as
	\begin{equation*}
	\mathcal B\colon L^2(\omega;\R^2)\rightarrow \mathcal X,\quad\mathcal Bh=\left(0,\Lambda h\right).
	\end{equation*}

Inspired by \eqref{eq D def}, we define the random operator $\mathcal A_{s,t}$ for $0\leq s<t<\infty$ by
	\begin{equation*}
	\mathcal A_{s,t}\colon H_{s,t}\rightarrow \mathcal X,\quad\mathcal A_{s,t}h=\int_s^t\mathcal J_{r,t}\mathcal Bh(r)dr,
	\end{equation*}
where $H_{s,t}:=L^2(s,t;L^2(\omega;\R^2))$, endowed with the inner product inherited from $H$.

For any $h\in H_{s,t}$, the trajectory $\mathcal A_{s,r}h=(v_r,p_r)$ satisfies
	\begin{equation*}
	\begin{cases}
	\partial_rv-\nu\Delta v+(v\cdot\nabla)u+(u\cdot\nabla)v+\nabla\pi=\overline{p}(r,x),\quad x\in D,\;r\in(s,t),\\
	\dvg v=0,\\
	v\cdot n=0,\quad\curl v=0,\quad x\in\partial D,\\
	\partial_rp-\Delta_{\omega}p=\Lambda h(r),\quad x\in\omega,\\
	p|_{\partial\omega}=0,\\
	v(s,\cdot)=0,\quad p(s,\cdot)=0.
	\end{cases}
	\end{equation*}
The adjoint operator $\mathcal A^*_{s,t}$ of $\mathcal A_{s,t}$ is given by
	\begin{equation*}
	\mathcal A^*_{s,t}\colon \mathcal X\rightarrow H_{s,t},\quad(\mathcal A^*_{s,t}V)(r)=\mathcal B^*\mathcal J^*_{r,t}V\quad\text{for }V\in \mathcal X,\;r\in[s,t].
	\end{equation*}
Here $\mathcal B^*$ denotes the adjoint operator of $\mathcal B$, given by
	\begin{equation*}
	\mathcal B^*\colon \mathcal X\rightarrow L^2(\omega;\R^2),\quad\mathcal B^*(w,q)=\Lambda q=\sum_{j\in\N^+}b_j\langle q,\psi_j\rangle\psi_j\quad\text{for }(w,q)\in \mathcal X.
	\end{equation*}
Meanwhile, $\mathcal J_{s,t}^*$ stands for the Hilbert adjoint of $\mathcal J_{s,t}$ with respect to the $\mathcal X$-inner product.

Similarly, for any $N\in\N^+$ and $0\leq s<t$, define the random operator $\mathcal A_{s,t,N}=\mathcal A_{s,t}\mathsf{P}_N$ as
	\begin{equation*}
	\mathcal A_{s,t,N}\colon H_{s,t}\rightarrow \mathcal X,\quad\mathcal A_{s,t,N}h=\int_s^t\mathcal J_{r,t}\mathcal B\mathsf{P}_Nh(r)dr,
	\end{equation*}
where $\mathsf{P}_N$ is the projection on $L^2(\omega;\R^2)$. The associated adjoint operator $\mathcal A^*_{s,t,N}$ is given by
	\begin{equation*}
	\mathcal A^*_{s,t,N}\colon \mathcal X\rightarrow H_{s,t},\quad(\mathcal A^*_{s,t,N}V)(r)=\mathsf{P}_N\mathcal B^*\mathcal J^*_{r,t}V\quad\text{for }V\in \mathcal X,\;r\in[s,t].
	\end{equation*}

\vspace{0.3em}

Finally, we define the truncated Malliavin matrix $\mathcal M_{s,t,N}$ by
	\begin{equation*}
	\mathcal M_{s,t,N}\colon \mathcal X\rightarrow \mathcal X,\quad\mathcal M_{s,t,N}=\mathcal A_{s,t,N}\mathcal A^*_{s,t,N}.
	\end{equation*}
In particular, one has
	\begin{equation*}
	\mathcal M_{s,t,N}=\int_s^t\mathcal J_{r,t}\mathcal B\mathsf{P}_N\mathcal B^*\mathcal J^*_{r,t}dr,
	\end{equation*}
and
	\begin{equation*}
	\langle\mathcal M_{s,t,N}V,V\rangle=\int_s^t\|(\mathcal A^*_{s,t,N}V)(r)\|^2dr=\sum_{1\leq j\leq N}b_j^2\int_s^t\langle q_r,\psi_j\rangle_{L^2(\omega)}^2dr\quad\forall\,V\in \mathcal X,
	\end{equation*}
where $(w_r,q_r)=\mathcal J^*_{r,t}V$.

    \section{Stabilization analysis for the coupled system via localized force}\label{Sec 3}

    In this section, we establish the quantitative stabilization result for the coupled linearized system stated in Theorem~\ref{thm control4}. It provides a finite-dimensional control acting only on the forced component, together with the estimates required for the Malliavin calculus argument proving the asymptotic strong Feller property in Section~\ref{Sec 4}.

    \vspace{0.6em}

    Let us consider the following controlled coupled linearized system:
	\begin{equation}\label{eq V}
	\begin{cases}
	\partial_rv-\nu\Delta v+(g\cdot\nabla)v+(v\cdot\nabla)g+\nabla\pi=\overline{p}(r,x),\quad x\in D,\;r\in(0,1),\\
	\dvg v=0,\\
	v\cdot n=0,\quad\curl v=0,\quad x\in\partial D,\\
	\partial_rp-\Delta_\omega p=\mathsf P_N\Lambda h(r,x),\quad x\in\omega,\\
	p|_{\partial\omega}=0,\\
	v(0,\cdot)=v_0(\cdot),\quad p(0,\cdot)=p_0(\cdot).
	\end{cases}
	\end{equation}
    Here $V_0=(v_0,p_0)\in \mathcal X$, $N\in\N^+$ and $h\in L^2(0,1;L^2(\omega;\R^2))$. Let $Q=D\times(0,1)$. The reference velocity field $g$ is assumed to satisfy
	\begin{equation}\label{eq g}
    \dvg g=0\text{ in } D,\quad g\cdot n=0\text{ on }\partial D,\quad g\in L^\infty(Q;\R^2),\quad\partial_rg\in L^2(0,1;L^4(D;\R^2)).    
	\end{equation}
    For $0\leq s<t\leq1$, let $\mathcal J_{s,t,g}\colon \mathcal X\rightarrow \mathcal X$ denote the solution operator of \eqref{eq V}, and set $\mathcal J_g:=\mathcal J_{0,1,g}$. Define the auxiliary low-frequency control map $\mathcal R_{N,g}\in\mathcal L(H_{0,1};\mathcal X)$ by
	\begin{equation*}
	\mathcal R_{N,g}h:=\int_0^1\mathcal J_{r,1,g}(0,\mathsf P_Nh(r))dr,\quad h\in L^2(0,1;L^2(\omega;\R^2)).
	\end{equation*}
    This operator corresponds to the auxiliary system obtained from \eqref{eq V} by replacing $\mathsf P_N\Lambda h$ with $\mathsf P_Nh$. The  control map $\mathcal A_{N,g}\in\mathcal L(H_{0,1};\mathcal X)$ and its Gramian $\mathcal G_{N,g}\in\mathcal L(\mathcal X)$ are given by
	\begin{equation*}
	\mathcal A_{N,g}h:=\mathcal R_{N,g}(\Lambda h)=\int_0^1\mathcal J_{r,1,g}(0,\mathsf P_N\Lambda h(r))dr,\quad\mathcal G_{N,g}:=\mathcal A_{N,g}\mathcal A_{N,g}^*.
	\end{equation*}

    In what follows, let $\mathcal C_g>0$ be a constant, which may vary from line to line, of the form
	\begin{equation*}
	\mathcal C_g=C\exp\left(\exp\left(CG_1^{20/3}+\exp(CG_0^2)\right)\right),\quad G_0=\|g\|_{L^\infty(Q)},\quad G_1=\|\partial_rg\|_{L^2(0,1;L^4(D))},
	\end{equation*}
    where $C>0$ depends only on $D$, $\omega$ and $\nu$.

    The main stabilization result is collected in the following theorem.

    \begin{theorem}\label{thm control4}
    For the equation \eqref{eq V}, there exists a constant $C>0$ such that the following holds. For any $\beta>0$, $N\in\N^+$, $V_0\in \mathcal X$ and $g$ satisfying \eqref{eq g}, the solution $V$ of \eqref{eq V} with the control $h\in H_{0,1}$ given by
	\begin{equation*}
	h(r)=-\left(\mathcal A_{N,g}^*(\mathcal G_{N,g}+\beta)^{-1}\mathcal J_gV_0\right)(r),\quad r\in(0,1),
	\end{equation*}
    satisfies that
	\begin{align}
	\|V(1)\|_{\mathcal X}&\leq\mathcal C_g\left(\widehat b_N\lambda_N\beta^{1/2}+\lambda_N^{-1/4}\right)\|V_0\|_{\mathcal X},\label{eq control4a}\\
	\|h\|_{L^2(0,1;L^2(\omega))}&\leq\mathcal C_g\left(\widehat b_N\lambda_N+\beta^{-1/2}\lambda_N^{-1/4}\right)\|V_0\|_{\mathcal X},\label{eq control4b}
	\end{align}
    where $\widehat b_N:=\max\{|b_j|^{-1}:1\leq j\leq N\}$.
    \end{theorem}

    The proof of Theorem \ref{thm control4} is based on the following  weak observability inequality.

    \begin{proposition}\label{prop obs2}
    For any $g$ satisfying \eqref{eq g}, there exists a constant $\mathcal C_g>0$ such that for any $N\in\N^+$ and $Z_1=(w_1,q_1)\in \mathcal X$, the solution $(w(r),q(r))=\mathcal J_{r,1,g}^*Z_1$
    satisfies that
	\begin{equation}\label{eq obs2}
	\|(w(0),q(0))\|_{\mathcal X}\leq\mathcal C_g\left(\lambda_N\|\mathsf P_Nq\|_{L^2(0,1;L^2(\omega))}+\lambda_N^{-1/4}\|Z_1\|_{\mathcal X}\right).
	\end{equation}
    \end{proposition}

    \vspace{0.6em}

    We now explain the proof strategy for Proposition~\ref{prop obs2} and Theorem~\ref{thm control4}. The construction combines the regular null controllability of the Navier--Stokes equation \cite{Guerrero-06} with the frequency-analysis strategy applied in \cite{LWXZZ-24,Xiang-24}. Specifically, our approach relies on two ingredients:

    \begin{itemize}
    \item[(\runum{1})] A low-frequency heat control for the coupled system. Starting from a linearized Navier--Stokes null control with additional regularity, we truncate it to the forced modes and lift it through the heat equation, thereby obtaining asymptotic null controllability;

    \item[(\runum{2})] A twofold application of the duality between controllability and observability. We first convert asymptotic null controllability into a weak observability, and then use it to construct explicit regularized controls with the structure required in the Malliavin argument.
    \end{itemize}

    \vspace{0.3em}

    The two regularizations above play different roles: the first one is spectral and makes the fluid control compatible with the forced heat modes, whereas the second one is variational and regularizes the inverse of the controllability Gramian. The proof is structured into five steps, as sketched below. See Sections~\ref{Sec 3.1}--\ref{Sec 3.3} for details.

    \vspace{0.3em}

    \noindent {\it Step 1: Quantitative regular Navier--Stokes control.}
    We first construct a localized control $\zeta$ for the linearized velocity $z$ governed by
    \begin{equation}\label{eq out-z}
    \partial_rz-\nu\Delta z+(g\cdot\nabla)z+(z\cdot\nabla)g+\nabla\pi=\zeta,\quad\dvg z=0,\quad z(0)=z_0.
    \end{equation}
    The control drives $z$ to zero and satisfies
    \begin{equation*}    z(1)=0,\quad\|\zeta\|_{H^1(0,1;L^2(\omega))}+\|\zeta\|_{L^2(0,1;H^2(\omega))}+\|\zeta\|_{L^\infty(0,1;H^1(\omega))}\leq\mathcal C_g\|z_0\|.
    \end{equation*}
    The additional regularity permits the low-frequency truncation in the next step, while the explicit dependence of $\mathcal C_g$ on $g$ is needed for random reference trajectories. The construction is based on the Carleman estimate of \cite{Guerrero-06} and a weighted Fenchel--Rockafellar argument.

    \vspace{0.6em}

    \noindent {\it Step 2: Low-frequency regularization.} Set $\zeta_N=\mathsf P_N\zeta$, and let $z_N$ solve \eqref{eq out-z} with $\zeta$ replaced by $\zeta_N$. Using the spatial regularity of $\zeta$ controls, we derive
    \begin{equation*}
    \|z_N(1)\|_{\mathcal H^1}\leq\lambda_{N+1}^{-1/4}\mathcal C_g\|z_0\|_{\mathcal H^1}, \quad \|\zeta_N\|_{H^1(0,1;L^2(\omega))}\leq \mathcal C_g\|z_0\|_{\mathcal H^1}.
    \end{equation*}
    Thus $\zeta_N$ is a finite-dimensional approximate null control, whose terminal error vanishes as $N\to\infty$. Its finite-dimensional structure allows it to be lifted through the first $N$ forced heat modes in the coupled system.

    \vspace{0.6em}

    \noindent {\it Step 3: Lifting to the coupled system.} We modify $\zeta_N$ only near the two endpoints to obtain a low-frequency heat path satisfying
    \begin{equation*}
    \widehat\zeta_{N,\delta}
    \in H^1(0,1;\mathsf P_NL^2(\omega;\R^2)),\quad\widehat\zeta_{N,\delta}=\zeta_N\text{ on }[\delta/2,1-\delta/2],\quad\widehat\zeta_{N,\delta}(0)=\mathsf P_Np_0,\quad\widehat\zeta_{N,\delta}(1)=0.
    \end{equation*}
    
    We then realize this path through the heat equation by setting $h_{N,\delta}=(\partial_r-\Delta_\omega)\widehat\zeta_{N,\delta}$. For the corresponding coupled solution
    $V_{N,\delta}=(v_{N,\delta},p_{N,\delta})$, one has
    \begin{equation*}
    \|V_{N,\delta}(1)\|_{\mathcal X}\leq\mathcal C_g\left(\lambda_{N+1}^{-1/4}+\delta^{1/2}\right)\|V_0\|_{\mathcal X},\quad\|h_{N,\delta}\|_{H_{0,1}}\leq\mathcal C_g\left(\lambda_N+\delta^{-1/2}\right)\|V_0\|_{\mathcal X}.
    \end{equation*}
    
    Therefore, this lifting transfers the localized fluid control to the dynamically forced component, while the remaining high modes are handled by heat dissipation.

    \vspace{0.6em}

    \noindent {\it Step 4: Weak observability for the coupled system.}
    Taking $\delta$ of order $\lambda_N^{-2}$ and using the classical duality between controllability and observability, we obtain, for the adjoint solution $Z=(w,q)$ with $Z(1)=Z_1$,
    \begin{equation*}
    \|(w(0),q(0))\|_{\mathcal X}\leq\mathcal C_g\left(\lambda_N\|\mathsf P_Nq\|_{L^2(0,1;L^2(\omega))}+\lambda_N^{-1/4}\|Z_1\|_{\mathcal X}\right),
    \end{equation*}
    which implies Proposition \ref{prop obs2}. Note that the terminal remainder makes this a weak observability inequality, corresponding to the weak forms of null controllability studied in \cite{AM-22,TWX-20}.

    \vspace{0.6em}
    \noindent {\it Step 5: Regularized stabilization.}
    Following the regularized control construction applied in \cite{LXZ-26}, we apply the Fenchel--Rockafellar duality to the quadratic minimization problem
    \begin{equation*}
    \inf_{h\in H_{0,1}}J(h)=\inf_{h\in H_{0,1}}
    \left(\frac12\|h\|_{H_{0,1}}^2+\frac1{2\beta}\|\mathcal A_{N,g}h+\mathcal J_gV_0\|_{\mathcal X}^2\right).
    \end{equation*}
    In particular, the resulting control has the explicit form $h=-\mathcal A_{N,g}^*(\mathcal G_{N,g}+\beta)^{-1}\mathcal J_gV_0$.

    Using the weak observability inequality from Step 4, we conclude that
    \begin{align*}
    \|V(1)\|_{\mathcal X}\leq\mathcal C_g\left(\widehat b_N\lambda_N\beta^{1/2}+\lambda_N^{-1/4}\right)\|V_0\|_{\mathcal X},\quad\|h\|_{H_{0,1}}\leq\mathcal C_g\left(\widehat b_N\lambda_N+\beta^{-1/2}\lambda_N^{-1/4}\right)\|V_0\|_{\mathcal X}.
    \end{align*}
    This completes the proof of Theorem \ref{thm control1}, and the resolvent formulation and the quantitative estimates are particularly suited to the Malliavin argument.

    \subsection{Regular control for the linearized Navier--Stokes}\label{Sec 3.1}

    We begin the stabilization analysis by constructing a localized null control for the linearized Navier--Stokes equation with the additional regularity required by the heat dynamics. A Carleman estimate and the resulting weighted observability inequality, combined with the Fenchel--Rockafellar duality, yield the quantitative control stated in Theorem~\ref{thm control1}.

    \vspace{0.6em}

    We consider the controlled linearized Navier--Stokes equation
	\begin{equation}\label{eq z}
	\begin{cases}
	\partial_rz-\nu\Delta z+(g\cdot\nabla)z+(z\cdot\nabla)g+\nabla\pi=\overline{\zeta}(r,x),\quad x\in D,\;r\in(0,1),\\
	\dvg z=0,\\
	z\cdot n=0,\quad\curl z=0,\quad x\in\partial D,\\
	z(0,\cdot)=z_0(\cdot)\in \mathcal H,
	\end{cases}
	\end{equation}
    where $\zeta=\zeta(r,x)$ is supported in $\omega$ and $\overline{\zeta}$ denotes its zero extension to $D$. The corresponding adjoint equation is
	\begin{equation}\label{eq dual}
	\begin{cases}
	-\partial_r\varphi-\nu\Delta\varphi-(g\cdot\nabla)\varphi+(\nabla g)^{\rm T}\varphi+\nabla\rho=0,\quad x\in D,\;r\in(0,1),\\
	\dvg\varphi=0,\\
	\varphi\cdot n=0,\quad\curl\varphi=0,\quad x\in\partial D,\\
	\varphi(1,\cdot)= w(\cdot)\in \mathcal H.
	\end{cases}
	\end{equation}
    To introduce the Carleman weights, let $\omega_1\Subset\omega$ be a smooth open set and $\vartheta\in C^2(\overline D)$ satisfy
	\begin{equation*}
	\vartheta>0\quad{\rm in}\ D,\quad\vartheta=0\quad{\rm on}\ \partial D,\quad|\nabla\vartheta|>0\quad{\rm on}\ \overline D\setminus\omega_1.
	\end{equation*}
    Set $K=\|\vartheta\|_{L^\infty(D)}$. For $\lambda>0$, define
	\begin{equation*}
	\alpha(x,r):=\frac{e^{2\lambda K}-e^{\lambda\vartheta(x)}}{r^4(1-r)^4},\quad\xi(x,r):=\frac{e^{\lambda\vartheta(x)}}{r^4(1-r)^4},\quad x\in D,\;r\in(0,1).
	\end{equation*}
    We also set
	\begin{equation*}
	\alpha_*(r)=\min_{\overline D}\alpha(\cdot,r),\quad\alpha^*(r)=\max_{\overline D}\alpha(\cdot,r),\quad\xi^*(r)=\max_{\overline D}\xi(\cdot,r),\quad r\in(0,1).
	\end{equation*}
    For $s>0$ and $\lambda>0$, define
	\begin{align*}
	\mathcal I(s,\lambda;\varphi):=&s^3\lambda^4\int_0^1\int_De^{-2s\alpha}\xi^3|\varphi|^2dxdr+s\lambda^2\int_0^1\int_De^{-2s\alpha}\xi|\nabla\varphi|^2dxdr\\
	&\quad+s^{-1}\int_0^1\int_De^{-2s\alpha}\xi^{-1}\left(|\partial_r\varphi|^2+|\Delta\varphi|^2\right)dxdr.
	\end{align*}

    \begin{proposition}\label{prop Carleman}  For the equation \eqref{eq dual}, there exist constants $\widehat s,\widehat\lambda,C>0$, depending only on $D$, $\omega$ and $\nu$, such that the following holds. For any $g$ satisfying \eqref{eq g} and $ w\in \mathcal H$, the solution $\varphi$ of \eqref{eq dual} satisfies that
	\begin{align}\label{eq Carleman}
	\mathcal I(s,\lambda;\varphi)\leq Cs^{15/2}\lambda^8\int_0^1\int_{\omega}e^{-4s\alpha_*+2s\alpha^*}(\xi^*)^{15/2}|\varphi|^2dxdr
	\end{align}
    for any $\lambda$ satisfying
	\begin{equation}\label{eq lambda}
	\lambda\geq\widehat\lambda\exp\left(\widehat\lambda\|g\|_{L^\infty(Q)}^2\right)\left(1+\|g\|_{L^\infty(Q)}^{10/3}+\|\partial_rg\|_{L^2(0,1;L^4(D))}^{10/3}\right),
	\end{equation}
    and any $s\geq\widehat se^{8\lambda K}$.
    \end{proposition}

    \begin{proof}
    Let $\tau$ be a smooth unit tangent vector on $\partial D$, and set $D\varphi=\nabla\varphi+(\nabla\varphi)^{\rm T}$. In dimension two, there exists $a_D\in C^\infty(\partial D)$, depending only on $D$, such that
	\begin{equation*}
	(D\varphi\,n)\cdot\tau=\curl\varphi+a_D\varphi\cdot\tau\quad{\rm on}\ \partial D.
	\end{equation*}
    Hence the boundary conditions in \eqref{eq dual} are equivalent to
    \begin{equation*}   
    \varphi\cdot n=0,\qquad(D\varphi\,n)\cdot\tau-a_D\varphi\cdot\tau=0,         
    \end{equation*}
    which are precisely the Navier boundary conditions considered in \cite[Proposition 2.3]{Guerrero-06}.

    Moreover, using $\dvg\varphi=0$, one has
	\begin{equation*}
	\nabla\cdot D\varphi=\Delta\varphi,\quad-(g\cdot\nabla)\varphi+(\nabla g)^{\rm T}\varphi=-D\varphi\,g+\nabla(g\cdot\varphi).
	\end{equation*}
    Hence, after replacing $\rho$ by $\widetilde\rho=\rho+g\cdot\varphi$, equation \eqref{eq dual} therefore coincides with the homogeneous adjoint system in \cite[Proposition 2.3]{Guerrero-06}, with $b=g$.

    The assumptions in \eqref{eq g} provide the required bounds on $b$ and $\partial_rb$. Applying \cite[Proposition 2.3]{Guerrero-06} gives \eqref{eq Carleman} under \eqref{eq lambda} and $s\geq\widehat se^{8\lambda K}$. The fixed boundary coefficient $a_D$ and the viscosity $\nu$ only affect the constants. This completes the proof.
    \end{proof}

    In what follows, let us fix
	\begin{equation*}
	\lambda_g:=\widehat\lambda\exp\left(\widehat\lambda G_0^2\right)\left(1+G_0^{10/3}+G_1^{10/3}\right),\quad s_g:=\widehat se^{8\lambda_gK},
	\end{equation*}
    and the weights $\alpha,\xi,\alpha_*,\alpha^*,\xi^*$ are evaluated at $\lambda=\lambda_g$ whenever no confusion is possible. Define
	\begin{equation}\label{eq weight}
	\Theta_g(r):=s_g^{15/2}\lambda_g^8e^{-4s_g\alpha_*(r)+2s_g\alpha^*(r)}(\xi^*(r))^{15/2}.
	\end{equation}

    \begin{proposition}\label{prop obs1} 
    For any $g$ satisfying \eqref{eq g}, there exists a constant $\mathcal C_g>0$ such that for any $ w\in \mathcal H$, the solution $\varphi$ of \eqref{eq dual} satisfies that
	\begin{equation}\label{eq obs1}
	\|\varphi(0)\|^2\leq\mathcal C_g\int_0^1\int_\omega\Theta_g(r)|\varphi|^2dxdr.
	\end{equation}
    \end{proposition}

    \begin{proof} The proof combines elementary energy estimates with  Proposition~\ref{prop Carleman}. Taking the $L^2(D)$ inner product of \eqref{eq dual} with $\varphi$ and using the identity in the preceding proof together with \eqref{eq g}, we obtain
	\begin{equation*}
	-\frac12\frac{d}{dr}\|\varphi(r)\|^2+\nu\|\curl\varphi(r)\|^2=\int_DD\varphi(r)g(r)\cdot\varphi(r)dx\leq\frac\nu2\|\curl\varphi(r)\|^2+CG_0^2\|\varphi(r)\|^2.
	\end{equation*}
    Using Gronwall's inequality, we then have
	\begin{equation*}
	\|\varphi(0)\|^2\leq C\exp(CG_0^2)\int_{1/4}^{3/4}\|\varphi(r)\|^2dr.
	\end{equation*}
    Meanwhile, on $\overline D\times[1/4,3/4]$, by the definitions of $\alpha$, $\xi$ and $s_g$, it follows that
	\begin{equation*}
	\sup_{\overline D\times[1/4,3/4]}\left(s_g^3\lambda_g^4e^{-2s_g\alpha}\xi^3\right)^{-1}\leq C\exp\left(C\exp(C\lambda_g)\right).
	\end{equation*}
    Moreover, $\lambda_g\leq C\left(\exp(CG_0^2)+G_1^{20/3}\right)$. Hence, after enlarging $\mathcal C_g$, one has
	\begin{equation*}
	\int_{1/4}^{3/4}\|\varphi(r)\|^2dr\leq\mathcal C_g\mathcal I(s_g,\lambda_g;\varphi).
	\end{equation*}

    Consequently, combining the preceding estimates with Proposition~\ref{prop Carleman} at $(s,\lambda)=(s_g,\lambda_g)$, we derive the desired inequality  \eqref{eq obs1}  and completes the proof of Proposition \ref{prop obs1}.
    \end{proof}

    \begin{theorem}\label{thm control1}
    For any $g$ satisfying \eqref{eq g}, there exists a constant $\mathcal C_g>0$ such that for any $z_0\in \mathcal H$, there exists a control
	\begin{equation*}
	\zeta\in H_0^1(0,1;L^2(\omega;\R^2))\cap L^2(0,1;H^2(\omega;\R^2))\cap C([0,1];H^1(\omega;\R^2))
	\end{equation*}
    such that the solution $z$ of \eqref{eq z} satisfies that $z(1)=0$ and
	\begin{align}\label{eq control1b}
	\|\zeta\|_{H^1(0,1;L^2(\omega))}+\|\zeta\|_{L^2(0,1;H^2(\omega))}+\|\zeta\|_{L^\infty(0,1;H^1(\omega))}&\leq\mathcal C_g\|z_0\|.
	\end{align}
    \end{theorem}

    \begin{proof} The proof is divided into four steps.
    
    \vspace{0.3em}
    \noindent {\it Step 1. Forward and adjoint solution operators.}
    For $u,v\in \mathcal H^1$, define
	\begin{equation*}
	a_g(r;u,v):=\nu\int_D\curl u\curl vdx+\int_D(g(r)\cdot\nabla)u\cdot vdx-\int_D(u\cdot\nabla)v\cdot g(r)dx.
	\end{equation*}
    Applying standard estimates, we have
	\begin{equation*}
	a_g(r;u,u)\geq\frac\nu2\|\curl u\|^2-CG_0^2\|u\|^2.
	\end{equation*}
    
    Therefore, for any $f\in L^2(Q;\R^2)$, standard parabolic theory gives a unique weak solution of the corresponding forward equation in
	\begin{equation*}
	L^2(0,1;\mathcal H^1)\cap H^1(0,1;(\mathcal H^1)')\hookrightarrow C([0,1];\mathcal H),
	\end{equation*}
    together with the energy estimate
	\begin{equation*}
	\|z\|_{C([0,1];\mathcal H)}^2+\|z\|_{L^2(0,1;\mathcal H^1)}^2\leq C\exp(CG_0^2)\left(\|z_0\|^2+\|f\|_{L^2(Q)}^2\right).
	\end{equation*}

    We introduce the weighted control space
	\begin{equation*}
	\mathcal U_g:=\left\{\zeta:\Theta_g^{-1/2}\zeta\in L^2(\omega\times(0,1);\R^2)\right\},\quad\|\zeta\|_{\mathcal U_g}^2:=\int_0^1\int_\omega\Theta_g^{-1}|\zeta|^2dxdr.
	\end{equation*}    
    By Lemma~\ref{lemma theta}, $\Theta_g$ is positive and bounded on $(0,1)$, $\Theta_g(0)=\Theta_g(1)=0$. Thus $\mathcal U_g$ is a Hilbert space and 
    \begin{equation*}
    \|\zeta\|_{L^2(\omega\times(0,1))}
    \leq \|\Theta_g\|_{L^\infty(0,1)}^{1/2}
    \|\zeta\|_{\mathcal U_g},
    \end{equation*}
    In particular, $\mathcal U_g$ is continuously embedded into the usual $L^2$ control space.
    
    Let $\bar{z}$ be the solution of \eqref{eq z} with zero control, and let $\mathcal A_g\zeta$ be the value at $r=1$ of the solution with zero initial value and control $\zeta$. The preceding energy estimate shows that $\mathcal A_g\in\mathcal L(\mathcal U_g;\mathcal H)$, and the solution corresponding to $\zeta$ satisfies $z_\zeta(1)=\bar{z}(1)+\mathcal A_g\zeta$.

    For $ w\in \mathcal H$, let $\varphi$ be the solution of \eqref{eq dual}. The forward-adjoint duality gives
	\begin{equation*}
	\langle\mathcal A_g\zeta, w\rangle=\int_0^1\int_\omega\zeta\cdot\varphi dxdr=\langle\zeta,\Theta_g\varphi\rangle_{\mathcal U_g}.
	\end{equation*}
    Thus $\mathcal A_g^* w=\Theta_g\varphi$ on $\omega\times(0,1)$.

    \vspace{0.6em}

    \noindent {\it Step 2. Fenchel--Rockafellar duality.}
    For $\varepsilon>0$, define two convex and continuous functions by
	\begin{align*}
	&F\colon\mathcal U_g\rightarrow\R^+,\quad F(\zeta)=\frac12\|\zeta\|_{\mathcal U_g}^2,\quad G\colon \mathcal H\rightarrow\R^+,\quad G(w)=\frac1{2\varepsilon}\|w+\bar{z}(1)\|^2.
	\end{align*}

    We consider the minimization problem
	\begin{equation*}
	\inf_{\zeta\in\mathcal U_g}J_\varepsilon(\zeta)=\inf_{\zeta\in\mathcal U_g}\left(F(\zeta)+G(\mathcal A_g\zeta)\right).
	\end{equation*}
    The Fenchel conjugates of $F$ and $G$ are given by
	\begin{align*}
	F^*(\zeta)=\sup_{\widetilde\zeta\in\mathcal U_g}\left(\langle\zeta,\widetilde\zeta\rangle_{\mathcal U_g}-F(\widetilde\zeta)\right)=\frac12\|\zeta\|_{\mathcal U_g}^2,\quad G^*( w)=\sup_{\widetilde w\in \mathcal H}\left(\langle w,\widetilde w\rangle-G(\widetilde w)\right)=\frac\varepsilon2\| w\|^2-\langle w,\bar{z}(1)\rangle.
	\end{align*}
    After changing the sign of the dual variable, the corresponding dual problem is
	\begin{equation*}
	\sup_{ w\in \mathcal H}J_\varepsilon^*( w)=\sup_{ w\in \mathcal H}\left(-\frac12\|\mathcal A_g^* w\|_{\mathcal U_g}^2-\frac\varepsilon2\| w\|^2+\langle w,\bar{z}(1)\rangle\right).
	\end{equation*}

    The primal functional is coercive, weakly lower semicontinuous and strictly convex, and hence admits a unique minimizer. The dual functional is weakly upper semicontinuous, strictly concave and satisfies $J_\varepsilon^*(\varphi_1)\to-\infty$ as $\|\varphi_1\|\to\infty$, so it also admits a unique maximizer. Since $G$ is continuous on $\mathcal H$, the Fenchel--Rockafellar theorem \cite{Rockafellar-67,ET-99} yields
    \begin{equation*}
    \inf_{\zeta\in\mathcal U_g}J_\varepsilon(\zeta)
    =
    \sup_{\varphi_1\in \mathcal H}J_\varepsilon^*(\varphi_1).
    \end{equation*} 

    Let $\zeta_\varepsilon$ be the minimizer and $z_\varepsilon$ the corresponding state. We denote the maximizer by $\varphi_\varepsilon(1)$, where $\varphi_\varepsilon$ is the solution of \eqref{eq dual}. The Fenchel dual relations give
	\begin{align*}
	\zeta_\varepsilon=-\mathcal A_g^*\varphi_\varepsilon(1)=-\Theta_g\varphi_\varepsilon,\quad 
	(\mathcal A_g\mathcal A_g^*+\varepsilon)\varphi_\varepsilon(1)=\bar{z}(1).
	\end{align*}

    Consequently,
	\begin{equation*}
	z_\varepsilon(1)=\bar{z}(1)+\mathcal A_g\zeta_\varepsilon=\bar{z}(1)-\mathcal A_g\mathcal A_g^*\varphi_\varepsilon(1)=\varepsilon\varphi_\varepsilon(1).
	\end{equation*}
    Moreover, the duality relation and the forward-adjoint identity imply
	\begin{equation*}
	\|\zeta_\varepsilon\|_{\mathcal U_g}^2+\frac1\varepsilon\|z_\varepsilon(1)\|^2=\langle\varphi_\varepsilon(1),\bar{z}(1)\rangle=\langle\varphi_\varepsilon(0),z_0\rangle.
	\end{equation*}
    By Proposition \ref{prop obs1} and the identity for $\mathcal A_g^*$, one has
	\begin{equation*}
	\|\varphi_\varepsilon(0)\|^2\leq\mathcal C_g\int_0^1\int_\omega\Theta_g|\varphi_\varepsilon|^2dxdr=\mathcal C_g\|\zeta_\varepsilon\|_{\mathcal U_g}^2.
	\end{equation*}

    Combining the last two estimates and applying Young's inequality, we obtain
	\begin{equation}\label{eq HUM2}
	\|\zeta_\varepsilon\|_{\mathcal U_g}^2+\frac1\varepsilon\|z_\varepsilon(1)\|^2\leq\mathcal C_g\|z_0\|^2.
	\end{equation} 
    
    \vspace{0.6em}

    \noindent {\it Step 3. Regularity of the controls.}
    Set $\psi_\varepsilon=\Theta_g\varphi_\varepsilon$. By \eqref{eq theta1}, $\psi_\varepsilon(0)=\psi_\varepsilon(1)=0$, and $\psi_\varepsilon$ satisfies
	\begin{equation*}
	\begin{cases}
	-\partial_r\psi_\varepsilon-\nu\Delta\psi_\varepsilon-(g\cdot\nabla)\psi_\varepsilon+(\nabla g)^{\rm T}\psi_\varepsilon+\nabla\rho_\varepsilon=-\Theta_g'\varphi_\varepsilon,\quad x\in D,\;r\in(0,1),\\
	\dvg\psi_\varepsilon=0,\\
	\psi_\varepsilon\cdot n=0,\quad\curl\psi_\varepsilon=0,\quad x\in\partial D,\\
	\psi_\varepsilon(1)=0.
	\end{cases}
	\end{equation*}

    Lemma \ref{lemma theta} and Proposition \ref{prop Carleman} imply
	\begin{equation*}
	\|\Theta_g'\varphi_\varepsilon\|_{L^2(Q)}^2\leq C\mathcal I(s_g,\lambda_g;\varphi_\varepsilon)\leq C\|\zeta_\varepsilon\|_{\mathcal U_g}^2.
	\end{equation*}
    Applying Lemma \ref{lemma l1} to the equation for $\psi_\varepsilon$, we obtain
	\begin{equation*}
	\|\partial_r\psi_\varepsilon\|_{L^2(Q)}+\|\psi_\varepsilon\|_{L^2(0,1;\mathcal H^2)}+\|\psi_\varepsilon\|_{L^\infty(0,1;\mathcal H^1)}\leq\mathcal C_g\|\zeta_\varepsilon\|_{\mathcal U_g}.
	\end{equation*}
    Since $\zeta_\varepsilon=-\psi_\varepsilon$ on $\omega\times(0,1)$, estimate \eqref{eq HUM2} gives
	\begin{equation}\label{eq HUM3}
	\|\zeta_\varepsilon\|_{H^1(0,1;L^2(\omega))}+\|\zeta_\varepsilon\|_{L^2(0,1;H^2(\omega))}+\|\zeta_\varepsilon\|_{L^\infty(0,1;H^1(\omega))}\leq\mathcal C_g\|z_0\|.
	\end{equation}
    In particular, the endpoint values of $\psi_\varepsilon$ show that $\zeta_\varepsilon\in H_0^1(0,1;L^2(\omega;\R^2))$. 

    \vspace{0.6em}

    \noindent {\it Step 4. Passage to the limit.}
    Let $\varepsilon_n\downarrow0$. By \eqref{eq HUM3}, after taking a subsequence, there exists a function $\zeta$ such that
	\begin{align*}
	\zeta_{\varepsilon_n}&\rightharpoonup\zeta\quad{\rm in}\ H^1(0,1;L^2(\omega;\R^2))\cap L^2(0,1;H^2(\omega;\R^2)),\\
	\zeta_{\varepsilon_n}&\overset{*}{\rightharpoonup}\zeta\quad{\rm in}\ L^\infty(0,1;H^1(\omega;\R^2)).
	\end{align*}

    Since $H_0^1(0,1;L^2(\omega;\R^2))$ is weakly closed, the limit $\zeta$ belongs to this space. Moreover,
	\begin{equation*}
	H^1(0,1;L^2(\omega;\R^2))\cap L^2(0,1;H^2(\omega;\R^2))\hookrightarrow C([0,1];H^1(\omega;\R^2)).
	\end{equation*}
    Therefore, $\zeta$ has the regularity required in Theorem \ref{thm control1}, while \eqref{eq control1b} follows from \eqref{eq HUM3} and lower semicontinuity.

    Let $z$ be the solution of \eqref{eq z} corresponding to $\zeta$. Since $\zeta_{\varepsilon_n}\rightharpoonup\zeta$ in $L^2(0,1;L^2(\omega;\R^2))$, and the control-to-final-state map from $L^2(0,1;L^2(\omega;\R^2))$ to $\mathcal H$ is bounded and linear, we have $z_{\varepsilon_n}(1)\rightharpoonup z(1)$ in $\mathcal H$. On the other hand, \eqref{eq HUM2} gives 
    \begin{equation*}
    \|z_{\varepsilon_n}(1)\|
    \leq\mathcal C_g\varepsilon_n^{1/2}\|z_0\|
    \longrightarrow0.
    \end{equation*}
    
    Consequently, $z(1)=0$.
    This completes the proof of Theorem \ref{thm control1}.
    \end{proof}

    \subsection{Control for the coupled linearized system}\label{Sec 3.2}
    With Theorem~\ref{thm control1} at hand, we now adapt the localized fluid control to the coupled system. We first truncate the control to finitely many Dirichlet modes and then lift the truncated path through the heat equation, obtaining the low-frequency and asymptotic null-control estimates in Propositions~\ref{prop control2} and \ref{prop control3}.
    
    Recall that $\mathsf P_N$ denotes the projection in $L^2(\omega;\R^2)$ onto the span of the first $N$ Dirichlet modes of $-\Delta_\omega$.  

    \begin{proposition}\label{prop control2}
    For any $g$ satisfying \eqref{eq g}, there exists a constant $\mathcal C_g>0$ such that the following holds. For any $N\in\N^+$ and $z_0\in \mathcal H^1$, there exists a control $\zeta_N\in H_0^1(0,1;\mathsf P_NL^2(\omega;\R^2))$ such that the solution $z_N$ of \eqref{eq z} with $\zeta=\zeta_N$ satisfies that
	\begin{align}
	\|z_N(1)\|_{\mathcal H^1}\leq\mathcal C_g\lambda_{N+1}^{-1/4}\|z_0\|_{\mathcal H^1},\label{eq control2a}\\
	\|\zeta_N\|_{H^1(0,1;L^2(\omega))}\leq\mathcal C_g\|z_0\|_{\mathcal H^1}.\label{eq control2b}
	\end{align}
    \end{proposition}

    \begin{proof} 
    By Theorem \ref{thm control1}, there exists a control $\zeta$ such that the solution $z$ of \eqref{eq z} satisfies that 
	\begin{equation*}
	z(1)=0,\quad \|\zeta\|_{H^1(0,1;L^2(\omega))}+\|\zeta\|_{L^2(0,1;H^1(\omega))}\leq\mathcal C_g\|z_0\|_{\mathcal H^1}.
	\end{equation*}
    Thus by setting $\zeta_N=\mathsf P_N\zeta$, we derive that
	\begin{equation*}
	\zeta_N\in H_0^1(0,1;\mathsf P_NL^2(\omega;\R^2)),\quad\|\zeta_N\|_{H^1(0,1;L^2(\omega))}\leq\mathcal C_g\|z_0\|_{\mathcal H^1},
	\end{equation*}
    which proves \eqref{eq control2b}.

    Additionally, let $z_N$ be the solution corresponding to $\zeta_N$, and set $e_N=z_N-z$. Then $e_N(0)=0$ and $e_N$ solves the linearized equation with the force $\overline{(\mathsf P_N-I)\zeta}$. Since $z(1)=0$, Lemma \ref{lemma l2}, the spectral estimate \eqref{eq PN} and \eqref{eq control1b} give
	\begin{align*}
	\|z_N(1)\|_{\mathcal H^1}=\|e_N(1)\|_{\mathcal H^1}&\leq C\exp\left(C(G_0^2+G_1)\right)\|(I-\mathsf P_N)\zeta\|_{L^2(0,1;L^2(\omega))}\\&
    \leq\mathcal C_g\lambda_{N+1}^{-1/4}\|\zeta\|_{L^2(0,1;H^1(\omega))}\leq\mathcal C_g\lambda_{N+1}^{-1/4}\|z_0\|_{\mathcal H^1}.
	\end{align*}
    This ensures \eqref{eq control2a} and completes the proof of Proposition \ref{prop control2}.
    \end{proof}

    \vspace{0.6em}

    Based on Proposition \ref{prop control2}, we can now establish an asymptotic null controllability result for the coupled system by lifting the low-frequency control to the heat component.

    \begin{proposition}\label{prop control3}
    For any $g$ satisfying \eqref{eq g}, there exists a constant $\mathcal C_g>0$ such that the following holds. For any $N\in\N^+$, $\delta\in(0,1/4)$ and $V_0=(v_0,p_0)\in \mathcal X$, there exists $h_{N,\delta}\in L^2(0,1;L^2(\omega;\R^2))$such that the solution $(v_{N,\delta},p_{N,\delta})$ of \eqref{eq V} by replacing $\mathsf P_N\Lambda h$ with $\mathsf P_Nh_{N,\delta}$ satisfies that $\mathsf P_Np_{N,\delta}(1)=0$,
	\begin{align}
	\|(v_{N,\delta}(1),p_{N,\delta}(1))\|_{\mathcal X}&\leq\mathcal C_g\left(\lambda_{N+1}^{-1/4}+\delta^{1/2}\right)\|V_0\|_{\mathcal X},\label{eq control3a}\\
	\|h_{N,\delta}\|_{L^2(0,1;L^2(\omega))}&\leq\mathcal C_g(\lambda_N+\delta^{-1/2})\|V_0\|_{\mathcal X}.\label{eq control3b}
	\end{align}
    \end{proposition}

    \begin{proof}
    Let $N\in\N^+$, $\delta\in(0,1/4)$ and $V_0=(v_0,p_0)\in \mathcal X$ be arbitrarily fixed. Fix $\chi\in C^\infty(\R;[0,1])$ such that
	\begin{equation*}
	\chi(\tau)=0\quad{\rm for}\ \tau\leq 1/4,\quad\chi(\tau)=1\quad{\rm for}\ \tau\geq 1/2,
	\end{equation*}
    and put
	\begin{equation*}
	\chi^-_\delta(r)=\chi(r/\delta),\quad\chi^+_\delta(r)=\chi((1-r)/\delta),\quad\chi_\delta(r)=\chi^-_\delta(r)\chi^+_\delta(r).
	\end{equation*}
    Then $\chi_\delta=1$ on $[\delta/2,1-\delta/2]$ and $\chi_\delta=0$ on $(0,\delta/4)\cup(1-\delta/4,1)$. Moreover, for any Hilbert space $Y$ and $f\in H_0^1(0,1;Y)$, the endpoint Poincar\'e inequality gives
	\begin{equation*}
	\|(1-\chi_\delta)f\|_{L^2(0,1;Y)}\leq C\delta\|\partial_rf\|_{L^2(0,1;Y)},\quad\|\chi_\delta'f\|_{L^2(0,1;Y)}\leq C\|\partial_rf\|_{L^2(0,1;Y)}.
	\end{equation*}

    \vspace{0.6em}

    Invoking Proposition \ref{prop control2}, applied to the initial value $v_0$, there exist a control $\zeta_N$ and the corresponding solution $z_N$ such that
	\begin{equation*}
	\zeta_N\in H_0^1(0,1;\mathsf P_NL^2(\omega;\R^2)),\quad\|\zeta_N\|_{H^1(0,1;L^2(\omega))}+\lambda_{N+1}^{1/4}\|z_N(1)\|_{\mathcal H^1}\leq\mathcal C_g\|v_0\|_{\mathcal H^1}.
	\end{equation*}
    We define
	\begin{equation*}
	\widehat\zeta_{N,\delta}(r):=\chi_\delta(r)\zeta_N(r)+(1-\chi^-_\delta(r))e^{r\Delta_\omega}\mathsf P_Np_0.
	\end{equation*}
    The first term reproduces the truncated Navier--Stokes control away from the endpoints, whereas the second one connects the prescribed initial heat state $\mathsf P_Np_0$ to this control path. In particular,
	\begin{equation*}
	\widehat\zeta_{N,\delta}(0)=\mathsf P_Np_0,\quad\widehat\zeta_{N,\delta}(1)=0.
	\end{equation*}

    We now realize $\widehat\zeta_{N,\delta}$ as a trajectory of the heat equation by setting
    \begin{equation*}
        h_{N,\delta}:=\partial_r\widehat\zeta_{N,\delta}-\Delta_\omega\widehat\zeta_{N,\delta}.
    \end{equation*}
    Since $(\partial_r-\Delta_\omega)e^{r\Delta_\omega}\mathsf P_Np_0=0$, the control $h_{N,\delta}$ is $\mathsf P_N L^2(\omega;\R^2)$-valued and satisfies
	\begin{equation*}
	h_{N,\delta}=\chi_\delta(\partial_r\zeta_N-\Delta_\omega\zeta_N)+\chi_\delta'\zeta_N-(\chi^-_\delta)'e^{r\Delta_\omega}\mathsf P_Np_0.
	\end{equation*}
    Therefore,
	\begin{align*}
	\|h_{N,\delta}\|_{L^2(0,1;L^2(\omega))}&\leq C(1+\lambda_N)\|\zeta_N\|_{H^1(0,1;L^2(\omega))}+C\delta^{-1/2}\|p_0\|\\
	&\leq\mathcal C_g(\lambda_N+\delta^{-1/2})\|V_0\|_{\mathcal X},
	\end{align*}
    which proves \eqref{eq control3b}.

    \vspace{0.6em}

    Let $(v_{N,\delta},p_{N,\delta})$ be the solution corresponding to $h_{N,\delta}$. Since the control is $\mathsf P_NL^2(\omega;\R^2)$-valued, by the uniqueness for the heat equation, it follows that
	\begin{equation*}
	p_{N,\delta}(r)=\widehat\zeta_{N,\delta}(r)+e^{r\Delta_\omega}(I-\mathsf P_N)p_0.
	\end{equation*}
    In particular, $\mathsf P_Np_{N,\delta}(1)=0$. Moreover, by the spectral theorem,
	\begin{equation*}
	\|p_{N,\delta}(1)\|\leq e^{-\lambda_{N+1}}\|p_0\|,\quad\|e^{r\Delta_\omega}(I-\mathsf P_N)p_0\|_{L^2(0,1;L^2(\omega))}\leq C\lambda_{N+1}^{-1/2}\|p_0\|.
	\end{equation*}

    \vspace{0.6em}

    It remains to estimate the fluid error introduced by the endpoint modification and by the uncontrolled high-frequency heat tail. The cut-off estimates yield 
	\begin{align*}
	\|\widehat\zeta_{N,\delta}-\zeta_N\|_{L^2(0,1;L^2(\omega))}&\leq C\delta\|\partial_r\zeta_N\|_{L^2(0,1;L^2(\omega))}+C\delta^{1/2}\|p_0\|\leq\mathcal C_g\left(\delta\|v_0\|_{\mathcal H^1}+\delta^{1/2}\|p_0\|\right).
	\end{align*}
    Let $w_{N,\delta}=v_{N,\delta}-z_N$. Then $w_{N,\delta}(0)=0$, and $w_{N,\delta}$ satisfies the linearized Navier--Stokes equation with forcing $\overline{\widehat\zeta_{N,\delta}-\zeta_N+e^{r\Delta_\omega}(I-\mathsf P_N)p_0}$. Applying Lemma \ref{lemma l2}, we obtain
	\begin{equation*}
	\|w_{N,\delta}(1)\|_{\mathcal H^1}\leq\mathcal C_g\left(\delta\|v_0\|_{\mathcal H^1}+(\delta^{1/2}+\lambda_{N+1}^{-1/2})\|p_0\|\right).
	\end{equation*}
    
    Combining this estimate with the terminal bound for $z_N$ and the spectral decay of $p_{N,\delta}(1)$ gives \eqref{eq control3a}. This completes the proof of Proposition \ref{prop control3}.
    \end{proof}

    \subsection{Weak observability and regularized stabilization}\label{Sec 3.3}

    The asymptotic null control from Proposition~\ref{prop control3} is now converted into the stabilization estimate needed in the Malliavin argument. Control--observation duality yields the weak observability inequality in Proposition~\ref{prop obs2}, while Fenchel--Rockafellar duality applied to a regularized quadratic problem gives the explicit control and estimates in Theorem~\ref{thm control4}.

    \begin{proof}[Proof of Proposition \ref{prop obs2}]
    Let $Z_1=(w_1,q_1)\in \mathcal X$ be arbitrarily given, and take $\delta_N=\min\{1/8,\lambda_N^{-2}\}$. Since $\lambda_{N+1}\geq\lambda_N\geq\lambda_1>0$, one has
	\begin{equation*}
	\lambda_{N+1}^{-1/4}+\delta_N^{1/2}\leq C\lambda_N^{-1/4},\quad\lambda_N+\delta_N^{-1/2}\leq C\lambda_N.
	\end{equation*}
    By Proposition \ref{prop control3} with $\delta=\delta_N$, for any $V_0\in \mathcal X$, there exists a control $h_N\in H_{0,1}$ such that
	\begin{equation*}
	\|\mathcal J_gV_0+\mathcal R_{N,g}h_N\|_{\mathcal X}\leq\mathcal C_g\lambda_N^{-1/4}\|V_0\|_{\mathcal X},\quad\|h_N\|_{H_{0,1}}\leq\mathcal C_g\lambda_N\|V_0\|_{\mathcal X}.
	\end{equation*}
    
    On the other hand, let $(w(r),q(r))=\mathcal J_{r,1,g}^*Z_1$. From the definition of $\mathcal R_{N,g}$, we have
	\begin{equation*}
	(\mathcal R_{N,g}^*Z_1)(r)=\mathsf P_Nq(r).
	\end{equation*}
    Consequently, for this control $h_N$,
	\begin{align*}
	|\langle\mathcal J_g^*Z_1,V_0\rangle_{\mathcal X}|&=|\langle Z_1,\mathcal J_gV_0+\mathcal R_{N,g}h_N\rangle_{\mathcal X}-\langle\mathcal R_{N,g}^*Z_1,h_N\rangle_{H_{0,1}}|\\
	&\leq\mathcal C_g\left(\lambda_N^{-1/4}\|Z_1\|_{\mathcal X}+\lambda_N\|\mathsf P_Nq\|_{H_{0,1}}\right)\|V_0\|_{\mathcal X}.
	\end{align*}
    Finally, taking the supremum over all $V_0\in \mathcal X$ with $\|V_0\|_{\mathcal X}\leq1$ and recalling that $\mathcal J_g^*Z_1=(w(0),q(0))$, we obtain \eqref{eq obs2}. This completes the proof of Proposition \ref{prop obs2}.
    \end{proof}

    We now combine Proposition \ref{prop obs2} with the Fenchel--Rockafellar duality for a regularized quadratic minimization problem to prove Theorem \ref{thm control4}.

    \begin{proof}[Proof of Theorem \ref{thm control4}]
    Let us define two convex and continuous functions by
	\begin{align*}
	F\colon H_{0,1}\rightarrow\R^+,\quad F(h)=\frac12\|h\|_{H_{0,1}}^2,\quad G\colon \mathcal X\rightarrow\R^+,\quad G(V)=\frac1{2\beta}\|V+\mathcal J_gV_0\|_{\mathcal X}^2.
	\end{align*}
    
    We then consider the following minimization problem
	\begin{equation}\label{eq FR1}
	\inf_{h\in H_{0,1}}J(h)=\inf_{h\in H_{0,1}}\left(F(h)+G(\mathcal A_{N,g}h)\right).
	\end{equation}
    Note that $J$ is coercive and strictly convex, so problem \eqref{eq FR1} admits a unique solution. The Fenchel conjugates $F^*\colon H_{0,1}\rightarrow\R^+$ and $G^*\colon \mathcal X\rightarrow\R$ are given by
	\begin{align*}
	&F^*(h)=\sup_{\widetilde h\in H_{0,1}}\left(\langle h,\widetilde h\rangle_{H_{0,1}}-F(\widetilde h)\right)=\frac12\|h\|_{H_{0,1}}^2=F(h),\\
	&G^*(Z)=\sup_{\widetilde V\in \mathcal X}\left(\langle Z,\widetilde V\rangle_{\mathcal X}-G(\widetilde V)\right)=\frac\beta2\|Z\|_{\mathcal X}^2-\langle Z,\mathcal J_gV_0\rangle_{\mathcal X}.
	\end{align*}
    After changing the sign of the dual variable, the corresponding dual problem is given by
	\begin{equation}\label{eq FR2}
	\sup_{Z\in \mathcal X}J^*(Z)=\sup_{Z\in \mathcal X}\left(-G^*(Z)-F^*(-\mathcal A_{N,g}^*Z)\right)=\sup_{Z\in \mathcal X}\left(-\frac12\langle(\mathcal G_{N,g}+\beta)Z,Z\rangle_{\mathcal X}+\langle Z,\mathcal J_gV_0\rangle_{\mathcal X}\right).
	\end{equation}
    Clearly, $F$ and $G$ are finite and continuous. Thus, by the Fenchel--Rockafellar duality theorem \cite{Rockafellar-67,ET-99}, problems \eqref{eq FR1} and \eqref{eq FR2} admit solutions and
	\begin{equation*}
	\inf_{h\in H_{0,1}}J(h)=\sup_{Z\in \mathcal X}J^*(Z).
	\end{equation*}
    
    Since $\mathcal G_{N,g}$ is bounded, self-adjoint and non-negative and $\beta>0$, the functional in \eqref{eq FR2} is strictly concave, and its unique maximizer $Z^*\in \mathcal X$ is given by
	\begin{equation}\label{eq FR3}
	(\mathcal G_{N,g}+\beta)Z^*=\mathcal J_gV_0,\quad Z^*=(\mathcal G_{N,g}+\beta)^{-1}\mathcal J_gV_0.
	\end{equation}
    Meanwhile, using the Fenchel dual relation, the unique minimizer $h^*$ of $J$ is given by
	\begin{equation}\label{eq FR4}
	h^*=-\mathcal A_{N,g}^*Z^*=-\mathcal A_{N,g}^*(\mathcal G_{N,g}+\beta)^{-1}\mathcal J_gV_0.
	\end{equation}
    This agrees with the control $h$ defined in Theorem \ref{thm control4}. In what follows, we verify the stabilization estimates \eqref{eq control4a},\eqref{eq control4b}.

    By definition and the Fenchel duality relation, one has
	\begin{equation}\label{eq FR5}
	\frac12\|h\|_{H_{0,1}}^2+\frac1{2\beta}\|V(1)\|_{\mathcal X}^2=J(h)=J^*(Z^*)=\frac12\langle Z^*,\mathcal J_gV_0\rangle_{\mathcal X}=\frac12\langle\mathcal J_g^*Z^*,V_0\rangle_{\mathcal X}.
	\end{equation}
    Moreover, from our construction, it follows that
	\begin{equation}\label{eq FR6}
	V(1)=\mathcal J_gV_0+\mathcal A_{N,g}h=\mathcal J_gV_0-\mathcal G_{N,g}Z^*=\beta Z^*,
	\end{equation}
    where the last equality follows from \eqref{eq FR3}.

    Let $(w^*(r),q^*(r))=\mathcal J_{r,1,g}^*Z^*$. By the definitions of $\mathcal A_{N,g}$ and $\Lambda$, we have
	\begin{equation*}
	(\mathcal A_{N,g}^*Z^*)(r)=\Lambda\mathsf P_Nq^*(r).
	\end{equation*}
    Hence, by \eqref{eq FR4} and the definition of $\widehat b_N$,
	\begin{equation*}
	\|\mathsf P_Nq^*\|_{H_{0,1}}\leq\widehat b_N\|\Lambda\mathsf P_Nq^*\|_{H_{0,1}}=\widehat b_N\|h\|_{H_{0,1}}.
	\end{equation*}
    Applying the weak observability inequality in Proposition \ref{prop obs2}, the Cauchy--Schwarz inequality and \eqref{eq FR6}, it follows that
	\begin{equation}\label{eq FR7}
	\begin{aligned}
	|\langle\mathcal J_g^*Z^*,V_0\rangle_{\mathcal X}|&\leq\|\mathcal J_g^*Z^*\|_{\mathcal X}\|V_0\|_{\mathcal X}\leq\mathcal C_g\left(\lambda_N\|\mathsf P_Nq^*\|_{H_{0,1}}+\lambda_N^{-1/4}\|Z^*\|_{\mathcal X}\right)\|V_0\|_{\mathcal X}\\
	&\leq\mathcal C_g\left(\widehat b_N\lambda_N\|h\|_{H_{0,1}}+\beta^{-1}\lambda_N^{-1/4}\|V(1)\|_{\mathcal X}\right)\|V_0\|_{\mathcal X}.
	\end{aligned}
	\end{equation}
    Combining \eqref{eq FR5} and \eqref{eq FR7}, we obtain
	\begin{equation*}
	\|h\|_{H_{0,1}}^2+\frac1\beta\|V(1)\|_{\mathcal X}^2\leq\mathcal C_g\left(\widehat b_N\lambda_N\|h\|_{H_{0,1}}+\beta^{-1}\lambda_N^{-1/4}\|V(1)\|_{\mathcal X}\right)\|V_0\|_{\mathcal X}.
	\end{equation*}
    Using Young's inequality, one has
	\begin{equation*}
	\|h\|_{H_{0,1}}^2+\frac1\beta\|V(1)\|_{\mathcal X}^2\leq\mathcal C_g^2\left(\widehat b_N^2\lambda_N^2+\beta^{-1}\lambda_N^{-1/2}\right)\|V_0\|_{\mathcal X}^2,
	\end{equation*}
    which implies \eqref{eq control4a},\eqref{eq control4b} and completes the proof.
    \end{proof}

    \section{Exponential mixing for  Navier--Stokes equations}\label{Sec 4}

    With the stabilization results of the previous sections in hand, we now prove Theorems~\ref{thm 1} and \ref{thm 2}. We first verify the Lyapunov, asymptotic strong Feller, and irreducibility hypotheses of Theorem~\ref{thm HM} for the coupled Markov process, thereby obtaining Theorem~\ref{thm 2}. Theorem~\ref{thm 1} then follows by starting the heat component from its stationary law and projecting the resulting mixing statement onto the velocity component.

    \begin{theorem}\label{thm HM}{\rm(}\hspace{-0.1mm}{\rm\cite[Theorems 3.4, 4.5]{HM-08})}
    Let $Z(t,x)$ be a stochastic semiflow on a Hilbert space $X$, with $C^1$ dependence on $x\in X$, and let $\{P_t\}_{t\geq0}$ and $\{P_t^*\}_{t\geq0}$ be the associated Markov semigroups. Assume that 
    \begin{itemize}
        \item[(1)] There exist constants $C,\gamma>0$ and a decreasing function $\xi\colon[0,1]\to(0,\infty)$ with $\xi(1)<1$ such that for any $x\in X$, $r\in[1/4,3]$ and $t\in[0,1]$,
	\begin{equation}\label{eq HM1}
	\E\exp\left(r\gamma\|Z(t,x)\|_{X}^2\right)\left(1+\|\nabla_xZ(t,x)\|_{\mathcal L(X)}\right)\leq C\exp\left(r\gamma\xi(t)\|x\|_{X}^2\right).
	\end{equation}
        \item[(2)] There exist constants $C,\alpha>0$ such that for any Fr\'echet differentiable function $\varphi\colon X\to\R$, $x\in X$ and $t\geq0$,
	\begin{equation}\label{eq HM2}
	\|\nabla P_t\varphi(x)\|_{X}\leq C\exp\left(\tfrac{\gamma}{2}\|x\|_{X}^2\right)\left(\sqrt{P_t(|\varphi|^2)(x)}+e^{-\alpha t}\sqrt{P_t(\|\nabla\varphi\|_{X}^2)(x)}\right).
	\end{equation}
        \item[(3)] For any $\varepsilon,R>0$ and $r\in(0,1)$, there exists $T>0$ such that for any $t\geq T$,
	\begin{equation}\label{eq HM3}
	\inf_{x,x'\in B_{X}(0,R)}\sup_{\Gamma\in\mathscr C(P_t^*\delta_x,P_t^*\delta_{x'})}\Gamma\left\{(y,y')\in X\times X:\rho_r(y,y')\leq\varepsilon\right\}>0,
	\end{equation}
        where $\mathscr C(\mu,\mu')$ is the set of couplings of $\mu$ and $\mu'$, and
	\begin{equation*}
	\rho_s(x,y):=\inf_{\substack{\ell\in C^1([0,1];X),\\
	\ell(0)=x,\ \ell(1)=y}}\int_0^1\exp\left(s\gamma\|\ell(r)\|_{X}^2\right)\|\ell'(r)\|_{X}dr,\quad s\in(0,1].
	\end{equation*}
    \end{itemize}
    
    Then there exist constants $C,\alpha>0$ such that
	\begin{equation*}
	\rho_1(P_t^*\mu,P_t^*\mu')\leq Ce^{-\alpha t}\rho_1(\mu,\mu')\quad\forall\,\mu,\mu'\in\mathcal P(X),\;t\geq0.
	\end{equation*}
    In particular, $\{P_t^*\}_{t\geq0}$ admits a unique invariant measure $\mu_*\in\mathcal P(X)$, and
	\begin{equation*}
	\left|P_t\varphi(x)-\int_{X}\varphi d\mu_*\right|\leq Ce^{-\alpha t}\exp\left(\gamma\|x\|_{X}^2\right)\left\|\varphi-\int_{X}\varphi d\mu_*\right\|_\gamma
	\end{equation*}
    for any $\varphi\in\mathcal O_\gamma$, $x\in X$ and $t\geq0$, where
	\begin{equation*}
	\mathcal O_\gamma:=\left\{\varphi\in C^1(X):\|\varphi\|_\gamma<\infty\right\},\quad\|\varphi\|_\gamma:=\sup_{x\in X}\left(e^{-\gamma\|x\|_{X}^2}\left(|\varphi(x)|+\|\nabla\varphi(x)\|_{X}\right)\right).
	\end{equation*}
    \end{theorem}

    We shall apply this criterion to the stochastic flow $U(t,U_0)$ associated with \eqref{eq NS},\eqref{eq xi}, using an equivalent Hilbert norm on the phase space $\mathcal X$. This norm is introduced in the proof of Theorem \ref{thm 2}. It remains to prove the asymptotic strong Feller property and irreducibility, which is done in the next two subsections.

\begin{remark}\label{Rem 1} We mention that the invariant measure of the heat equation \eqref{eq eta} admits an explicit formulation.  Let $\mu_s$ be a Gaussian measure on $L^2(\omega;\R^2)$ with covariance $\Sigma$ given by
	\begin{equation*}
	\mu_s:=\mathcal N(0,\Sigma),\quad\Sigma h:=\sum_{j\in\N^+}\tfrac{b_j^2}{2\lambda_j}\langle h,\psi_j\rangle_{L^2(\omega)}\psi_j.
	\end{equation*}
Then $\mu_s$ is the unique invariant measure of \eqref{eq eta}. Moreover, $\eta^{s}$ satisfies that $\law(\eta^{s}(t))=\mu_s$ for any $t\in\R$. Additionally, setting $\Pi_u$ ae the projection from $\mathcal X$ to $L^2(\omega;\R^2)$ and noting the heat component evolves autonomously, we have $(\Pi_\eta)_\#\mu_*$ is invariant for \eqref{eq eta}. Thus $(\Pi_\eta)_\#\mu_*=\mu_s$, and the invariant measure $\mu_*$ on $\mathcal X$  can be disintegrated as 
	\begin{equation}\label{eq mu_*}
	\mu_*(du,d\eta)=p_\eta(du)\,\mu_s(d\eta),
	\end{equation}
where $p_\eta$ denotes a probability kernel on $\mathcal H^1$.  The measure $\mu$ in Theorem \ref{thm 1} can be described as 
	\begin{equation*}
	\mu=(\Pi_u)_\#\mu_*=\int_{L^2(\omega;\R^2)}p_\eta\,\mu_s(d\eta),
	\end{equation*}
where $\Pi_u$ is the projection from $\mathcal X$ to $\mathcal H^1$. In particular, $\mu$ is the velocity marginal of the unique stationary law of the coupled system.
\end{remark}

    \subsection{Asymptotic strong Feller property}\label{Sec 4.1}

    With Theorem~\ref{thm control4} at hand, we prove the asymptotic strong Feller estimate in Proposition~\ref{prop ASF}. Following the Malliavin-control argument of \cite{HM-06,HM-11b}, we construct a finite-dimensional, generally non-adapted control that asymptotically cancels the tangent process; the stabilization bounds and the Malliavin integration-by-parts formula then yield the desired gradient estimate.

    \begin{proposition}\label{prop ASF}
    For the equation given by \eqref{eq NS},\eqref{eq xi} and for any $B_0,\gamma,\alpha>0$, there exists a constant $N\in\N^+$ such that if the sequence $\{b_j\}_{j\in\N^+}$ satisfies 
	\begin{equation*}
	\sum_{j\in\N^+}b_j^2\leq B_0\quad\text{and}\quad b_j\neq0\quad\text{for }1\leq j\leq N,
	\end{equation*}
    then there exists a constant $C>0$ such that, for any Fr\'echet differentiable function $\varphi\colon \mathcal X\to\R$, $U_0\in \mathcal X$ and $t\geq0$,
	\begin{align}\label{eq ASF}
	\|\nabla P_t\varphi(U_0)\|_{\mathcal X}\leq C\exp\left(\gamma\|U_0\|_{\mathcal X}^2\right)\left(\sqrt{P_t(|\varphi|^2)(U_0)}+e^{-\alpha t}\sqrt{P_t(\|\nabla\varphi\|_{\mathcal X}^2)(U_0)}\right).
	\end{align}
    \end{proposition}

    The overall strategy is inspired by the approach in \cite{HM-06,HM-11b,LXZ-26}.  Namely, it suffices to construct controls $v\in L^2(\Omega;H)$, which are not necessarily adapted, to stabilize the random process 
	\begin{equation}\label{eq rho}
	\rho_t=\mathcal{J}_{0,t}V_0-\mathcal{A}_{0,t}v.
	\end{equation}

    Specifically, we now describe the construction of control $v$. Let $N\in\N^+$ and $\beta>0$ be fixed. For $n\in\N$, define
	\begin{equation*}
	\hat{\mathcal J}_n=\mathcal J_{2n,2n+1},\quad\check{\mathcal J}_n=\mathcal J_{2n+1,2n+2},\quad\mathcal A_n=\mathcal A_{2n+1,2n+2,N},\quad\mathcal M_n=\mathcal A_n\mathcal A_n^*.
	\end{equation*}
    For each $n\in\N$, we define $v_n\in L^2(\Omega;H_{2n+1,2n+2})$ by
	\begin{equation}\label{eq v1}
	v_n=\mathcal A_n^*(\mathcal M_n+\beta)^{-1}\check{\mathcal J}_n\hat{\mathcal J}_n\rho_{2n}.
	\end{equation}
    The global control $v$ is then defined as
	\begin{equation}\label{eq v2}
	v(t)=\begin{cases}
	v_n(t),&t\in(2n+1,2n+2),\;n\in\N,\\
	0,&t\in(2n,2n+1),\;n\in\N.
	\end{cases}
	\end{equation}
    By the definition of $v$, it follows that for any $n\in\N$, 
	\begin{equation*}
	\rho_0=V_0,\quad\rho_{2n+1}=\hat{\mathcal J}_n\rho_{2n},\quad\rho_{2n+2}=\check{\mathcal J}_n\rho_{2n+1}-\mathcal A_nv_n=\beta(\mathcal M_n+\beta)^{-1}\check{\mathcal J}_n\hat{\mathcal J}_n\rho_{2n}.
	\end{equation*}

    \vspace{0.6em}

    Moreover, applying the stabilization result, i.e. Theorem \ref{thm control4}, we derive the following lemma.    
    \begin{lemma}\label{lemma ASF1}
    For the equation given by \eqref{eq NS},\eqref{eq xi} and any $B_0,\gamma,\alpha>0$, there exists $N\in\N^+$ such that if the sequence $\{b_j\}_{j\in\N^+}$ satisfies
	\begin{equation*}
	\sum_{j\in\N^+}b_j^2\leq B_0\quad\text{and}\quad b_j\neq0\quad\text{for }1\leq j\leq N,
	\end{equation*}
    then there exist constants $C,\beta>0$ such that the process $\rho$ defined by \eqref{eq rho} with $v$ given by \eqref{eq v1},\eqref{eq v2}, satisfies that
	\begin{align}
	\E\|\rho_t\|_{\mathcal X}^2\leq Ce^{-\alpha t}\exp\left(\gamma\|U_0\|_{\mathcal X}^2\right)\|V_0\|_{\mathcal X}^2\quad\forall\,U_0,V_0\in \mathcal X,\;t\geq0,\label{eq rho2}\\
	\E\left|\int_0^tv(s)dW(s)\right|^2\leq C\exp\left(\gamma\|U_0\|_{\mathcal X}^2\right)\|V_0\|_{\mathcal X}^2\quad\forall\,U_0,V_0\in \mathcal X,\;t\geq0.\label{eq v3}
	\end{align}
    \end{lemma}

   Invoking Lemma \ref{lemma ASF1}, we establish Proposition~\ref{prop ASF} as follows.
   
    \begin{proof}[Proof of Proposition \ref{prop ASF}]
    Fix $V_0\in \mathcal X$ with $\|V_0\|_{\mathcal X}\leq1$, and let $v$ and $\rho$ be given by the preceding construction, with the parameters chosen in Lemma \ref{lemma ASF1}. By \eqref{eq D def} and \eqref{eq rho},
	\begin{align*}
	\langle\nabla P_t\varphi(U_0),V_0\rangle_{\mathcal X}&=\E\left\langle\nabla\varphi(U_t),\mathcal J_{0,t}V_0\right\rangle_{\mathcal X}\\
	&=\E\left\langle\nabla\varphi(U_t),\mathcal A_{0,t}v\right\rangle_{\mathcal X}+\E\left\langle\nabla\varphi(U_t),\rho_t\right\rangle_{\mathcal X}\\
	&=\E\left\langle\mathcal D\varphi(U_t),v\right\rangle_{H}+\E\left\langle\nabla\varphi(U_t),\rho_t\right\rangle_{\mathcal X}.
	\end{align*}
    Note that the Malliavin integration-by-parts formula implies that
	\begin{equation*}
	\E\left\langle\mathcal D\varphi(U_t),v\right\rangle_{H}=\E\left(\varphi(U_t)\int_0^tv(s)dW(s)\right).
	\end{equation*}
    Therefore, combining \eqref{eq rho2}, and \eqref{eq v3}, we derive
	\begin{align*}
	|\langle\nabla P_t\varphi(U_0),V_0\rangle_{\mathcal X}|\leq C\exp\left(\gamma\|U_0\|_{\mathcal X}^2\right)\sqrt{P_t(|\varphi|^2)(U_0)}+Ce^{-\alpha t}\exp\left(\gamma\|U_0\|_{\mathcal X}^2\right)\sqrt{P_t(\|\nabla\varphi\|_{\mathcal X}^2)(U_0)}.
	\end{align*}
    
    Consequently, by taking the supremum over all $V_0\in \mathcal X$ with $\|V_0\|_{\mathcal X}\leq1$, this proves \eqref{eq ASF}, thereby completing the proof of Proposition \ref{prop ASF}.
    \end{proof}

    \vspace{0.3em}

     It remains to prove Lemma~\ref{lemma ASF1}. The key input is the following conditional one-block contraction. Once this estimate is available, its iteration and the corresponding bounds on the divergence cost follow by standard arguments; the details are given in Appendix~\ref{Sec D}.

    \begin{lemma}\label{lemma ASF2}
    For the equation given by \eqref{eq NS},\eqref{eq xi} and any $B_0,\gamma,\varepsilon>0$, there exists $N\in\N^+$ such that if the sequence $\{b_j\}_{j\in\N^+}$ satisfies
	\begin{equation*}
	\sum_{j\in\N^+}b_j^2\leq B_0\quad\text{and}\quad b_j\neq0\quad\text{for }1\leq j\leq N,
	\end{equation*}
    then there exists a constant $\beta>0$ such that the process $\rho$ defined by \eqref{eq rho} with $v$ given by \eqref{eq v1},\eqref{eq v2}, satisfies that
	\begin{equation}\label{eq ASF step}
	\E\left(\|\rho_{2n+2}\|_{\mathcal X}^4|\mathcal F_{2n}\right)\leq\varepsilon\exp\left(\gamma\|U_{2n}\|_{\mathcal X}^2\right)\|\rho_{2n}\|_{\mathcal X}^4\quad\forall\,U_0,V_0\in \mathcal X,\;n\in\N.
	\end{equation}
    \end{lemma}

    \begin{proof}
    Let $\gamma,\varepsilon>0$ be fixed.
    Throughout this proof, $\gamma_0$ denotes the constant in Lemma \ref{lemma L bdd}.

    \vspace{0.3em}

    \noindent{\it Step 1.}
    We first derive two estimates that are uniform with respect to the block index $n$. Set
	\begin{align}
	Y_n=\|u\|_{L^\infty((2n+1,2n+2)\times D)}+\|\partial_tu\|_{L^2((2n+1,2n+2);L^4(D))},\quad K_n=\sup_{2n\leq r\leq t\leq2n+2}\|\mathcal J_{r,t}\|_{\mathcal L(\mathcal X)}.\label{eq ASF Kn}
	\end{align}
    Let us first establish the estimates used below. For every $m\geq1$ and every $\kappa>0$, there exists a constant $C=C(m,\kappa,B_0)>0$ such that
	\begin{equation}\label{eq ASF K2}
	\E\left(K_n^m|\mathcal F_{2n}\right)\leq C\exp\left(\kappa\|U_{2n}\|_{\mathcal X}^2\right),\quad n\in\N.
	\end{equation}
    Indeed, by \eqref{eq ASF Kn} and Lemma \ref{lemma J1}, for every $\theta>0$,
	\begin{equation*}
	K_n\leq C_\theta\exp\left(\theta\int_{2n}^{2n+2}\|u_s\|_{\mathcal H^2}^2ds\right).
	\end{equation*}
   Then for any  $\theta\leq\gamma_0/m$,  \eqref{eq L2H2 bdd} with initial time $2n$ gives
	\begin{align*}
	\E\left(K_n^m|\mathcal F_{2n}\right)\leq C_\theta\E\left(\exp\left(m\theta\int_{2n}^{2n+2}\|u_s\|_{\mathcal H^2}^2ds\right)|\mathcal F_{2n}\right)\leq C_\theta\exp\left(cm\theta\|U_{2n}\|_{\mathcal X}^2\right).
	\end{align*}
    Choosing $\theta$ so small that $cm\theta\leq\kappa$ proves \eqref{eq ASF K2}.

    \vspace{0.3em}
    We shall also use the following local tail estimate: for every $R>0$ and $M>0$,
	\begin{equation*}
	\Pb\left(Y_n>M|\mathcal F_{2n}\right)\leq\tau_R(M)\quad\text{on }\{\|U_{2n}\|_{\mathcal X}\leq R\},\quad n\in\N.
	\end{equation*}
    Here $\tau_R$ is the deterministic function
	\begin{equation*}
	\tau_R(M):=C\exp\left(c\gamma_0R^2-\gamma_0M^2/8\right)+C(1+R^2)M^{-1},\quad M>0.
	\end{equation*}
    To prove the estimate, decompose $Y_n=Y_n^{(1)}+Y_n^{(2)}$ as
	\begin{equation*}
	Y_n^{(1)}=\|u\|_{L^\infty((2n+1,2n+2)\times D)},\quad Y_n^{(2)}=\|\partial_tu\|_{L^2((2n+1,2n+2);L^4(D))}.
	\end{equation*}
    Then applying \eqref{eq L-inf bdd} and  \eqref{eq L bdd} with $m=1$, we obtain
	\begin{align*}
	\Pb\left(Y_n^{(1)}>M/2|\mathcal F_{2n}\right)&\leq C\exp\left(c\gamma_0\|U_{2n}\|_{\mathcal X}^2-\gamma_0M^2/8\right),\\
	\Pb\left(Y_n^{(2)}>M/2|\mathcal F_{2n}\right)&\leq C(1+\|U_{2n}\|_{\mathcal X}^2)M^{-1}.
	\end{align*}
    The union bound gives the displayed estimate for $\tau_R(M)$, and clearly $\tau_R(M)\to0$ as $M\to\infty$.

    \vspace{0.6em}

    \noindent{\it Step 2.}
    Let us set
	\begin{equation*}
	\Omega_{n,M}:=\{Y_n\leq M\}.
	\end{equation*}
    On $\Omega_{n,M}$, the translated trajectory satisfies the assumptions of Theorem~\ref{thm control4} with a deterministic control cost bounded by
	\begin{equation*}
	G_0=\|u\|_{L^\infty((2n+1,2n+2)\times D)},\quad G_1=\|\partial_tu\|_{L^2(2n+1,2n+2;L^4(D))}.
	\end{equation*}
    Thus invoking \ref{thm control4} and setting $\mathcal C_M=C\exp(\exp(CM^{20/3}+\exp(CM^2)))$, on the event $\Omega_{n,M}$, we derive that
	\begin{equation*}
	\|\rho_{2n+2}\|_{\mathcal X}\leq\mathcal C_M\left(\widehat b_N\lambda_N\beta^{1/2}+\lambda_N^{-1/4}\right)\|\rho_{2n+1}\|_{\mathcal X}\leq\mathcal C_M\left(\widehat b_N\lambda_N\beta^{1/2}+\lambda_N^{-1/4}\right)K_n\|\rho_{2n}\|_{\mathcal X}
	\end{equation*}
    on $\Omega_{n,M}$. On the other hand, one has
	\begin{equation*}
	\|\rho_{2n+2}\|_{\mathcal X}=\|\beta(\mathcal M_n+\beta)^{-1}\check{\mathcal J}_n\hat{\mathcal J}_n\rho_{2n}\|_{\mathcal X}\leq K_n\|\rho_{2n}\|_{\mathcal X}\quad\text{a.s.},
	\end{equation*}
    where the latter bound follows from $\|\beta(\mathcal M_n+\beta)^{-1}\|_{\mathcal L(\mathcal X)}\leq1$.
    \vspace{0.6em}

    \noindent{\it Step 3.}    
    Taking $m=4$ and $\kappa=\gamma/2$ in \eqref{eq ASF K2}, there exists a constant $C_1>0$ such that
	\begin{equation*}
	\E\left(K_n^4|\mathcal F_{2n}\right)\leq C_1\exp\left(\tfrac{\gamma}{2}\|U_{2n}\|_{\mathcal X}^2\right).
	\end{equation*}
    With $\gamma,\varepsilon,C_1>0$ fixed, we now take $R>0$ such that $C_1\leq\varepsilon\exp\left(\tfrac{\gamma}{2}R^2\right)$.  Since $\rho_{2n}$ is $\mathcal F_{2n}$-measurable, the rough estimate above gives
	\begin{align*}
	\E\left(\|\rho_{2n+2}\|_{\mathcal X}^4|\mathcal F_{2n}\right)\leq\E\left(K_n^4|\mathcal F_{2n}\right)\|\rho_{2n}\|_{\mathcal X}^4\leq C_1\exp\left(\tfrac{\gamma}{2}\|U_{2n}\|_{\mathcal X}^2\right)\|\rho_{2n}\|_{\mathcal X}^4.
	\end{align*}
    
    \noindent{\it Case a: $\|U_{2n}\|_{\mathcal X}>R$.}
    In this case, the definition of $R$ makes the last coefficient bounded by $\varepsilon\exp(\gamma\|U_{2n}\|_{\mathcal X}^2)$, so \eqref{eq ASF step} holds.

    \vspace{0.3em}

    \noindent{\it Case b: $\|U_{2n}\|_{\mathcal X}\leq R$.}
    Taking $m=4,8$ and $\kappa=1$ in \eqref{eq ASF K2}, and absorbing $\exp(\|U_{2n}\|_{\mathcal X}^2)\leq e^{R^2}$ into the constant, we derive that
	\begin{equation*}
	\E\left(K_n^4|\mathcal F_{2n}\right)+\E\left(K_n^8|\mathcal F_{2n}\right)^{1/2}\leq C_R.
	\end{equation*}
    Let $M>0$ be arbitrary for the moment. Using the two estimates from Step 2 and Cauchy's inequality on $\Omega_{n,M}^c$, we get
	\begin{align*}
	\E\left(\|\rho_{2n+2}\|_{\mathcal X}^4|\mathcal F_{2n}\right)&=\E\left(\|\rho_{2n+2}\|_{\mathcal X}^4\mathbf1_{\Omega_{n,M}}|\mathcal F_{2n}\right)+\E\left(\|\rho_{2n+2}\|_{\mathcal X}^4\mathbf1_{\Omega_{n,M}^c}|\mathcal F_{2n}\right)\\
	&\leq\mathcal C_M^4\left(\widehat b_N\lambda_N\beta^{1/2}+\lambda_N^{-1/4}\right)^4\E\left(K_n^4|\mathcal F_{2n}\right)\|\rho_{2n}\|_{\mathcal X}^4+\E\left(K_n^4\mathbf1_{\Omega_{n,M}^c}|\mathcal F_{2n}\right)\|\rho_{2n}\|_{\mathcal X}^4\\
	&\leq C_R\mathcal C_M^4\left(\widehat b_N\lambda_N\beta^{1/2}+\lambda_N^{-1/4}\right)^4\|\rho_{2n}\|_{\mathcal X}^4+\E\left(K_n^8|\mathcal F_{2n}\right)^{\tfrac{1}{2}}\Pb\left(\Omega_{n,M}^c|\mathcal F_{2n}\right)^{\tfrac{1}{2}}\|\rho_{2n}\|_{\mathcal X}^4\\
	&\leq C_R\left[\mathcal C_M^4\left(\widehat b_N\lambda_N\beta^{1/2}+\lambda_N^{-1/4}\right)^4+\tau_R(M)^{1/2}\right]\|\rho_{2n}\|_{\mathcal X}^4.
	\end{align*}
    
    We then first take the parameter $M$ so large that
	\begin{equation*}
	C_R\tau_R(M)^{1/2}\leq\varepsilon/2.
	\end{equation*}
    With this $M$ fixed, we then take $N$ sufficiently large to derive
    \begin{equation*}
        \lambda_N^{-1/4}\leq(2\mathcal C_MC_R^{1/4})^{-1}(\varepsilon/2)^{1/4}.
    \end{equation*}
    Finally, after $M,N$ are fixed, we choose $\beta>0$ so that
	\begin{align*}
	\widehat b_N\lambda_N\beta^{1/2}\leq(2\mathcal C_MC_R^{1/4})^{-1}(\varepsilon/2)^{1/4}.
	\end{align*} 
    
    Therefore the preceding estimate implies that, for $\|U_{2n}\|_{\mathcal X}\leq R$, it follows that 
	\begin{equation*}
	\E\left(\|\rho_{2n+2}\|_{\mathcal X}^4|\mathcal F_{2n}\right)\leq\varepsilon\|\rho_{2n}\|_{\mathcal X}^4\leq\varepsilon\exp\left(\gamma\|U_{2n}\|_{\mathcal X}^2\right)\|\rho_{2n}\|_{\mathcal X}^4.
	\end{equation*}

    Consequently, collecting these two cases, we derive inequality \eqref{eq ASF step}, which thereby completing the proof of Lemma \ref{lemma ASF2}. 
    \end{proof}

    \vspace{0.6em}

    \subsection{Irreducibility}
    This subsection establishes the accessibility estimate in Proposition~\ref{prop irreducibility}, which provides the irreducibility input for Theorem~\ref{thm HM}. The proof combines the global stability of the unforced system with the usual support lemmas.

    \begin{proposition}\label{prop irreducibility}
    For the equation given by \eqref{eq NS},\eqref{eq eta}, assume that the sequence $\{b_j\}_{j\in\N^+}$ satisfies $\sum_{j\in\N^+}b_j^2<\infty$. Then for any $\varepsilon,R>0$, there exists $T=T(\varepsilon,R)>0$ such that for any $t\geq T$, the solution $U$ of \eqref{eq NS2} satisfies that
	\begin{equation}\label{eq irreducibility}
	\inf_{U_0\in B_{\mathcal X}(0,R)}\Pb\left(\|U(t,U_0)\|_{\mathcal X}\leq\varepsilon\right)>0.
	\end{equation}
    \end{proposition}

    \begin{proof}
    Let $\varepsilon,R>0$ be fixed. By the global exponential stability in Lemma \ref{lemma stable}, for the unforced solution $\Phi$ of \eqref{eq NSunforced}, there exists $T=T(\varepsilon,R)>0$ such that
	\begin{equation}\label{eq stable}
	\|\Phi(t,U_0)\|_{\mathcal X}\leq\varepsilon/2\quad\forall\,U_0\in B_{\mathcal X}(0,R),\;t\geq T.
	\end{equation}
    Fix any $t\geq T$. We compare the stochastic solution with the unforced trajectory driven by the same initial condition. Let
	\begin{equation*}
	Z_s=\sum_{j\in\N^+}b_j\int_0^se^{(s-r)\Delta_\omega}\psi_jd\beta_j(r),\quad0\leq s\leq t,
	\end{equation*}
    be the stochastic convolution associated with the heat equation $\eta$. Then  $\eta$ can be written as
	\begin{equation*}
	\eta_s=\theta_s+Z_s,\quad0\leq s\leq t,
	\end{equation*}
    where $\Phi(s,U_0)=(\phi_s,\theta_s)$ is the unforced solution of equation \eqref{eq NSunforced} and $\theta_s=e^{s\Delta_\omega}\eta_0$.

    We next compare the velocity component $u$ with $\phi$. Set $v=u-\phi$, which satisfies
	\begin{equation*}
	\begin{cases}
	\partial_sv-\nu\Delta v+(u\cdot\nabla)v+(v\cdot\nabla)\phi+\nabla\pi=\overline{Z},\quad x\in D,\;s\in(0,t),\\
	\dvg v=0,\\
	v\cdot n=0,\quad\curl v=0,\quad x\in\partial D,\\
	v(0)=0.
	\end{cases}
	\end{equation*}
    Put $r=\curl v$ and $\varpi=\curl\phi$. Taking curl in the equation for $v$, we have
	\begin{equation*}
	\partial_sr-\nu\Delta r+u\cdot\nabla r+v\cdot\nabla\varpi=\curl\overline{Z},\quad r|_{\partial D}=0,\quad r(0)=0.
	\end{equation*}
    Multiplying by $r$ and using $u\cdot n=0$ on $\partial D$, we obtain
	\begin{align*}
	\frac{1}{2}\partial_s\|r\|^2+\nu\|\nabla r\|^2&=-\int_D(v\cdot\nabla\varpi)rdx+\langle\curl\overline{Z},r\rangle.
	\end{align*}
    The div-curl estimate and the Gagliardo--Nirenberg inequality give
	\begin{align*}
	\left|\int_D(v\cdot\nabla\varpi)rdx\right|&\leq C\|v\|_{L^4(D)}\|\nabla\varpi\|\|r\|_{L^4(D)}\leq C\|\nabla\varpi\|\|r\|^{3/2}\|\nabla r\|^{1/2}\\
	&\leq\frac{\nu}{4}\|\nabla r\|^2+C(1+\|\nabla\varpi\|^2)\|r\|^2,
	\end{align*}
    while
	\begin{equation*}
	|\langle\curl\overline{Z},r\rangle|\leq\left|\int_D\overline{Z}\cdot\nabla^\perp rdx\right|\leq\frac{\nu}{4}\|\nabla r\|^2+C\|Z\|^2.
	\end{equation*}
    Hence
	\begin{equation*}
	\partial_s\|r\|^2\leq C(1+\|\nabla\varpi\|^2)\|r\|^2+C\|Z\|^2.
	\end{equation*}
    Moreover, integrating the vorticity energy inequality in the proof of Lemma \ref{lemma stable} gives
	\begin{equation*}
	\int_0^t\|\nabla\varpi_s\|^2ds\leq C\left(\|\curl u_0\|^2+\int_0^t\|\theta_s\|^2ds\right)\leq C_{t,R}\quad\forall\,U_0\in B_{\mathcal X}(0,R).
	\end{equation*}
    Therefore, by Gronwall's inequality, we derive
	\begin{equation}\label{eq Z}
	\|u_t-\phi_t\|_{\mathcal H^1}^2\leq C\|r_t\|^2\leq C_{t,R}\int_0^t\|Z_s\|^2ds.
	\end{equation}

    Combining \eqref{eq stable} and \eqref{eq Z}, we conclude that for any $U_0\in B_{\mathcal X}(0,R)$
	\begin{align*}
	\Pb\left(\|U(t,U_0)\|_{\mathcal X}\leq\varepsilon\right)&\geq\Pb\left(\|U(t,U_0)-\Phi(t,U_0)\|_{\mathcal X}\leq\varepsilon/2,\quad\|\Phi(t,U_0)\|_{\mathcal X}\leq\varepsilon/2\right)\\
	&=\Pb\left(\|u_t-\phi_t\|_{\mathcal H^1}^2+\|Z_t\|^2\leq\varepsilon^2/4\right)\\
	&\geq\Pb\left(C_{t,R}\int_0^t\|Z_s\|^2ds+\|Z_t\|^2\leq\varepsilon^2/4\right)\\
	&:=\Pb(\Omega_{t,R,\varepsilon})>0.
	\end{align*}

    Here in the last line, we use the fact that $Z$ is a centered Gaussian random variable in the Hilbert space $L^2(0,t;L^2(\omega;\R^2))\times L^2(\omega;\R^2)$.  Thus the standard support property of Gaussian measures implies that the positivity of the probability, which is uniform on $U_0\in B_{\mathcal X}(0,R)$.  This completes the proof of \eqref{eq irreducibility}, i.e. Proposition \ref{prop irreducibility}.
    \end{proof}

    \subsection{Proof of Theorem \ref{thm 2}}

    Combining the preceding estimates, we now apply Theorem~\ref{thm HM} to the coupled Markov process. In an equivalent Hilbert norm on $\mathcal X$, the Lyapunov estimate and Propositions~\ref{prop ASF} and \ref{prop irreducibility} verify the three hypotheses of the criterion, yielding the unique invariant measure and exponential mixing asserted in Theorem~\ref{thm 2}.

    \begin{proof}
    By Section \ref{Sec 2.1}, the coupled system \eqref{eq NS},\eqref{eq xi} defines a Feller Markov process with $C^1$ dependence on the initial condition. We verify the hypotheses of Theorem~\ref{thm HM} in an equivalent Hilbert space $(\widehat{\mathcal X},\|\cdot\|_{\widehat{\mathcal X}})$ adapted to the coupled energy estimate. 
    
    Fix $B_0>0$, and let $K>0$ be the constant chosen in Lemma~\ref{lemma L bdd}. Define
	\begin{equation*}
	\|U\|_{\widehat{\mathcal X}}^2:=\|\curl u\|^2+K\|\eta\|^2,\quad\mathcal Q(U):=\|\nabla\curl u\|^2+K\|\nabla\eta\|^2.
	\end{equation*}
    By the div-curl estimate and the boundary conditions, $\|\cdot\|_{\widehat{\mathcal X}}$ is equivalent to the original phase-space norm: there exists a constant $C_0\geq1$ such that
	\begin{equation}\label{eq Xhat}
	C_0^{-1}\|U\|_{\mathcal X}^2\leq\|U\|_{\widehat{\mathcal X}}^2\leq C_0\|U\|_{\mathcal X}^2,\quad U\in \mathcal X.
	\end{equation}
    According to Theorem \ref{thm HM}, it suffices to verify conditions \eqref{eq HM1}-\eqref{eq HM3} with $X=\widehat{\mathcal X}$.

    \vspace{0.3em}
    \begin{itemize}[leftmargin=2em]
        \item[(1)] {\it Lyapunov estimate.} By the proof of Lemma \ref{lemma L bdd}(1), there exist constants $c_0,C_1,C>0$ such that
	\begin{equation*}
	d\|U_t\|_{\widehat{\mathcal X}}^2+c_0\|U_t\|_{\widehat{\mathcal X}}^2dt+c_0\mathcal Q(U_t)dt\leq Cdt+dM_t,\quad d\langle M\rangle_t\leq C_1\|U_t\|_{\widehat{\mathcal X}}^2dt.
	\end{equation*}
        We then take $\gamma_*=\gamma_*(c_0,C_1)=\gamma_*(B_0)>0$ satisfying $3\gamma_*C_1=c_0/2$. Then, for $c=c_0/2$ and any $\gamma\in(0,\gamma_*]$, $r\in[1/4,3]$,
	\begin{equation*}
	d\|U_t\|_{\widehat{\mathcal X}}^2+c\|U_t\|_{\widehat{\mathcal X}}^2dt+c\mathcal Q(U_t)dt\leq Cdt+dM_t-r\gamma d\langle M\rangle_t.
	\end{equation*}
        Hence for any $0\leq t\leq1$, one has
	\begin{equation*}
	\|U_t\|_{\widehat{\mathcal X}}^2-e^{-ct}\|U_0\|_{\widehat{\mathcal X}}^2+c\int_0^te^{-c(t-s)}\mathcal Q(U_s)ds-C\leq\int_0^te^{-c(t-s)}dM_s-r\gamma\int_0^te^{-c(t-s)}d\langle M\rangle_s.
	\end{equation*}
        
        Consequently, using  \cite[Lemma A.1]{Mattingly-02b} and arguments as in the proof of Lemma \ref{lemma L bdd}(1), we derive that
	\begin{equation*}
	\E\exp\left(r\gamma\|U_t\|_{\widehat{\mathcal X}}^2+r\gamma c\int_0^te^{-c(t-s)}\mathcal Q(U_s)ds\right)\leq C\exp\left(r\gamma e^{-ct}\|U_0\|_{\widehat{\mathcal X}}^2\right).
	\end{equation*}
        On the other hand, by \eqref{eq J1} and \eqref{eq Xhat}, for any $\kappa>0$, there exists a constant $C_\kappa>0$ such that
	\begin{equation*}
	\|\mathcal J_{0,t}\|_{\mathcal L(\widehat{\mathcal X})}\leq C_\kappa\exp\left(C\kappa\int_0^t\mathcal Q(U_s)ds\right),\quad0\leq t\leq1.
	\end{equation*}
        Since $\int_0^t\mathcal Q(U_s)ds\leq e^c\int_0^te^{-c(t-s)}\mathcal Q(U_s)ds$, we then choose $\kappa=\kappa(\gamma)>0$ such that $C\kappa e^c=\gamma c/4$. Then for every $r\in[1/4,3]$,
	\begin{align*}
	\E\exp\left(r\gamma\|U_t\|_{\widehat{\mathcal X}}^2\right)\|\mathcal J_{0,t}\|_{\mathcal L(\widehat{\mathcal X})}&\leq C\E\exp\left(r\gamma\|U_t\|_{\widehat{\mathcal X}}^2+r\gamma c\int_0^te^{-c(t-s)}\mathcal Q(U_s)ds\right)\\
	&\leq C\exp\left(r\gamma e^{-ct}\|U_0\|_{\widehat{\mathcal X}}^2\right).
	\end{align*}
         Therefore, combining this with the preceding Lyapunov estimate, we conclude that, for the same $\gamma_*$ and $c$, and for any $\gamma\in(0,\gamma_*]$,  
	\begin{equation*}
	\E\exp\left(r\gamma\|U_t\|_{\widehat{\mathcal X}}^2\right)\left(1+\|\mathcal J_{0,t}\|_{\mathcal L(\widehat{\mathcal X})}\right)\leq C_\gamma\exp\left(r\gamma e^{-ct}\|U_0\|_{\widehat{\mathcal X}}^2\right),
	\end{equation*}
        for any $U_0\in \mathcal X$, $r\in[1/4,3]$ and $t\in[0,1]$. This verifies \eqref{eq HM1} on $\widehat{\mathcal X}$ with $\xi(t)=e^{-ct}$.

        \vspace{0.3em}

        \item[(2)] {\it Asymptotic strong Feller property.} Let $\gamma\in(0,\frac{1}{2}\min\{\gamma_*,C_0^{-1}\}]$. Invoking Proposition \ref{prop ASF} with this $\gamma$ and $\alpha=1$, there exists $N=N(B_0,\gamma)\in\N^+$ such that, if
	\begin{equation*}
	\sum_{j\in\N^+}b_j^2\leq B_0,\quad b_j\neq0\quad\text{for }1\leq j\leq N,
	\end{equation*}
        then, using \eqref{eq Xhat} and the equivalence of $\|\cdot\|_{\mathcal X}$ and $\|\cdot\|_{\widehat{\mathcal X}}$, for every Fr\'echet differentiable $\varphi\colon \mathcal X\to\R$, $U_0\in \mathcal X$ and $t\geq0$,
	\begin{align*}
	\|\nabla P_t\varphi(U_0)\|_{\widehat{\mathcal X}}\leq C\exp\left(\tfrac{\gamma}{2}\|U_0\|_{\widehat{\mathcal X}}^2\right)\left(\sqrt{P_t(|\varphi|^2)(U_0)}+e^{-t}\sqrt{P_t(\|\nabla\varphi\|_{\widehat{\mathcal X}}^2)(U_0)}\right),
	\end{align*}
        which verifies the asymptotic strong Feller property \eqref{eq HM2} on $\widehat{\mathcal X}$.

        \vspace{0.3em}

        \item[(3)] {\it Irreducibility.} Let $\rho_{r,\widehat{\mathcal X}}$ be the distance appearing in \eqref{eq HM3} with $X=\widehat{\mathcal X}$. For any $\varepsilon,R>0$, $\gamma\in(0,\gamma_*]$ and $r\in(0,1)$, take $\delta=\delta(\varepsilon,r,\gamma,C_0)>0$ such that for  any $V,\hat V\in B_{\mathcal X}(0,\delta)$,
	\begin{equation*}
	\rho_{r,\widehat{\mathcal X}}(V,\hat V)\leq2C_0^{1/2}\delta\exp\left(r\gamma C_0\delta^2\right)\leq\varepsilon.
	\end{equation*}
        Moreover, note that $B_{\widehat{\mathcal X}}(0,R)\subset B_{\mathcal X}(0,C_0^{1/2}R)$. For any $U_0,\hat U_0\in B_{\widehat{\mathcal X}}(0,R)$ and $t\geq0$, let $\hat{\Gamma}_t\in\mathscr{C}(P_t^*\delta_{U_0},P_t^*\delta_{\hat U_0})$ be the independent coupling
	\begin{equation*}
	\hat{\Gamma}_t(A_1\times A_2)=P_t(U_0,A_1)P_t(\hat U_0,A_2),\quad A_1,A_2\in\mathcal{B}(\mathcal X).
	\end{equation*}
        Then we compute that, for any $\varepsilon,R>0$, $\gamma\in(0,\gamma_*]$ and $r\in(0,1)$,
	\begin{align*}
	&\inf_{U_0,\hat U_0\in B_{\widehat{\mathcal X}}(0,R)}\sup_{\Gamma\in\mathscr{C}(P_t^*\delta_{U_0},P_t^*\delta_{\hat U_0})}\Gamma\left\{(V,\hat V)\in \mathcal X\times \mathcal X:\rho_{r,\widehat{\mathcal X}}(V,\hat V)\leq\varepsilon\right\}\\
	&\qquad\geq\inf_{U_0,\hat U_0\in B_{\widehat{\mathcal X}}(0,R)}\hat{\Gamma}_t\left\{(V,\hat V)\in \mathcal X\times \mathcal X:\rho_{r,\widehat{\mathcal X}}(V,\hat V)\leq\varepsilon\right\}\\
	&\qquad\geq\inf_{U_0,\hat U_0\in B_{\widehat{\mathcal X}}(0,R)}\hat{\Gamma}_t\left(B_{\mathcal X}(0,\delta)\times B_{\mathcal X}(0,\delta)\right)\\
	&\qquad\geq\left(\inf_{U_0\in B_{\mathcal X}(0,C_0^{1/2}R)}\Pb\left(U(t,U_0)\in B_{\mathcal X}(0,\delta)\right)\right)^2>0,
	\end{align*}
        provided that $t\geq T=T(\delta,C_0^{1/2}R)$, where $T$ is given by Proposition \ref{prop irreducibility}. 
        This ensures the irreducibility \eqref{eq HM3} on $\widehat{\mathcal X}$.
    \end{itemize}

    \vspace{0.3em}

    Collecting these results with $\gamma=\frac{1}{2}\min\{\gamma_*,C_0^{-1}\}$ and applying Theorem \ref{thm HM} on $\widehat{\mathcal X}$, we conclude that for the above $N=N(B_0,\gamma)=N(B_0)$, if
	\begin{equation*}
	\sum_{j\in\N^+}b_j^2\leq B_0\quad\text{and}\quad b_j\neq0\quad\text{for }1\leq j\leq N,
	\end{equation*}
    then, the Markov process on the original phase space $\mathcal X$ admits a unique invariant measure $\mu_*$.
    
    Additionally, using Theorem \ref{thm HM}, it follows that
	\begin{equation*}
	\left|P_tf(U_0)-\int_Xfd\mu_*\right|\leq Ce^{-\alpha t}\exp\left(\gamma\|U_0\|_{\widehat{\mathcal X}}^2\right)\|f\|_{\gamma,\widehat{\mathcal X}},
	\end{equation*}
    where
	\begin{equation*}
	\|f\|_{\gamma,\widehat{\mathcal X}}:=\sup_{U\in \mathcal X}\exp\left(-\gamma\|U\|_{\widehat{\mathcal X}}^2\right)\left(|f(U)|+\|\nabla f(U)\|_{\widehat{\mathcal X}}\right).
	\end{equation*}
    Finally, by \eqref{eq Xhat}, with $\bar{\gamma}=C_0^{-1}\gamma$,
	\begin{equation*}
	\exp\left(\gamma\|U_0\|_{\widehat{\mathcal X}}^2\right)\leq\exp\left(C_0\bar{\gamma}\|U_0\|_{\mathcal X}^2\right),\quad\|f\|_{\gamma,\widehat{\mathcal X}}\leq C\|f\|_{\bar{\gamma}}.
	\end{equation*}
    Relabeling $\bar{\gamma}$ as $\gamma$ gives \eqref{eq mixing2}.
    This completes the proof of Theorem \ref{thm 2}.
    \end{proof}

     \subsection{Proof of Theorem \ref{thm 1}} 

     We finally pass from the coupled Markov system to the Navier--Stokes equation driven by the stationary forcing. Starting the heat component from its invariant Gaussian law gives the same forcing path law as $\eta^s$, so averaging the mixing estimate of Theorem~\ref{thm 2} over $\mu_s$ yields exponential convergence to the velocity marginal $\mu=(\Pi_n)_\#\mu_*$. The disintegration of $\mu_*$ then provides the jointly stationary solution required in Theorem~\ref{thm 1}.

    \begin{proof}
    By Theorem \ref{thm 2}, there exists a unique invariant measure $\mu_*$ for the coupled system $(u,\eta)$ of equation \eqref{eq NS},\eqref{eq eta} and constant $\gamma,\alpha,C>0$ such that \eqref{eq mixing2} holds. We then define
	\begin{equation*}
	\mu:=(\Pi_u)_\#\mu_*\in\mathcal{P}(\mathcal H^1),
	\end{equation*}
    where $\Pi_u$ denotes the projection form $\mathcal X$ to $\mathcal H^1$. Meanwhile, let us define 
	\begin{equation*}
	\eta_0^{s}:=\sum_{j\in\N^+}b_j\int_{-\infty}^{0}e^{\lambda_jr}d\hat\beta_j(r)\psi_j,
	\end{equation*}
    which satisfies that $\law(\eta_0^{s})=\mu_s$.  Moreover, $\eta_0^{s}$ depends only on the Brownian paths before time $0$ and is therefore independent of the future increments
    $\{\hat\beta_j(t)-\hat\beta_j(0):t\geq0\}$. In view of the exponential mixing \eqref{eq mixing2} and \eqref{eq mu_*}, one has
	\begin{equation*}
	M:=\int_{L^2(\omega;\R^2)}\exp\left(C_0\gamma\|\eta\|^2\right)\mu_s(d\eta)<\infty,
	\end{equation*}
    where $C_0>0$ is the constant given by \eqref{eq Xhat}.
    
    Let $\eta(t,\eta_0^{s})$ be the solution of the heat equation \eqref{eq eta} driven by these future increments. Splitting the stochastic integral in \eqref{eq eta_s} at time $0$ gives, for any $t\geq0$, one has
	\begin{align*}
	\eta(t,\eta_0^{s})=e^{t\Delta_\omega}\eta_0^{s}+\sum_{j\in\N^+}b_j\int_0^te^{-\lambda_j(t-r)}d\hat\beta_j(r)\psi_j=\eta^{s}(t).
	\end{align*}
    Consequently, the velocity $u(t,u_0)$ in Theorem~\ref{thm 1} has the same law as the velocity component of the coupled solution starting from $(u_0,\eta_0^{s})$.

   For any $\varphi\in\mathcal O_{\mathcal H^1,\gamma}$, by setting $F(u,\eta)=\varphi(u)$, we then have $F\in\mathcal O_{\mathcal X,\gamma}$ and
	\begin{equation*}
	\int_XFd\mu_*=\int_{\mathcal H^1}\varphi d\mu,\quad\|F\|_{\mathcal X,\gamma}\leq\|\varphi\|_{\mathcal H^1,\gamma}.
	\end{equation*}
    
    Therefore, invoking Theorem~\ref{thm 2}, we derive that
	\begin{align*}
	\left|\E\varphi(u(t,u_0))-\int_{\mathcal H^1}\varphi d\mu\right|&=\left|\int_{L^2(\omega;\R^2)}P_tF(u_0,\eta)\,\mu_s(d\eta)-\int_XFd\mu_*\right|\\
	&\leq\int_{L^2(\omega;\R^2)}\left|P_tF(u_0,\eta)-\int_XFd\mu_*\right|\mu_s(d\eta)\\
	&\leq Ce^{-\alpha t}\exp\left(C_0\gamma\|u_0\|_{\mathcal H^1}^2\right)\|\varphi\|_{\mathcal H^1,\gamma}\int_{L^2(\omega;\R^2)}\exp\left(C_0\gamma\|\eta\|^2\right)\mu_s(d\eta)\\
	&\leq CMe^{-\alpha t}\exp\left(C_0\gamma\|u_0\|_{\mathcal H^1}^2\right)\|\varphi\|_{\mathcal H^1,\gamma},
	\end{align*}
    which verifies \eqref{eq mixing1}.

    \vspace{0.3em}
    It remains to verify \eqref{eq stationary}. By Remark~\ref{Rem 1}, one has $(\Pi_\eta)_\#\mu_*=\mu_s$. Moreover, using the disintegration \eqref{eq mu_*}, on an enlargement of the two-sided Brownian probability space, we choose $u^s(0)$ conditionally on $\eta^s(0)$ according to $p_{\eta^s(0)}$, with the auxiliary randomization independent of $\{\hat\beta_j\}_{j\in\N^+}$. Then
	\begin{equation*}
	\law\left(u^s(0),\eta^s(0)\right)=\mu_*,
	\end{equation*}
    and this initial pair is independent of the future Brownian increments. The corresponding coupled solution therefore starts from its invariant law and is stationary for nonnegative times.

    By taking a standard two-sided stationary extension, we derive a jointly stationary process $(u^s,\widetilde\eta^s)$ satisfying
	\begin{equation*}
	\law\left(u^s(t),\widetilde\eta^s(t)\right)=\mu_*,\quad t\in\R.
	\end{equation*}
    The heat component is autonomous and has the unique stationary path law associated with \eqref{eq eta}. After passing to an equivalent realization, it may therefore be identified with the process $\eta^s$ defined in \eqref{eq eta_s}. Consequently,
	\begin{equation*}
	\law\left(u^s(t)\right)=(\Pi_u)_\#\mu_*=\mu,\quad t\in\R,
	\end{equation*}
    which proves \eqref{eq stationary}.  This completes the proof of Theorem \ref{thm 1}.
    \end{proof}

\section*{Appendix}

    \stepcounter{section}
    
    \numberwithin{equation}{subsection}
    \numberwithin{theorem}{subsection}
    
    \setcounter{equation}{0}
    \setcounter{theorem}{0}

    In the appendix, we summarize some useful supplementary materials and technical proofs.
    
    \renewcommand\thesubsection{\Alph{subsection}}

    \subsection{Deterministic Navier--Stokes estimates}\label{Sec A}

    This appendix collects deterministic estimates for the 2D Navier--Stokes equation. We first consider the equation with a non-autonomous deterministic forcing $f\in L^2(0,T;H^1(D;\R^2))$:
	\begin{equation}\label{eq NS det}
	\begin{cases}
	\partial_tu-\nu\Delta u+(u\cdot\nabla)u+\nabla\pi=f(t,x),\quad x\in D,\;t\in(0,T),\\
	\dvg u=0,\\
	u\cdot n=0,\quad\curl u=0,\quad x\in\partial D,\\
	u(0,\cdot)=u_0(\cdot)\in \mathcal H^1.
	\end{cases}
	\end{equation}

     \begin{lemma}\label{NS estimate} For every $T>0$, there exists a constant $C>0$ such that for any $u_0\in \mathcal H^1$ and $f\in L^2(0,T;H^{1}(D;\R^2))$, the solution $u$ of \eqref{eq NS det} satisfies
	\begin{align}
	\|u\|_{L^2(0,T;\mathcal H^2)}&\leq C\left(\|u_0\|_{\mathcal H^1}+\|f\|_{L^2(0,T;H^{1}(D))}\right).\label{eq u1}\\
	\|u\|_{L^\infty((T/2,T)\times D)}&\leq C\left(\|u_0\|_{\mathcal H^1}+\|f\|_{L^2(0,T;H^{1}(D))}\right),\label{eq u2}\\
	\|\partial_tu\|_{L^2(T/2,T;L^4(D))}&\leq C\left(1+\|u_0\|_{\mathcal H^1}^2+\|f\|_{L^2(0,T;H^{1}(D))}^2\right).\label{eq u3}
	\end{align}
    \end{lemma}

\begin{proof}
The constant $C$ below may vary from line to line, but depends only on $T,\nu$ and $D$. Set
	\begin{equation*}
	M:=\|u_0\|_{\mathcal H^1}+\|f\|_{L^2(0,T;H^1(D))}.
	\end{equation*}

For any $1<q<\infty$ and any divergence-free vector field $v$ satisfying $v\cdot n=0$ on $\partial D$, the stream-function representation and the simple connectedness of $D$ imply that
	\begin{equation}\label{eq d-c}
	\|v\|_{W^{1,q}(D)}\leq C\|\curl v\|_{L^q(D)}.
	\end{equation}
Indeed, one may write $v=\nabla^\perp\psi$ with $\psi|_{\partial D}=0$ and $-\Delta\psi=\curl v$, and then apply the elliptic estimate for the Dirichlet Laplacian.

\vspace{0.3em}

Let $\varpi=\curl u$. Taking curl in \eqref{eq NS det}, we compute that
	\begin{equation}\label{eq varpi}
	\partial_t\varpi+u\cdot\nabla\varpi-\nu\Delta\varpi=\curl f,\quad\varpi|_{\partial D}=0,\quad\varpi(0)=\curl u_0.
	\end{equation}
Multiplying \eqref{eq varpi} by $\varpi$, we obtain
	\begin{equation*}
	\frac{1}{2}\partial_t\|\varpi\|^2+\nu\|\nabla\varpi\|^2=\langle\curl f,\varpi\rangle\leq\frac{\nu}{2}\|\nabla\varpi\|^2+C\|f\|_{H^1(D)}^2.
	\end{equation*}
Since $\|\curl u_0\|\leq C\|u_0\|_{\mathcal H^1}$, by integration, we have
	\begin{equation}\label{eq varpi2}
	\|\varpi\|_{L^\infty(0,T;L^2(D))}+\|\nabla\varpi\|_{L^2(0,T;L^2(D))}\leq CM.
	\end{equation}
Using \eqref{eq d-c} with $q=2$ and the elliptic estimate for the stream function, the desired inequality \eqref{eq u1} follows from \eqref{eq varpi2} by
	\begin{equation*}
	\|u\|_{L^\infty(0,T;\mathcal H^1)}+\|u\|_{L^2(0,T;\mathcal H^2)}\leq CM.
	\end{equation*}

\vspace{0.3em}

We next prove the $L^\infty$ estimate \eqref{eq u2}. From \eqref{eq varpi2}, there exists $r_0\in(T/4,3T/8)$ such that
	\begin{equation*}
	\|\varpi_{r_0}\|^2+\|\nabla\varpi_{r_0}\|^2\leq CM^2.
	\end{equation*}
Thus the Ladyzhenskaya implies $\|\varpi_{r_0}\|_{L^4(D)}\leq CM$. Multiplying \eqref{eq varpi} by $|\varpi|^2\varpi$, one has
	\begin{equation*}
	\frac14\partial_t\|\varpi\|_{L^4(D)}^4+\frac{3\nu}{4}\|\nabla(\varpi^2)\|^2=\int_D\curl f\,|\varpi|^2\varpi dx.
	\end{equation*}
The 2D Gagliardo--Nirenberg inequality yields
	\begin{equation*}
	\|\varpi\|_{L^6(D)}^3=\|\varpi^2\|_{L^3(D)}^{3/2}\leq C\|\varpi^2\|^{1/2}\|\nabla(\varpi^2)\|=C\|\varpi\|_{L^4(D)}\|\nabla(\varpi^2)\|.
	\end{equation*}
Hence
	\begin{align*}
	\left|\int_D\curl f\,|\varpi|^2\varpi dx\right|\leq C\|\curl f\|\|\varpi\|_{L^4(D)}\|\nabla(\varpi^2)\|\leq\frac{3\nu}{8}\|\nabla(\varpi^2)\|^2+C\|\curl f\|^2\|\varpi\|_{L^4(D)}^2.
	\end{align*}
This implies that
	\begin{equation*}
	\partial_t\|\varpi\|_{L^4(D)}^4=2\|\varpi\|_{L^4(D)}^2\partial_t\|\varpi\|_{L^4(D)}^2\leq C\|\curl f\|^2\|\varpi\|_{L^4(D)}^2.
	\end{equation*}
Thus we have
	\begin{equation}\label{eq omegaL4}
	\|\varpi\|_{L^\infty(3T/8,T;L^4(D))}\leq CM.
	\end{equation}
Combining \eqref{eq d-c}, \eqref{eq omegaL4} and the Sobolev embedding $W^{1,4}(D)\hookrightarrow L^\infty(D)$ gives
	\begin{equation}\label{eq u4}
	\|u\|_{L^\infty(3T/8,T;W^{1,4}(D))}+\|u\|_{L^\infty((3T/8,T)\times D)}\leq CM,
	\end{equation}
which proves \eqref{eq u2}.

\vspace{0.3em}

It remains to estimate $\partial_tu$ in $L^2(T/2,T;L^4(D))$. Let $P_4$ be the Helmholtz projection on $L^4(D;\R^2)$ and let $A_4$ be the corresponding slip Stokes operator. We use the $L^p$-$L^q$ maximal for the Stokes operator with
Navier-type boundary conditions; see e.g. \cite[Theorem 7.17]{AAE-17}, which gives
	\begin{equation}\label{eq v11}
	\|\partial_tv\|_{L^2(0,T;L^4(D))}+\|A_4v\|_{L^2(0,T;L^4(D))}\leq C\|G\|_{L^2(0,T;L^4(D))}
	\end{equation}
for the equation $\partial_tv+\nu A_4v=P_4G$ with $v(0)=0$.

Choose $\chi\in C^\infty([0,T])$ satisfying $\chi=0$ on $[0,3T/8]$ and $\chi=1$ on $[T/2,T]$. Set $v=\chi u$. Then
	\begin{equation*}
	\partial_tv+\nu A_4v=P_4\left(\chi f-\chi(u\cdot\nabla)u+\chi'u\right),\quad v(0)=0.
	\end{equation*}
By \eqref{eq u4},
	\begin{align*}
	\|(u\cdot\nabla)u\|_{L^2(3T/8,T;L^4(D))}&\leq\|u\|_{L^\infty((3T/8,T)\times D)}\|\nabla u\|_{L^2(3T/8,T;L^4(D))}\leq CM^2.
	\end{align*}
Moreover,
	\begin{equation*}
	\|f\|_{L^2(0,T;L^4(D))}\leq CM,\quad\|\chi'u\|_{L^2(0,T;L^4(D))}\leq CM.
	\end{equation*}
Applying \eqref{eq v11}, we obtain
	\begin{equation*}
	\|\partial_tu\|_{L^2(T/2,T;L^4(D))}\leq C(M+M^2)\leq C(1+M^2).
	\end{equation*}
This proves inequality \eqref{eq u3} and completes the proof of Lemma \ref{NS estimate}.
\end{proof}

\vspace{0.6em}

We also formulate the deterministic stability. Consider the unforced  couple system given by
	\begin{equation}\label{eq NSunforced}
	\begin{cases}
	\partial_t\phi-\nu\Delta\phi+(\phi\cdot\nabla)\phi+\nabla\pi=\overline{\theta}(t,x),\quad x\in D,\;t>0,\\
	\dvg\phi=0,\\
	\phi\cdot n=0,\quad\curl\phi=0,\quad x\in\partial D,\\
	\partial_t\theta-\Delta_\omega\theta=0,\quad x\in\omega,\\
	\theta|_{\partial\omega}=0,\\
	\phi(0,\cdot)=u_0(\cdot),\quad\theta(0,\cdot)=\eta_0(\cdot).
	\end{cases}
	\end{equation}
The solution of \eqref{eq NSunforced} is denoted by
	\begin{equation*}
	\Phi(t,U_0)=(\phi(t),\theta(t)),\quad U_0=(u_0,\eta_0)\in \mathcal X.
	\end{equation*}

\begin{lemma}\label{lemma stable}
There exist constants $C,c>0$ such that for any $U_0\in \mathcal X$ and $t\geq0$, the solution $\Phi(t,U_0)$ of \eqref{eq NSunforced} satisfies
	\begin{equation}\label{eq global}
	\|\Phi(t,U_0)\|_{\mathcal X}^2\leq Ce^{-ct}\|U_0\|_{\mathcal X}^2.
	\end{equation}
\end{lemma}

\begin{proof}
The constant $C$ below may vary from line to line, but is independent of $U_0$ and $t$. Taking the $L^2(\omega)$ inner product of the heat equation in \eqref{eq NSunforced} with $\theta$, we obtain
	\begin{equation*}
	\frac{1}{2}\partial_t\|\theta\|^2+\|\nabla\theta\|^2=0.
	\end{equation*}
Using the Poincare inequality on $\omega$, there exists $\lambda_\omega>0$ such that
	\begin{equation}\label{eq theta}
	\partial_t\|\theta\|^2+2\lambda_\omega\|\theta\|^2\leq0,\quad\|\theta_t\|^2\leq e^{-2\lambda_\omega t}\|\eta_0\|^2.
	\end{equation}

\vspace{0.3em}

We next estimate the velocity. Put $\varpi=\curl\phi$. Taking curl in the velocity equation gives
	\begin{equation*}
	\partial_t\varpi+\phi\cdot\nabla\varpi-\nu\Delta\varpi=\curl\overline{\theta},\quad\varpi|_{\partial D}=0,\quad\varpi(0)=\curl u_0.
	\end{equation*}
Multiplying by $\varpi$, we get
	\begin{align*}
	\frac{1}{2}\partial_t\|\varpi\|^2+\nu\|\nabla\varpi\|^2=\langle\curl\overline{\theta},\varpi\rangle\leq\frac{\nu}{2}\|\nabla\varpi\|^2+C\|\theta\|^2.
	\end{align*}
By the Poincare inequality for $\varpi\in H_0^1(D)$, there exist constants $C_1,c_1>0$ such that
	\begin{equation}\label{eq curl}
	\partial_t\|\varpi\|^2+c_1\|\varpi\|^2\leq C_1\|\theta\|^2.
	\end{equation}

\vspace{0.3em}

Finally we combine \eqref{eq theta} and \eqref{eq curl}. Let
	\begin{equation*}
	\mathcal{V}(t):=\|\varpi_t\|^2+K\|\theta_t\|^2,
	\end{equation*}
which is equivalent to the norm $\|\cdot\|_{\mathcal X}$, and the parameter $K\geq1$ is specified below. Multiplying \eqref{eq theta} by $K$ and adding it to \eqref{eq curl}, we have
	\begin{equation*}
	\mathcal{V}'(t)+c_1\|\varpi\|^2+(2K\lambda_\omega-C_1)\|\theta\|^2\leq0.
	\end{equation*}
We now take $K$ sufficiently large such that $2K\lambda_\omega-C_1>0$ and set $c=\min\{c_1,2\lambda_\omega-CK^{-1}\}$. It then follows that
	\begin{equation*}
	\mathcal{V}'(t)+c\mathcal{V}(t)\leq0,
	\end{equation*}
which implies the estimate \eqref{eq global}.
\end{proof}

    \subsection{Linearized Navier--Stokes estimates}

    This appendix collects two deterministic estimates for linearized Navier--Stokes equations. Let $T>0$, and let $g$ satisfy \eqref{eq g} with time interval replaced by $(0,T)$. 

    We first consider the backward linearized equation
	\begin{equation}\label{eq l1}
	\begin{cases}
	-\partial_t\psi-\nu\Delta\psi-(g\cdot\nabla)\psi+(\nabla g)^{\rm T}\psi+\nabla\rho=f(t,x),\quad x\in D,\;t\in(0,T),\\
	\dvg\psi=0,\\
	\psi\cdot n=0,\quad\curl\psi=0,\quad x\in\partial D,\\
	\psi(T,\cdot)=0.
	\end{cases}
	\end{equation}

    \begin{lemma}\label{lemma l1}
    For every $T>0$, there exists a constant $C>0$ such that for any $f\in L^2(0,T;L^2(D;\R^2))$, the solution $\psi$ of \eqref{eq l1} satisfies
	\begin{align*}
	\|\psi\|_{L^\infty(0,T;\mathcal H^1)}+\|\psi\|_{L^2(0,T;\mathcal H^2)}+\|\partial_t\psi\|_{L^2(0,T;\mathcal H)}\leq C\exp\left(C\|g\|_{L^\infty((0,T)\times D)}^2\right)\|f\|_{L^2(0,T;L^2(D))}.
	\end{align*}
    \end{lemma}
    \begin{proof} It suffices to consider the forward system 
	\begin{equation*}
	\begin{cases}
	\partial_t\psi-\nu\Delta\psi-(g\cdot\nabla)\psi+(\nabla g)^{\rm T}\psi+\nabla\rho=f(t,x),\quad x\in D,\;t\in(0,T),\\
	\dvg\psi=0,\\
	\psi\cdot n=0,\quad\curl\psi=0,\quad x\in\partial D,\\
	\psi(0)=0.
	\end{cases}
	\end{equation*}
    We use the identity
	\begin{equation*}
	-(g\cdot\nabla)\psi+(\nabla g)^{\rm T}\psi=-\left(\nabla\psi+(\nabla\psi)^{\rm T}\right)g+\nabla(g\cdot\psi).
	\end{equation*}
After absorbing the gradient into the pressure, the equation can be written as
	\begin{equation*}
	\partial_t\psi-\nu\Delta\psi+\nabla\rho=f+\left(\nabla\psi+(\nabla\psi)^{\rm T}\right)g.
	\end{equation*}
Taking the $L^2(D)$ inner product with $\partial_t\psi$, the pressure term vanishes, and the slip boundary condition gives
	\begin{align*}
	\|\partial_t\psi\|^2+\frac{\nu}{2}\partial_t\|\curl\psi\|^2&=\langle f,\partial_t\psi\rangle+\left\langle\left(\nabla\psi+(\nabla\psi)^{\rm T}\right)g,\partial_t\psi\right\rangle.
	\end{align*}
By the div-curl estimate, $\|\psi\|_{\mathcal H^1}\leq C\|\curl\psi\|$. Hence
	\begin{align*}
	\|\partial_t\psi\|^2+\frac{\nu}{2}\partial_t\|\curl\psi\|^2&\leq\|f\|\|\partial_t\psi\|+C\|g\|_{L^\infty((0,T)\times D)}\|\psi\|_{\mathcal H^1}\|\partial_t\psi\|\\
	&\leq\frac12\|\partial_t\psi\|^2+C\|f\|^2+C\|g\|_{L^\infty((0,T)\times D)}^2\|\psi\|_{\mathcal H^1}^2.
	\end{align*}
    Since $\psi(0)=0$, by Gronwall's inequality, we compute that 
	\begin{align*}
	\|\psi\|_{L^\infty(0,T;\mathcal H^1)}+\|\partial_t\psi\|_{L^2(0,T;\mathcal H)}\leq C\exp\left(C\|g\|_{L^\infty((0,T)\times D)}^2\right)\|f\|_{L^2(0,T;L^2(D))}.
	\end{align*}
    Finally, the slip Stokes elliptic estimate applied to the equation above gives
	\begin{align*}
	\|\psi\|_{L^2(0,T;\mathcal H^2)}&\leq C\|f\|_{L^2(0,T;L^2(D))}+C\|\partial_t\psi\|_{L^2(0,T;\mathcal H)}+C\|g\|_{L^\infty((0,T)\times D)}\|\psi\|_{L^2(0,T;\mathcal H^1)}\\
	&\leq C\exp\left(C\|g\|_{L^\infty((0,T)\times D)}^2\right)\|f\|_{L^2(0,T;L^2(D))}.
	\end{align*}
    
    Thus the proof of Lemma \ref{lemma l1} is completed by 
	\begin{align*}	\|\psi\|_{L^\infty(0,T;\mathcal H^1)}+\|\psi\|_{L^2(0,T;\mathcal H^2)}+\|\partial_t\psi\|_{L^2(0,T;\mathcal H)}\leq C\exp\left(C\|g\|_{L^\infty((0,T)\times D)}^2\right)\|f\|_{L^2(0,T;L^2(D))}.
	\end{align*}
    \end{proof}

\vspace{0.6em}

We also consider the forward linearized equation
	\begin{equation}\label{eq l2}
	\begin{cases}
	\partial_tv-\nu\Delta v+(g\cdot\nabla)v+(v\cdot\nabla)g+\nabla\pi=f(t,x),\quad x\in D,\;t\in(0,T),\\
	\dvg v=0,\\
	v\cdot n=0,\quad\curl v=0,\quad x\in\partial D,\\
	v(0,\cdot)=0.
	\end{cases}
	\end{equation}

\begin{lemma}\label{lemma l2}
For every $T>0$, there exists a constant $C>0$ such that for any $f\in L^2(0,T;L^2(D;\R^2))$, the solution $v$ of \eqref{eq l2} satisfies
	\begin{align*}
	\|v\|_{L^\infty(0,T;\mathcal H^1)}\leq C\exp\left(C\left(\|g\|_{L^\infty((0,T)\times D)}^2+\|\partial_tg\|_{L^2(0,T;L^4(D))}\right)\right)\|f\|_{L^2(0,T;L^2(D))}.
	\end{align*}
\end{lemma}
\begin{proof} The weak form of \eqref{eq l2} is given by
	\begin{align*}
	\langle\partial_tv,w\rangle+\nu\int_D\curl v\,\curl wdx+\int_D(g\cdot\nabla)v\cdot wdx-\int_D(v\cdot\nabla)w\cdot gdx=\langle f,w\rangle.
	\end{align*}
We define
	\begin{equation*}
	\mathcal E(t)=\frac{\nu}{2}\|\curl v_t\|^2-\int_D(v_t\cdot\nabla)v_t\cdot g_tdx+K_0\|v_t\|^2,
	\end{equation*}
where $K_0:=C_0\left(1+\|g\|_{L^\infty((0,T)\times D)}^2\right)$ with a sufficiently large constant $C_0>0$. Note that the mixed term satisfies
	\begin{align*}
	\left|\int_D(v\cdot\nabla)v\cdot gdx\right|\leq\|g\|_{L^\infty((0,T)\times D)}\|v\|\|\nabla v\|\leq\frac{\nu}{4}\|\curl v\|^2+C\|g\|_{L^\infty((0,T)\times D)}^2\|v\|^2,
	\end{align*}
Hence, by the choice of $K_0$, there exists a constant $C_1>0$ such that
	\begin{equation*}
	C_1^{-1}\left(\|\curl v_t\|^2+K_0\|v_t\|^2\right)\leq\mathcal E(t)\leq C_1\left(\|\curl v_t\|^2+K_0\|v_t\|^2\right).
	\end{equation*}
    Taking $w=\partial_tv$ in the weak form gives
	\begin{align*}
	\|\partial_tv\|^2+\frac{\nu}{2}\partial_t\|\curl v\|^2+\int_D(g\cdot\nabla)v\cdot\partial_tvdx-\int_D(v\cdot\nabla)\partial_tv\cdot gdx=\langle f,\partial_tv\rangle.
	\end{align*}
    Combining this with the definition of $\mathcal E$, it follows that
	\begin{align*}
	\mathcal E'(t)+\|\partial_tv\|^2&=\langle f,\partial_tv\rangle-\int_D(g\cdot\nabla)v\cdot\partial_tvdx-\int_D(\partial_tv\cdot\nabla)v\cdot gdx\\
	&\quad-\int_D(v\cdot\nabla)v\cdot\partial_tgdx+2K_0\langle v,\partial_tv\rangle.
	\end{align*}
    Using the 2D div-curl estimate and the Sobolev embedding $H^1(D)\hookrightarrow L^4(D)$, one has
	\begin{align*}	\left|\int_D(g\cdot\nabla)v\cdot\partial_tvdx\right|+\left|\int_D(\partial_tv\cdot\nabla)v\cdot gdx\right|&\leq C\|g\|_{L^\infty((0,T)\times D)}\|v\|_{\mathcal H^1}\|\partial_tv\|,\\
	\left|\int_D(v\cdot\nabla)v\cdot\partial_tgdx\right|&\leq C\|\partial_tg(t)\|_{L^4(D)}\|v\|_{\mathcal H^1}^2.
	\end{align*}
Moreover,
	\begin{equation*}
	|\langle f,\partial_tv\rangle|+2K_0|\langle v,\partial_tv\rangle|\leq\frac14\|\partial_tv\|^2+C\|f\|^2+CK_0^2\|v\|^2.
	\end{equation*}
    
    Using again the equivalence between $\mathcal E$ and $\|v\|_{\mathcal H^1}^2$, and the definition of $K_0$, we obtain
	\begin{equation*}
	\mathcal E'(t)+\frac12\|\partial_tv\|^2\leq C\left(1+\|g\|_{L^\infty((0,T)\times D)}^2+\|\partial_tg(t)\|_{L^4(D)}\right)\mathcal E(t)+C\|f\|^2.
	\end{equation*}
    Gronwall's inequality then implies that
    \begin{align*}
	\sup_{0\leq t\leq T}\mathcal E(t)\leq C\exp\left(C\|g\|_{L^\infty((0,T)\times D)}^2+C\|\partial_tg\|_{L^1(0,T;L^4(D))}\right)\|f\|_{L^2(0,T;L^2(D))}^2.
	\end{align*}

    Combining this, the lower bound for $\mathcal E(t)$, and $\|\partial_tg\|_{L^1(0,T;L^4(D))}\leq C\|\partial_tg\|_{L^2(0,T;L^4(D))}$, we complete the proof of Lemma \ref{lemma l2}.
    \end{proof}

    \subsection{Stochastic and Jacobian estimates}

    This appendix collects some estimates for the stochastic  equation and its Jacobian flow.

    \begin{lemma}\label{lemma L bdd}
    For the equation given by \eqref{eq NS},\eqref{eq eta}, assume that the sequence $\{b_j\}_{j\in\N^+}$ satisfies $\sum_{j\in\N^+}b_j^2<\infty$. Then there exist constants $c,C,\gamma_0>0$ such that the solution $U_t=U(t,U_0)=(u_t,\eta_t)$ satisfies the following estimates.  
    \begin{itemize}
        \item [(1)] For any $\gamma\in(0,\gamma_0]$, $U_0\in \mathcal X$ and $t\geq0$, 
	\begin{align}
	\E\exp\left(\gamma\|U_t\|_{\mathcal X}^2\right)\leq C\exp\left(C\gamma e^{-ct}\|U_0\|_{\mathcal X}^2\right).\label{eq L2 Ubdd}
	\end{align}
        \item [(2)] For any $\gamma\in(0,\gamma_0]$, $U_0\in \mathcal X$ and $n\in\N$, 
	\begin{align}
	\E\exp\left(\gamma\sum_{0\leq k\leq n}\|U_k\|_{\mathcal X}^2\right)\leq\exp\left(C\gamma\|U_0\|_{\mathcal X}^2\right)\exp\left(Cn\right).\label{eq L2 Ubdd3}
	\end{align}
        \item [(3)]  For any $\gamma\in(0,\gamma_0]$, $U_0\in \mathcal X$ and $n\in\N$, 
	\begin{align}
	\E\left(\exp\left(\gamma\|u\|_{L^\infty((n+1,n+2)\times D)}^2\right)|\mathcal{F}_{n}\right)&\leq C\exp\left(C\gamma\|U_{n}\|_{\mathcal X}^2\right),\label{eq L-inf bdd}\\
	\E\left(\exp\left(\gamma\|u\|_{L^2(n,n+2;\mathcal H^2)}^2\right)|\mathcal{F}_{n}\right)&\leq C\exp\left(C\gamma\|U_{n}\|_{\mathcal X}^2\right).\label{eq L2H2 bdd}
	\end{align}
        \item [(4)]  For any $m\in\N^+$, there exists a constant $C_m>0$ such that for any $U_0\in \mathcal X$ and $n\in\N$, 
	\begin{equation}\label{eq L bdd}
	\E\left(\|\partial_tu\|_{L^2(n+1,n+2;L^4(D))}^m|\mathcal{F}_{n}\right)\leq C_m(1+\|U_{n}\|_{\mathcal X}^{2m}).
	\end{equation}
    \end{itemize}    
\end{lemma}
\begin{proof}
For ease of notation, we use $C\geq1$ to denote generic constants that may vary from line to line below, which are independent of $U_0,t,n$. Set $B_0=\sum_{j\in\N^+}b_j^2$. The constant $\gamma_0>0$ will be chosen sufficiently small in the proof.

\vspace{0.6em}
    
\noindent {\it Point (1).} Applying the Ito formula to $\|\eta_t\|^2$, one has
	\begin{equation*}
	d\|\eta_t\|^2+2\|\nabla\eta_t\|^2dt=B_0dt+2\sum_{j\in\N^+}b_j\langle\eta_t,\psi_j\rangle d\beta_j(t).
	\end{equation*}
By the Poincare inequality on $\omega$, this gives
	\begin{equation*}
	d\|\eta_t\|^2\leq-c\|\eta_t\|^2dt+Cdt+2\sum_{j\in\N^+}b_j\langle\eta_t,\psi_j\rangle d\beta_j(t).
	\end{equation*}
On the other hand, let $\varpi_t=\curl u_t$. The vorticity equation reads
	\begin{equation*}
	\partial_t\varpi+u\cdot\nabla\varpi-\nu\Delta\varpi=\curl\overline{\eta},\quad\varpi|_{\partial D}=0,\quad\varpi_0=\curl u_0.
	\end{equation*}
Multiplying by $\varpi_t$, we obtain
	\begin{align*}
	\frac{1}{2}\partial_t\|\varpi_t\|^2+\nu\|\nabla\varpi_t\|^2=\langle\curl\overline{\eta}_t,\varpi_t\rangle\leq\frac{\nu}{2}\|\nabla\varpi_t\|^2+C\|\eta_t\|^2.
	\end{align*}
By the Poincare inequality for $\varpi_t\in H_0^1(D)$, this gives
	\begin{equation*}
	\partial_t\|\varpi_t\|^2\leq-c\|\varpi_t\|^2+C\|\eta_t\|^2.
	\end{equation*}
We now take $K\geq1$ sufficiently large and define functional
	\begin{equation*}
	\mathcal{V}(U_t):=\|\varpi_t\|^2+K\|\eta_t\|^2,
	\end{equation*}
which is equivalent to the $\|\cdot\|_{\mathcal X}$-norm. Combining the two inequalities above, there exists a constant $c_0>0$ such that
	\begin{equation*}
	d\mathcal{V}(U_t)\leq-c_0\mathcal{V}(U_t)dt+Cdt+dM_t,\quad M_t=2K\sum_{j\in\N^+}b_j\int_0^t\langle\eta_s,\psi_j\rangle d\beta_j(s).
	\end{equation*}
Moreover, there exists a constant $C_K>0$ such that
	\begin{equation*}
	d\langle M\rangle_t=4K^2\sum_{j\in\N^+}b_j^2\langle\eta_t,\psi_j\rangle^2dt\leq C_K\mathcal{V}(U_t)dt.
	\end{equation*}
Set $c=c_0/2$ and choose $\gamma_1>0$ so that
	\begin{equation*}
	\gamma_1\leq\frac{c_0}{2C_K}.
	\end{equation*}
Then, for every $\gamma\in(0,\gamma_1]$, we derive that 
	\begin{align*}
	\mathcal{V}(U_t)-e^{-ct}\mathcal{V}(U_0)-C\leq\int_0^te^{-c(t-s)}dM_s-\gamma\int_0^te^{-c(t-s)}d\langle M\rangle_s.
	\end{align*}
Thus, by \cite[Lemma A.1]{Mattingly-02b}, for any $R>0$,
	\begin{equation*}
	\Pb\left(\mathcal{V}(U_t)-e^{-ct}\mathcal{V}(U_0)-C\geq R\right)\leq e^{-2\gamma R}.
	\end{equation*}
In view of the fact that a random variable $X$ satisfies $\E X\leq2$ provided $\Pb(X\geq R)\leq R^{-2}$ for any $R\geq0$, this implies
	\begin{equation*}
	\E\exp\left(\gamma\mathcal{V}(U_t)\right)\leq C\exp\left(\gamma e^{-ct}\mathcal{V}(U_0)\right),
	\end{equation*}
which thus ensures \eqref{eq L2 Ubdd}.

\vspace{0.6em}
    
\noindent {\it Point (2).} By Point (1), for any $\gamma>0$ sufficiently small, there exists $C_1>0$ such that
	\begin{equation*}
	\E\left(\exp\left(\gamma\mathcal{V}(U_k)\right)|\mathcal{F}_{k-1}\right)\leq C_1\exp\left(\gamma e^{-c}\mathcal{V}(U_{k-1})\right).
	\end{equation*}
Therefore,
	\begin{align*}
	\E\exp\left(\gamma\sum_{0\leq k\leq n}\mathcal{V}(U_k)\right)&=\E\left(\exp\left(\gamma\sum_{0\leq k\leq n-1}\mathcal{V}(U_k)\right)\E\left(e^{\gamma\mathcal{V}(U_n)}|\mathcal{F}_{n-1}\right)\right)\\
	&\leq C_1\E\exp\left(\gamma\sum_{0\leq k\leq n-1}\mathcal{V}(U_k)+\gamma e^{-c}\mathcal{V}(U_{n-1})\right).
	\end{align*}
Applying this procedure repeatedly, one has
	\begin{align*}
	\E\exp\left(\gamma\sum_{0\leq k\leq n}\mathcal{V}(U_k)\right)\leq C_1^n\exp\left(c_1\gamma\mathcal{V}(U_0)\right),
	\end{align*}
where
	\begin{equation*}
	c_1:=\sum_{k\in\N}e^{-ck}=(1-e^{-c})^{-1},
	\end{equation*}
provided that $0<\gamma\leq\gamma_2:=c_1^{-1}\gamma_1$.  As $C^{-1}\|U\|_{\mathcal X}^2\leq\mathcal{V}(U)\leq C\|U\|_{\mathcal X}^2$, estimate \eqref{eq L2 Ubdd3} follows.

\vspace{0.6em}

Before proving Points (3),(4), we collect two estimates for the localized force on $[n,n+2]$. Since $\eta_t\in H_0^1(\omega;\R^2)$ and $\overline{\eta}_t$ is the zero extension of $\eta_t$, we have
	\begin{equation}\label{eq eta11}
	\|\overline{\eta}\|_{L^2(n,n+2;H^1(D))}=\|\eta\|_{L^2(n,n+2;H^1(\omega))}.
	\end{equation}
Using the Poincare inequality on $\omega$, we obtain
	\begin{align}\label{eq eta12}
	\|\eta\|_{L^2(n,n+2;H^1(\omega))}\leq C\int_{n}^{n+2}\|\nabla\eta_t\|^2dt\leq C\left(1+\|\eta_{n}\|^2+\sum_{j\in\N^+}b_j\int_{n}^{n+2}\langle\eta_t,\psi_j\rangle d\beta_j(t)\right).
	\end{align}
Denote the martingale in the right-hand side by $M_{n,t}$, $n\leq t\leq n+2$. Then one has
	\begin{align*}
	\langle M_n\rangle_{n+2}=\sum_{j\in\N^+}b_j^2\int_{n}^{n+2}\langle\eta_t,\psi_j\rangle^2dt\leq C\|\eta\|_{L^2(n,n+2;L^2(\omega))}^2\leq C\|\eta\|_{L^2(n,n+2;H^1(\omega))}^2.
	\end{align*}
By \eqref{eq eta12}, for any $\gamma>0$,
	\begin{equation*}
	\exp(\gamma\|\eta\|_{L^2(n,n+2;H^1(\omega))}^2)\leq C\exp\left(C\gamma\|\eta_n\|^2+C\gamma M_{n,n+2}\right).
	\end{equation*}
Since $\langle M_n\rangle_{n+2}\leq C\|\eta\|_{L^2(n,n+2;H^1(\omega))}^2$, choosing $\gamma_3>0$ sufficiently small, for any $\gamma\in(0,\gamma_3]$, 
	\begin{align*}
	&\E\left(\exp(\gamma\|\eta\|_{L^2(n,n+2;H^1(\omega))}^2)|\mathcal{F}_{n}\right)\\
	&\leq C\exp\left(C\gamma\|\eta_n\|^2\right)\E\left(e^{C\gamma M_{n,n+2}}|\mathcal{F}_{n}\right)\\
	&\leq C\exp\left(C\gamma\|\eta_n\|^2\right)\E\left(e^{2C\gamma M_{n,n+2}-2C^2\gamma^2\langle M_n\rangle_{n+2}}|\mathcal{F}_{n}\right)^{1/2}\E\left(e^{2C^2\gamma^2\langle M_n\rangle_{n+2}}|\mathcal{F}_{n}\right)^{1/2}\\
	&\leq C\exp\left(C\gamma\|\eta_n\|^2\right)\E\left(\exp(C\gamma^2\|\eta\|_{L^2(n,n+2;H^1(\omega))}^2|\mathcal{F}_{n}\right)^{1/2}\\
	&\leq C\exp\left(C\gamma\|\eta_n\|^2\right)\E\left(\exp(\gamma\|\eta\|_{L^2(n,n+2;H^1(\omega))}^2|\mathcal{F}_{n}\right)^{1/2}.
	\end{align*}
which implies
	\begin{equation}\label{eq eta13}
	\E\left(\exp(\gamma\|\eta\|_{L^2(n,n+2;H^1(\omega))}^2|\mathcal{F}_{n}\right)\leq C\exp\left(C\gamma\|\eta_{n}\|^2\right).
	\end{equation}
On the other hand, by \eqref{eq eta12} and the Burkholder--Davis--Gundy inequality, for any $m\in\N^+$,
	\begin{align*}
	\E\left(\|\eta\|_{L^2(n,n+2;H^1(\omega))}^{2m}|\mathcal{F}_{n}\right)&\leq C_m\left(1+\|\eta_{n}\|^{2m}\right)+C_m\E\left(\left(\sum_{j\in\N^+}b_j^2\int_{n}^{n+2}\langle\eta_t,\psi_j\rangle^2dt\right)^{m/2}|\mathcal{F}_{n}\right)\\
	&\leq C_m\left(1+\|\eta_{n}\|^{2m}\right)+C_m\E\left(\|\eta\|_{L^2(n,n+2;H^1(\omega))}^m|\mathcal{F}_{n}\right).
	\end{align*}
Using Young's inequality to absorb the last term into the left-hand side, we get
	\begin{equation}\label{eq eta14}
	\E\left(\|\eta\|_{L^2(n,n+2;H^1(\omega))}^{2m}|\mathcal{F}_{n}\right)\leq C_m\left(1+\|\eta_{n}\|^{2m}\right).
	\end{equation}

\vspace{0.6em}

\noindent {\it Point (3).} Applying Lemma \ref{NS estimate}, there exists a deterministic constant $C>0$ such that
	\begin{equation*}
	\|u\|_{L^\infty((n+1,n+2)\times D)}+\|u\|_{L^2(n,n+2;\mathcal H^2)}\leq C\left(\|u_{n}\|_{\mathcal H^1}+\|\overline{\eta}\|_{L^2(n,n+2;H^1(D))}\right)\quad\forall\,n\in\N.
	\end{equation*}
Combining this with \eqref{eq eta11} and \eqref{eq eta13}, we obtain
	\begin{align*}
	\E\left(\exp\left(\gamma\|u\|_{L^\infty((n+1,n+2)\times D)}^2\right)|\mathcal{F}_{n}\right)&\leq\exp\left(C\gamma\|u_{n}\|_{\mathcal H^1}^2\right)\E\left(\exp\left(C\gamma\|\eta\|_{L^2(n,n+2;H^1(\omega))}^2\right)|\mathcal{F}_{n}\right)\\
	&\leq C\exp\left(c\gamma\|U_{n}\|_{\mathcal X}^2\right),
	\end{align*}
after reducing $\gamma_0$ if necessary. This proves \eqref{eq L-inf bdd}. The proof of \eqref{eq L2H2 bdd} is similar.

\vspace{0.6em}

\noindent {\it Point (4).} By Lemma \ref{NS estimate}, there exists a deterministic constant $C>0$ such that
	\begin{equation*}
	\|\partial_tu\|_{L^2(n+1,n+2;L^4(D))}\leq C\left(1+\|u_n\|_{\mathcal H^1}^2+\|\overline{\eta}\|_{L^2(n,n+2;H^1(D))}^2\right)\quad\forall\,n\in\N.
	\end{equation*}
Using \eqref{eq eta11} and \eqref{eq eta14}, we derive that
	\begin{align*}
	\E\left(\|\partial_tu\|_{L^2(n+1,n+2;L^4(D))}^m|\mathcal{F}_{n}\right)&\leq C_m\left(1+\|u_n\|_{\mathcal H^1}^{2m}\right)+C_m\E\left(\|\eta\|_{L^2(n,n+2;H^1(\omega))}^{2m}|\mathcal{F}_{n}\right)\\
	&\leq C_m(1+\|U_n\|_{\mathcal X}^{2m}).
	\end{align*}
This proves Point (4) and completes the proof of the lemma.

\end{proof}

\begin{lemma}\label{lemma J1}
For any $\gamma>0$, there exists a constant $C_\gamma>0$ such that for any $t\geq s\geq0$ and $V_s\in \mathcal X$, the solution $(v,p)=\mathcal J_{s,\cdot}V_s$ of \eqref{eq J} satisfies
	\begin{align}\label{eq J1}
	\sup_{s\leq r\leq t}\|\mathcal J_{s,r}V_s\|_{\mathcal X}^2+\int_s^t\left(\|v_r\|_{\mathcal H^2}^2+\|\nabla p_r\|^2\right)dr\leq C_\gamma\exp\left(\gamma\int_s^t\|u_r\|_{\mathcal H^2}^2dr+C_\gamma(t-s)\right)\|V_s\|_{\mathcal X}^2.
	\end{align}
\end{lemma}

\begin{proof}
Let $\phi=\curl v$ and $\varpi=\curl u$. The heat component satisfies
	\begin{equation*}
	\frac{1}{2}\partial_r\|p\|^2+\|\nabla p\|^2=0.
	\end{equation*}
Taking curl in the velocity equation in \eqref{eq J}, we obtain
	\begin{equation*}
	\partial_r\phi+u\cdot\nabla\phi+v\cdot\nabla\varpi-\nu\Delta\phi=\curl\overline p,\quad\phi|_{\partial D}=0,\quad\phi(s)=\curl v_s.
	\end{equation*}
Multiplying by $\phi$ gives
	\begin{align*}
	\frac{1}{2}\partial_r\|\phi\|^2+\nu\|\nabla\phi\|^2&=-\int_D(v\cdot\nabla\varpi)\phi dx+\langle\curl\overline p,\phi\rangle\leq-\int_D(v\cdot\nabla\varpi)\phi dx+\frac{\nu}{4}\|\nabla\phi\|^2+C\|p\|^2.
	\end{align*}
Using div-curl estimate and the Gagliardo--Nirenberg inequality, one has
	\begin{align*}
	\left|\int_D(v\cdot\nabla\varpi)\phi dx\right|&\leq C\|v\|_{L^4(D)}\|\nabla\varpi\|\|\phi\|_{L^4(D)}\leq C\|u\|_{\mathcal H^2}\|\phi\|^{3/2}\|\nabla\phi\|^{1/2}\\
	&\leq\frac{\nu}{4}\|\nabla\phi\|^2+\left(\gamma\|u\|_{\mathcal H^2}^2+C_\gamma\right)\|\phi\|^2.
	\end{align*}
Combining the last estimates with the heat energy identity, we get
	\begin{equation*}
	\partial_r\left(\|\phi\|^2+\|p\|^2\right)+c\|\nabla\phi\|^2+\|\nabla p\|^2\leq\left(\gamma\|u\|_{\mathcal H^2}^2+C_\gamma\right)\left(\|\phi\|^2+\|p\|^2\right).
	\end{equation*}
By Gronwall's inequality, it follows that
	\begin{align*}
	&\sup_{s\leq r\leq t}\left(\|\phi_r\|^2+\|p_r\|^2\right)+\int_s^t\left(\|\nabla\phi_r\|^2+\|\nabla p_r\|^2\right)dr\\
	&\quad\leq C_\gamma\exp\left(\gamma\int_s^t\|u_r\|_{\mathcal H^2}^2dr+C_\gamma(t-s)\right)\left(\|\phi_s\|^2+\|p_s\|^2\right).
	\end{align*}
Meanwhile, note that
	\begin{align*}
	\int_s^t\|v_r\|_{\mathcal H^2}^2dr&\leq C\int_s^t\|\nabla\phi_r\|^2dr+C(t-s)\sup_{s\leq r\leq t}\|\phi_r\|^2,\\
	\|\mathcal J_{s,r}V_s\|_{\mathcal X}^2&\leq C\left(\|\phi_r\|^2+\|p_r\|^2\right),\quad\|\phi_s\|^2+\|p_s\|^2\leq C\|V_s\|_{\mathcal X}^2.
	\end{align*}
   This completes the proof of \eqref{eq J1}.
\end{proof}

\vspace{0.6em}

We also consider the equation solved by the second derivative of the flow. Let $V_r^i\in \mathcal X$, $i=1,2$, and write $(v^{(i)},p^{(i)})=\mathcal J_{s,\cdot}V_s^i$ for $i=1,2$. Then $\mathcal J^{(2)}_{s,t}(V_s^1,V_s^2)=(w_t,0)$, where $w$ solves
	\begin{equation}\label{eq J-2}
	\begin{cases}
	\partial_tw-\nu\Delta w+(u\cdot\nabla)w+(w\cdot\nabla)u+\nabla\pi^{(2)}=-(v^{(1)}\cdot\nabla)v^{(2)}-(v^{(2)}\cdot\nabla)v^{(1)},\quad x\in D,\;t>s,\\
	\dvg w=0,\\
	w\cdot n=0,\quad\curl w=0,\quad x\in\partial D,\\
	w(s,\cdot)=0.
	\end{cases}
	\end{equation}

\begin{lemma}
For every $\gamma>0$, there exists a constant $C_\gamma>0$ such that for almost every realization of the noise, the bilinear map $\mathcal J^{(2)}$ defined by \eqref{eq J-2} satisfies
	\begin{equation}\label{eq J2}
	\|\mathcal J^{(2)}_{s,t}\|_{\mathcal L(\mathcal X\times \mathcal X;\mathcal X)}\leq C_\gamma\exp\left(\gamma\int_s^t\|u_r\|_{\mathcal H^2}^2dr+C_\gamma(t-s)\right),\quad t\geq s\geq0.
	\end{equation}
\end{lemma}
\begin{proof}
Let $\varpi=\curl u$, $\phi^{(i)}=\curl v^{(i)}$ and $\vartheta=\curl w$. Taking curl in \eqref{eq J-2}, we have
	\begin{equation*}
	\partial_r\vartheta-\nu\Delta\vartheta+u\cdot\nabla\vartheta+w\cdot\nabla\varpi=-v^{(1)}\cdot\nabla\phi^{(2)}-v^{(2)}\cdot\nabla\phi^{(1)},\quad\vartheta|_{\partial D}=0,\quad\vartheta(s)=0.
	\end{equation*}
Multiplying by $\vartheta$, we obtain
	\begin{align*}
	\frac{1}{2}\partial_r\|\vartheta\|^2+\nu\|\nabla\vartheta\|^2&\leq\left|\int_D(w\cdot\nabla\varpi)\vartheta dx\right|+\left|\int_D(v^{(1)}\cdot\nabla\phi^{(2)})\vartheta dx\right|+\left|\int_D(v^{(2)}\cdot\nabla\phi^{(1)})\vartheta dx\right|.
	\end{align*}
The first term is bounded by
	\begin{equation*}
	\left|\int_D(w\cdot\nabla\varpi)\vartheta dx\right|\leq\frac{\nu}{8}\|\nabla\vartheta\|^2+\left(\frac{\gamma}{4}\|u\|_{\mathcal H^2}^2+C_\gamma\right)\|\vartheta\|^2.
	\end{equation*}
Using $\dvg v^{(1)}=0$ and $v^{(1)}\cdot n=0$,
	\begin{align*}
	\left|\int_D(v^{(1)}\cdot\nabla\phi^{(2)})\vartheta dx\right|&=\left|\int_D(v^{(1)}\cdot\nabla\vartheta)\phi^{(2)}dx\right|\leq C\|v^{(1)}\|_{\mathcal H^1}^{1/2}\|v^{(1)}\|_{\mathcal H^2}^{1/2}\|v^{(2)}\|_{\mathcal H^1}\|\nabla\vartheta\|\\
	&\leq\frac{\nu}{8}\|\nabla\vartheta\|^2+C\|v^{(1)}\|_{\mathcal H^1}\|v^{(1)}\|_{\mathcal H^2}\|v^{(2)}\|_{\mathcal H^1}^2.
	\end{align*}
    The term with $v^{(2)}\cdot\nabla\phi^{(1)}$ is treated in the same way. Hence
	\begin{align*}
	\partial_r\|\vartheta\|^2+\nu\|\nabla\vartheta\|^2&\leq\left(\gamma\|u\|_{\mathcal H^2}^2+C_\gamma\right)\|\vartheta\|^2+C\|v^{(1)}\|_{\mathcal H^1}\|v^{(1)}\|_{\mathcal H^2}\|v^{(2)}\|_{\mathcal H^1}^2\\
	&\quad+C\|v^{(2)}\|_{\mathcal H^1}\|v^{(2)}\|_{\mathcal H^2}\|v^{(1)}\|_{\mathcal H^1}^2.
	\end{align*}
    
    Let us use \eqref{eq J1} with $\gamma$ replaced by $\gamma/8$ and set
	\begin{equation*}
	Q_{s,t}:=C_\gamma\exp\left(\tfrac{\gamma}{8}\|u\|_{L^2(s,t;\mathcal H^2)}^2+C_\gamma(t-s)\right).
	\end{equation*}
    Then, for $i=1,2$,
	\begin{equation*}
	\|v^{(i)}\|_{L^\infty(s,t;\mathcal H^1)}^2+\|v^{(i)}\|_{L^2(s,t;\mathcal H^2)}^2\leq Q_{s,t}\|V_s^i\|_{\mathcal X}^2.
	\end{equation*}
    
    Consequently,
	\begin{align*}
	\int_s^t\|v^{(1)}_r\|_{\mathcal H^1}\|v^{(1)}_r\|_{\mathcal H^2}\|v^{(2)}_r\|_{\mathcal H^1}^2dr&\leq(t-s)^{1/2}\|v^{(1)}\|_{L^\infty(s,t;\mathcal H^1)}\|v^{(1)}\|_{L^2(s,t;\mathcal H^2)}\|v^{(2)}\|_{L^\infty(s,t;\mathcal H^1)}^2\\
	&\leq C_\gamma Q_{s,t}^2\|V_s^1\|_{\mathcal X}^2\|V_s^2\|_{\mathcal X}^2,
	\end{align*}
    where the factor $(t-s)^{1/2}$ is absorbed into $C_\gamma e^{C_\gamma(t-s)}$. The symmetric term satisfies the same bound. Since $\vartheta(s)=0$, Gronwall's inequality gives
	\begin{equation*}
	\|\vartheta_t\|^2\leq C_\gamma\exp\left(\frac{\gamma}{4}\int_s^t\|u_r\|_{\mathcal H^2}^2dr+C_\gamma(t-s)\right)Q_{s,t}^2\|V_s^1\|_{\mathcal X}^2\|V_s^2\|_{\mathcal X}^2.
	\end{equation*}
    Using the div-curl estimate for $w$ and enlarging $C_\gamma$, we obtain
	\begin{equation*}
	\|w_t\|_{\mathcal H^1}\leq C_\gamma\exp\left(\gamma\int_s^t\|u_r\|_{\mathcal H^2}^2dr+C_\gamma(t-s)\right)\|V_s^1\|_{\mathcal X}\|V_s^2\|_{\mathcal X}.
	\end{equation*}
    Taking the supremum over $\|V_s^1\|_{\mathcal X}\leq1$ and $\|V_s^2\|_{\mathcal X}\leq1$ gives \eqref{eq J2}.
    \end{proof}

    \subsection{Verification of the asymptotic strong Feller}\label{Sec D}

    This appendix proves the estimates used in Section \ref{Sec 4.1}. Throughout this subsection $B_0>0$ is fixed, and the constants may depend on $B_0$ but are independent of $n,N$ and $U_0$. To begin with, recall that $\mathcal{D}$ denotes the Malliavin derivative operator. By the Riesz representation theorem, $\mathcal{D}$ can be formulated as
	\begin{equation*}
	\langle\mathcal{D}F,v\rangle_{H}=\sum_{j\in\N^+}\int_{\R^+}\mathcal{D}_t^jF\langle v(t),\psi_j\rangle dt\quad\text{for }F\in\mathbb D^{1,2}(\mathcal X),\;v\in H.
	\end{equation*}
    In particular, we have 
	\begin{equation*}
	\mathcal{D}_s^jU_t=\mathcal{J}_{s,t}(0,b_j\psi_j),\quad\forall\,0<s<t,\;j\in\N^+.
	\end{equation*}
    Additionally, for any $F\in\mathcal{F}_{s}$,
	\begin{equation*}
	\mathcal{D}_tF=0\quad a.s.\quad\forall\,t>s.
	\end{equation*}

 \subsubsection{ Proof of inequality \eqref{eq rho2} }  
    We first show that, under the assumptions of Lemma \ref{lemma ASF1}, for any $\gamma,\varepsilon>0$, there exist constants $\bar{\gamma}=\bar{\gamma}(\varepsilon,\gamma)>0$, $N\in\N^+$ and $\beta>0$ such that
	\begin{equation}\label{eq ASF11}
	\E\left(e^{\bar{\gamma}\|U_{2n}\|_{\mathcal X}^2}\|\rho_{2n}\|_{\mathcal X}^2\right)\leq\varepsilon^n\exp\left(\gamma\|U_0\|_{\mathcal X}^2\right)\|V_0\|_{\mathcal X}^2\quad\forall\,U_0,V_0\in \mathcal X,\;n\in\N.
	\end{equation}

    \vspace{0.3em}

    By homogeneity, we assume $\rho_0=V_0$ with $\|V_0\|_{\mathcal X}=1$. Let $\gamma>0$ and $\varepsilon\in(0,1)$ be given. We choose $\bar{\gamma}>0$ and $\bar{\varepsilon}>0$ later. Define
	\begin{equation*}
	\Phi(x,y):=\begin{cases}
	x/y\quad&\text{for }y\neq0,\\
	0\quad&\text{for }y=0.
	\end{cases}
	\end{equation*}
    Invoking \eqref{eq ASF step}, there exists $N=N(\bar{\gamma},\bar{\varepsilon})$ and $\beta=\beta(\bar{\gamma},\bar{\varepsilon})$ such that
	\begin{equation}\label{eq ASF12}
	\E\left(\|\rho_{2n+2}\|_{\mathcal X}^4|\mathcal F_{2n}\right)\leq\bar{\varepsilon}e^{\bar{\gamma}\|U_{2n}\|_{\mathcal X}^2}\|\rho_{2n}\|_{\mathcal X}^4\quad\forall\,U_0\in \mathcal X,\;n\in\N.
	\end{equation}
    For each $k\in\N^+$, set
	\begin{equation*}
	F_k:=\Phi(\|\rho_{2k}\|_{\mathcal X},\|\rho_{2k-2}\|_{\mathcal X})^2\exp\left(-\tfrac{\bar{\gamma}}{2}\|U_{2k-2}\|_{\mathcal X}^2\right),\quad G_k:=\exp\left(\tfrac{\bar{\gamma}}{2}\|U_{2k-2}\|_{\mathcal X}^2\right).
	\end{equation*}
    Repeated use of conditional expectation with \eqref{eq ASF12} implies
	\begin{align*}
	\E\left(\prod_{1\leq k\leq n}F_k^2\right)&=\E\left(\E\left(\prod_{1\leq k\leq n}F_k^2|\mathcal F_{2n-2}\right)\right)=\E\left(\prod_{1\leq k\leq n-1}F_k^2\E\left(F_n^2|\mathcal F_{2n-2}\right)\right)\\
	&\leq\bar{\varepsilon}\E\left(\prod_{1\leq k\leq n-1}F_k^2\right)\leq\cdots\leq\bar{\varepsilon}^n.
	\end{align*}
    On the other hand, noting that for any $n\in\N$ and $t\geq2n$, $\mathbf{1}_{\{\|\rho_{2n}\|_{\mathcal X}=0\}}\rho_t=0$, we have
	\begin{equation*}
	\prod_{1\leq k\leq n}F_kG_k=\|\rho_{2n}\|_{\mathcal X}^2.
	\end{equation*}
    Consequently, using \eqref{eq L2 Ubdd3}, there exist constants $C_1,c_1,\gamma_0>0$ such that for any $\bar{\gamma}\in(0,\gamma_0]$,
	\begin{align*}
	\E\left(e^{\bar{\gamma}\|U_{2n}\|_{\mathcal X}^2}\|\rho_{2n}\|_{\mathcal X}^2\right)&\leq\left(\E\prod_{1\leq k\leq n}F_k^2\right)^{\frac{1}{2}}\left(\E e^{2\bar{\gamma}\|U_{2n}\|_{\mathcal X}^2+\bar{\gamma}\sum_{1\leq k\leq n}\|U_{2k-2}\|_{\mathcal X}^2}\right)^{\frac{1}{2}}\\
	&\leq\bar{\varepsilon}^{n/2}\exp\left(c_1\bar{\gamma}\|U_0\|_{\mathcal X}^2\right)\exp\left(C_1n\right).
	\end{align*}
    Therefore, the desired inequality \eqref{eq ASF11} follows by taking
	\begin{equation*}
	\bar{\varepsilon}=e^{-2C_1}\varepsilon^2,\quad\bar{\gamma}=(c_1+1)^{-1}(\gamma\wedge\gamma_0).
	\end{equation*}

    \vspace{0.6em}

    We now prove \eqref{eq rho2}. Let $\gamma,\alpha>0$ be arbitrarily given, and apply \eqref{eq ASF11} with $\varepsilon=e^{-4\alpha}$. In view of the definition of $v$ given by \eqref{eq v1},\eqref{eq v2}, it follows that for each $n\in\N$,
	\begin{equation*}
	\rho_t=\begin{cases}
	\mathcal J_{2n,t}\rho_{2n},\quad&2n\leq t\leq2n+1,\\
	\mathcal J_{2n+1,t}\rho_{2n+1}-\mathcal A_{2n+1,t,N}v_n,\quad&2n+1<t\leq2n+2.
	\end{cases}
	\end{equation*}
    Recalling that $K_n$ is given by \eqref{eq ASF Kn} and using spectral theorem, it follows that
	\begin{equation*}
	\|v_n\|_{H_{2n+1,2n+2}}=\|\mathcal A_n^*(\mathcal M_n+\beta)^{-1}\check{\mathcal J}_n\hat{\mathcal J}_n\rho_{2n}\|_{H_{2n+1,2n+2}}\leq\beta^{-1/2}K_n\|\rho_{2n}\|_{\mathcal X}.
	\end{equation*}
    Since $\|\mathcal A_{2n+1,t,N}\|\leq C(B_0)K_n$, we have
	\begin{equation*}
	\sup_{2n\leq t\leq2n+2}\|\rho_t\|_{\mathcal X}\leq C_\beta K_n^2\|\rho_{2n}\|_{\mathcal X}.
	\end{equation*}
    Using \eqref{eq ASF K2} with $m=4$ and $\kappa=\bar{\gamma}$, together with \eqref{eq ASF11}, we obtain
	\begin{align*}
	\E\|\rho_t\|_{\mathcal X}^2&\leq C_\beta\E\left(K_n^4\|\rho_{2n}\|_{\mathcal X}^2\right)\leq C_\beta\E\left(e^{\bar{\gamma}\|U_{2n}\|_{\mathcal X}^2}\|\rho_{2n}\|_{\mathcal X}^2\right)\leq Ce^{-4\alpha n}\exp\left(\gamma\|U_0\|_{\mathcal X}^2\right)\|V_0\|_{\mathcal X}^2.
	\end{align*}
    Noting that $e^{-4\alpha n}\leq Ce^{-\alpha t}$ for $t\in[2n,2n+2]$, this proves the desired inequality \eqref{eq rho2}.

\vspace{0.6em}

 \subsubsection{ Proof of inequality \eqref{eq v3} }
    By homogeneity, it suffices to consider the case $\|V_0\|_{\mathcal X}=1$. To begin with, by the spectral theorem, we have
	\begin{equation}\label{eq ASF13}
	\begin{aligned}
	\|\mathcal A_n^*(\mathcal M_n+\beta)^{-1/2}\|_{\mathcal L(\mathcal X;H_{2n+1,2n+2})}&\leq1,\\
	\|(\mathcal M_n+\beta)^{-1/2}\mathcal A_n\|_{\mathcal L(H_{2n+1,2n+2};\mathcal X)}&\leq1,\\
	\|(\mathcal M_n+\beta)^{-1/2}\|_{\mathcal L(\mathcal X)}&\leq\beta^{-1/2}.
	\end{aligned}
	\end{equation}
    Recall that $v$ is defined by \eqref{eq v1},\eqref{eq v2}. By the generalized It\^o isometry, for any $t>0$,
	\begin{align}\label{eq ASF14}
	\E\left|\int_0^tv(s)dW(s)\right|^2&=\E\|v\|_{H_{0,t}}^2+\E\int_0^t\int_0^t\operatorname{Tr}\left(\mathcal D_sv(r)\mathcal D_rv(s)\right)dsdr\notag\\
	&\leq\E\|v\|_{H}^2+\sum_{n\in\N}\sum_{j\in\N^+}\int_{2n+1}^{(2n+2)\wedge t}\int_{2n+1}^{(2n+2)\wedge t}\E\|\mathcal D_s^jv(r)\|_{L^2(\omega)}^2dsdr\notag\\
	&\leq\sum_{n\in\N}\E\|v_n\|_{H}^2+\sum_{n\in\N}\sum_{j\in\N^+}\int_{2n+1}^{2n+2}\E\|\mathcal D_s^jv_n\|_{H}^2ds.
	\end{align}
    Invoking \eqref{eq ASF13} and the definition of $K_n$, we compute that
	\begin{equation*}
	\begin{aligned}
	\|v_n\|_{H}=\|\mathcal A_n^*(\mathcal M_n+\beta)^{-1}\check{\mathcal J}_n\hat{\mathcal J}_n\rho_{2n}\|_{H}\leq\beta^{-1/2}\|\check{\mathcal J}_n\hat{\mathcal J}_n\rho_{2n}\|_{\mathcal X}\leq\beta^{-1/2}K_n\|\rho_{2n}\|_{\mathcal X},
	\end{aligned}
	\end{equation*}
    and therefore, using \eqref{eq ASF K2} with $m=2$ and $\kappa=\bar{\gamma}/2$,
	\begin{align}\label{eq ASF15}
	\sum_{n\in\N}\E\|v_n\|_{H}^2\leq C_\beta\sum_{n\in\N}\E\left(K_n^2\|\rho_{2n}\|_{\mathcal X}^2\right)\leq C_\beta\sum_{n\in\N}\E\left(e^{\frac{\bar{\gamma}}{2}\|U_{2n}\|_{\mathcal X}^2}\|\rho_{2n}\|_{\mathcal X}^2\right)\leq C\exp\left(\gamma\|U_0\|_{\mathcal X}^2\right).
	\end{align}

    \vspace{0.3em}

    It remains to estimate the Malliavin derivative of $v_n$. For $s\in(2n+1,2n+2)$, the random variable $\rho_{2n+1}=\hat{\mathcal J}_n\rho_{2n}$ is $\mathcal F_{2n+1}$-measurable, hence $\mathcal D_s^j\rho_{2n+1}=0$. By the product rule,
	\begin{align*}
	\mathcal D_s^jv_n&=\mathcal A_n^*(\mathcal M_n+\beta)^{-1}(\mathcal D_s^j\check{\mathcal J}_n)\rho_{2n+1}+(\mathcal D_s^j\mathcal A_n^*)(\mathcal M_n+\beta)^{-1}\check{\mathcal J}_n\rho_{2n+1}\\
	&\quad-\mathcal A_n^*(\mathcal M_n+\beta)^{-1}((\mathcal D_s^j\mathcal A_n)\mathcal A_n^*+\mathcal A_n(\mathcal D_s^j\mathcal A_n^*))(\mathcal M_n+\beta)^{-1}\check{\mathcal J}_n\rho_{2n+1}.
	\end{align*}
    We now estimate the terms appearing above. Fix $\theta>0$, to be chosen below, and set
	\begin{equation*}
	Q_n(\theta):=\exp\left(\theta\int_{2n}^{2n+2}\|u_r\|_{\mathcal H^2}^2dr\right).
	\end{equation*}
    By Lemma \ref{lemma J1}, one has $K_n\leq C_\theta Q_n(\theta)$. Moreover, for $r,s\in(2n+1,2n+2)$ and $V\in \mathcal X$,
	\begin{equation*}
	\mathcal D_s^j\mathcal J_{r,2n+2}V=\begin{cases}
	\mathcal J^{(2)}_{s,2n+2}\left((0,b_j\psi_j),\mathcal J_{r,s}V\right),\quad&r\leq s,\\
	\mathcal J^{(2)}_{r,2n+2}\left(\mathcal J_{s,r}(0,b_j\psi_j),V\right),\quad&s<r.
	\end{cases}
	\end{equation*}
    Using \eqref{eq J1} and \eqref{eq J2} on subintervals of $[2n,2n+2]$, we obtain
	\begin{equation*}
	\|\mathcal D_s^j\check{\mathcal J}_n\|_{\mathcal L(\mathcal X)}\leq C_\theta|b_j|Q_n(\theta)^2.
	\end{equation*}
    Similarly, for $h\in H_{2n+1,2n+2}$,
	\begin{equation*}
	\mathcal D_s^j\mathcal A_nh=\int_{2n+1}^{2n+2}\mathcal D_s^j\mathcal J_{r,2n+2}\mathcal B\mathsf P_Nh(r)dr,
	\end{equation*}
    and therefore
	\begin{equation*}
	\|\mathcal D_s^j\mathcal A_n\|_{\mathcal L(H_{2n+1,2n+2};\mathcal X)}+\|\mathcal D_s^j\mathcal A_n^*\|_{\mathcal L(\mathcal X;H_{2n+1,2n+2})}\leq C_\theta|b_j|Q_n(\theta)^2.
	\end{equation*}
    Since $\|\mathcal A_n\|+\|\mathcal A_n^*\|\leq C_\theta Q_n(\theta)$ and $\|\check{\mathcal J}_n\rho_{2n+1}\|_{\mathcal X}+\|\rho_{2n+1}\|_{\mathcal X}\leq C_\theta Q_n(\theta)^2\|\rho_{2n}\|_{\mathcal X}$, the preceding estimates and \eqref{eq ASF13} imply
	\begin{equation*}
	\|\mathcal D_s^jv_n\|_{H}\leq C_{\beta,\theta}|b_j|Q_n(\theta)^5\|\rho_{2n}\|_{\mathcal X}.
	\end{equation*}
    Hence
	\begin{equation*}
	\sum_{j\in\N^+}\int_{2n+1}^{2n+2}\|\mathcal D_s^jv_n\|_{H}^2ds\leq C_{\beta,\theta}Q_n(\theta)^{10}\|\rho_{2n}\|_{\mathcal X}^2.
	\end{equation*}
    Finally, by \eqref{eq L2H2 bdd}, for every $m\geq1$ and $\kappa>0$, choosing $\theta>0$ sufficiently small gives
	\begin{equation*}
	\E\left(Q_n(\theta)^m|\mathcal F_{2n}\right)\leq C\exp\left(\kappa\|U_{2n}\|_{\mathcal X}^2\right).
	\end{equation*}
    Applying this with $m=10$ and $\kappa=\bar{\gamma}/2$, and using \eqref{eq ASF11}, we obtain
	\begin{align}\label{eq ASF16}
	\sum_{n\in\N}\sum_{j\in\N^+}\int_{2n+1}^{2n+2}\E\|\mathcal D_s^jv_n\|_{H}^2ds&\leq C_{\beta,\theta}\sum_{n\in\N}\E\left(Q_n(\theta)^{10}\|\rho_{2n}\|_{\mathcal X}^2\right)\leq C_{\beta,\theta}\sum_{n\in\N}\E\left(e^{\frac{\bar{\gamma}}{2}\|U_{2n}\|_{\mathcal X}^2}\|\rho_{2n}\|_{\mathcal X}^2\right)\notag\\
	&\leq C\exp\left(\gamma\|U_0\|_{\mathcal X}^2\right).
	\end{align}
    Combining \eqref{eq ASF14}, \eqref{eq ASF15}, and \eqref{eq ASF16}, we get
	\begin{equation*}
	\sup_{t\geq0}\E\left|\int_0^tv(s)dW(s)\right|^2\leq C\exp\left(\gamma\|U_0\|_{\mathcal X}^2\right),
	\end{equation*}
    which proves \eqref{eq v3}. This completes the proof of Lemma \ref{lemma ASF1}.

\subsection{Auxiliary proofs for control problems}

In this appendix, we prove two auxiliary estimates used in the stabilization analysis of Section \ref{Sec 3}.

    \begin{lemma}\label{lemma theta}
    For any $g$ satisfying \eqref{eq g}, the weight $\Theta_g$ defined by \eqref{eq weight} extends by zero to a function in $C^\infty([0,1])$ and satisfies
	\begin{equation}\label{eq theta1}
	\partial_r^k\Theta_g(0)=\partial_r^k\Theta_g(1)=0,\quad k\in\N.
	\end{equation}
    Moreover, there exists a constant $C>0$, depending only on $D$, $\omega$ and $\nu$, such that
	\begin{equation}\label{eq theta2}
	|\Theta_g'(r)|^2\leq Cs_g^3\lambda_g^4e^{-2s_g\alpha(x,r)}\xi(x,r)^3,\quad(x,r)\in Q.
	\end{equation}
    \end{lemma}

    \begin{proof}
    Set $\ell(r)=r^4(1-r)^4$ and $E=e^{\lambda_gK}$. Since the minimum and maximum of $\vartheta$ on $\overline D$ are $0$ and $K$, respectively, the definitions of the weights give
	\begin{equation*}
	\alpha_*(r)=\frac{E^2-E}{\ell(r)},\quad\alpha^*(r)=\frac{E^2-1}{\ell(r)},\quad\xi^*(r)=\frac{E}{\ell(r)}.
	\end{equation*}
    Substituting these expressions into \eqref{eq weight}, we obtain
	\begin{equation*}
	\Theta_g(r)=s_g^{15/2}\lambda_g^8E^{15/2}\ell(r)^{-15/2}\exp\left(-\frac{2s_g(E-1)^2}{\ell(r)}\right),\quad r\in(0,1).
	\end{equation*}
    The exponential factor decays faster than any negative power of $\ell(r)$ as $r$ tends to $0$ or $1$. Hence this expression extends by zero to a smooth function on $[0,1]$, and all its derivatives vanish at the endpoints. This proves \eqref{eq theta1}.

    Differentiating the preceding expression, we have
	\begin{equation*}
	\Theta_g'(r)=\Theta_g(r)\frac{\ell'(r)}{\ell(r)}\left(\frac{2s_g(E-1)^2}{\ell(r)}-\frac{15}{2}\right).
	\end{equation*}
    For $x\in\overline D$, set
	\begin{equation*}
	B(x)=4(E-1)^2-2\left(E^2-e^{\lambda_g\vartheta(x)}\right).
	\end{equation*}
    Since $|\ell'|^2\leq C\ell^{3/2}$ and $e^{\lambda_g\vartheta(x)}\geq1$, the quotient of $|\Theta_g'|^2$ by $s_g^3\lambda_g^4e^{-2s_g\alpha}\xi^3$ is bounded by $R_1+R_2$, where
	\begin{align*}
	R_1&\leq Cs_g^{14}\lambda_g^{12}E^{15}(E-1)^4\ell(r)^{-29/2}\exp\left(-\frac{s_gB(x)}{\ell(r)}\right),\\
	R_2&\leq Cs_g^{12}\lambda_g^{12}E^{15}\ell(r)^{-25/2}\exp\left(-\frac{s_gB(x)}{\ell(r)}\right).
	\end{align*}
    After increasing $\widehat\lambda$ if necessary, we have
	\begin{equation*}
	B(x)=2E^2-8E+4+2e^{\lambda_g\vartheta(x)}\geq cE^2,\quad x\in\overline D.
	\end{equation*}
    We now use
	\begin{equation*}
	\sup_{\rho>0}\rho^{-p}e^{-a/\rho}\leq C_pa^{-p},\quad a>0,
	\end{equation*}
    with $a=cs_gE^2$. Since $(E-1)^4\leq E^4$ and $s_g=\widehat sE^8$, it follows that
	\begin{align*}
	R_1\leq C\lambda_g^{12}s_g^{14}E^{19}(s_gE^2)^{-29/2}\leq C\lambda_g^{12}E^{-14},\quad R_2\leq C\lambda_g^{12}s_g^{12}E^{15}(s_gE^2)^{-25/2}\leq C\lambda_g^{12}E^{-14}.
	\end{align*}
    Finally, since $E=e^{\lambda_gK}$, the quantity $\lambda_g^{12}E^{-14}$ is uniformly bounded for $\lambda_g\geq\widehat\lambda$. This proves \eqref{eq theta2}.
    \end{proof}

    \begin{lemma}\label{lemma PN}
    There exists a constant $C>0$, depending only on $\omega$, such that for any $N\in\N^+$ and $f\in H^1(\omega;\R^2)$,
	\begin{equation}\label{eq PN}
	\|(I-\mathsf P_N)f\|_{L^2(\omega)}\leq C\lambda_{N+1}^{-1/4}\|f\|_{H^1(\omega)},
	\end{equation}
    where $\mathsf P_N$ denotes the orthogonal projection in $L^2(\omega;\R^2)$ onto ${\rm span}\{\psi_j:1\leq j\leq N\}$.
    \end{lemma}

    \begin{proof}
    Set $d(x)=\dist(x,\partial\omega)$. Since $\partial\omega$ is smooth, there exists $\varepsilon_0\in(0,1)$ such that for any $\varepsilon\in(0,\varepsilon_0)$, we can choose $\chi_\varepsilon\in C_c^\infty(\omega)$ satisfying
	\begin{equation*}
	\chi_\varepsilon=0\quad{\rm if}\ d\leq\varepsilon,\quad\chi_\varepsilon=1\quad{\rm if}\ d\geq2\varepsilon,\quad|\nabla\chi_\varepsilon|\leq C\varepsilon^{-1}.
	\end{equation*}
    Normal coordinates in a boundary collar and the trace theorem give
	\begin{align*}
	\int_{\{d<2\varepsilon\}}|f|^2dx\leq C\varepsilon\|f\|_{L^2(\partial\omega)}^2+C\varepsilon^2\|\nabla f\|_{L^2(\{d<2\varepsilon\})}^2\leq C\varepsilon\|f\|_{H^1(\omega)}^2.
	\end{align*}
    Set $f_\varepsilon=\chi_\varepsilon f$. Then $f_\varepsilon\in H_0^1(\omega;\R^2)$ and
	\begin{align*}
	\|f-f_\varepsilon\|_{L^2(\omega)}\leq C\varepsilon^{1/2}\|f\|_{H^1(\omega)},\quad\|\nabla f_\varepsilon\|_{L^2(\omega)}\leq C\varepsilon^{-1/2}\|f\|_{H^1(\omega)}.
	\end{align*}
    Moreover, the spectral theorem yields
	\begin{align*}
	\|(I-\mathsf P_N)f_\varepsilon\|_{L^2(\omega)}^2&=\sum_{j>N}|\langle f_\varepsilon,\psi_j\rangle_{L^2(\omega)}|^2\leq\lambda_{N+1}^{-1}\sum_{j\in\N^+}\lambda_j|\langle f_\varepsilon,\psi_j\rangle_{L^2(\omega)}|^2=\lambda_{N+1}^{-1}\|\nabla f_\varepsilon\|_{L^2(\omega)}^2.
	\end{align*}
    Consequently,
	\begin{equation*}
	\|(I-\mathsf P_N)f\|_{L^2(\omega)}\leq C\left(\varepsilon^{1/2}+\lambda_{N+1}^{-1/2}\varepsilon^{-1/2}\right)\|f\|_{H^1(\omega)}.
	\end{equation*}
    Fix $c_0\in(0,\varepsilon_0\lambda_1^{1/2})$ and take $\varepsilon=c_0\lambda_{N+1}^{-1/2}$. Then $\varepsilon\in(0,\varepsilon_0)$ for every $N\in\N^+$, and the preceding estimate gives \eqref{eq PN}.
    \end{proof}

    \subsection{Index of symbols}

In this appendix, we collect the most used symbols of the paper.

{\footnotesize
\begin{longtable}{@{}r@{\hspace{0.8em}}|@{\hspace{0.8em}}p{0.72\textwidth}@{}}
    \hline
    Basic notation & Meaning \\
    \hline
    $D$, $\partial D$, $\omega$ & 2D domain, its smooth boundary, and the localized forcing region $\omega\Subset D$ \\
    $n$ & outward unit normal vector on $\partial D$ \\
    $\|\cdot\|$, $\langle\cdot,\cdot\rangle$ & $L^2$-norm and inner product on $D$ or $\omega$, depending on the context \\
    $\mathcal L(X;Y)$, $\mathcal L(X)$ & bounded linear maps between Banach spaces, with $\mathcal L(X)=\mathcal L(X;X)$ \\
    $B_X(x,r)$ & closed ball in the Banach space $X$ centered at $x$ with radius $r$ \\
    $\mathcal H$, $\mathcal H^1$ & $\mathcal H=\left\{u\in L^2(D;\R^2):\dvg u=0\ \text{in }D,\;u\cdot n=0\ \text{on }\partial D\right\}$, $\mathcal H^1=H^1(D;\R^2)\cap \mathcal H$\\
    $\mathcal H^2$ & $H^2(D;\R^2)\cap \mathcal H^1$ with $\curl u=0$ on $\partial D$ \\
    $(\mathcal X,\|\cdot\|_{\mathcal X})$ & Hilbert phase space $\mathcal X=\mathcal H^1\times L^2(\omega;\R^2)$ with norm $\|(u,\eta)\|_{\mathcal X}^2=\|u\|_{\mathcal H^1}^2+\|\eta\|^2$ \\
    $\widehat{\mathcal X}$, $\|\cdot\|_{\widehat{\mathcal X}}$ & equivalent Hilbert structure of $\mathcal X$; $\|(u,\eta)\|_{\widehat{\mathcal X}}^2=\|\curl u\|^2+K\|\eta\|^2$ \\
    $\Pi$ & Leray projector from $L^2(D;\R^2)$ onto $\mathcal H$ \\
    $A$ & slip Stokes operator $Au=-\nu\Pi\Delta u$ \\
    $B(u,v)$ & Navier--Stokes bilinear term $\Pi((u\cdot\nabla)v)$ \\
    $-\Delta_\omega$ & Dirichlet Laplacian on $L^2(\omega;\R^2)$ \\
    $\{\psi_j\}_{j\in\N^+}$, $\{\lambda_j\}_{j\in\N^+}$ & eigenvectors and eigenvalues of $-\Delta_\omega$ \\
    $\mathsf P_N$ & orthogonal projection in $L^2(\omega;\R^2)$ onto ${\rm span}\{\psi_j:1\leq j\leq N\}$ \\
    $C$, $C_\gamma$ & generic positive constants, with $C_\gamma$ allowed to depend on $\gamma$ \\
    \hline
    Stochastic system & \\
    \hline
    $U(t,U_0,W)$, $U_t,(u_t,\eta_t)$ & solution of \eqref{eq NS},\eqref{eq eta} \\
    $u$, $\pi$, $\eta$ & velocity, pressure, and localized heat component \\
    $\eta^{s}$, $\mu_s$ & localized stationary heat process and its one-time law; see \eqref{eq eta_s} \\
    $\{\hat\beta_j\}_{j\in\N^+}$, $\{\beta_j\}_{j\in\N^+}$ & independent two-sided and one-sided standard Brownian motions, respectively \\
    $\{b_j\}_{j\in\N^+}$, $\widehat b_N$ & real coefficients and $\widehat b_N=\max_{1\leq j\leq N}|b_j|^{-1}$ \\
    $W(t)$ & cylindrical Brownian motion $W(t)=\sum_{j\in\N^+}\beta_j(t)\psi_j$ \\
    $\Lambda$ & diagonal operator $\Lambda h=\sum_{j\in\N^+}b_j\langle h,\psi_j\rangle_{L^2(\omega)}\psi_j$ \\
    $H$, $H_{s,t}$ & $H=L^2(\R^+;L^2(\omega;\R^2))$ and $H_{s,t}=L^2(s,t;L^2(\omega;\R^2))$ \\
    $\mathcal D$, $\mathbb D^{1,2}(\mathcal X)$ & Malliavin derivative and Malliavin--Sobolev space; see \eqref{eq D def} \\
    $P_t$, $P_t^*$ & Markov semigroup and its dual associated with the coupled system \\
    $B_b(\mathcal X)$, $\mathcal B(\mathcal X)$, $\mathcal P(\mathcal X)$ & bounded Borel functions, Borel sets, and Borel probability measures on $\mathcal X$ \\
    $\mu$, $\mu_*$ & invariant probability measures in Theorems \ref{thm 1} and \ref{thm 2}, respectively \\
    $\mathcal O_{\mathcal H^1,\gamma}$, $\mathcal O_{\mathcal X,\gamma}$ & weighted observable classes with norms $\|\cdot\|_{\mathcal H^1,\gamma}$ and $\|\cdot\|_{\mathcal X,\gamma}$; see \eqref{eq mixing1}, \eqref{eq mixing2} \\
    \hline
    Linearized/control operators &  \\
    \hline
    $V_s=(v_s,p_s)$ & initial tangent vector in $\mathcal X$ at time $s$ \\
    $\mathcal J_{s,t}$ & homogeneous linearized flow; see \eqref{eq J} \\
    $\mathcal J_{s,t}^*$ & adjoint of $\mathcal J_{s,t}$ on $\mathcal X$ \\
    $\mathcal J^{(2)}_{s,t}$ & second derivative of the stochastic flow; see \eqref{eq J-2} \\
    $\mathcal B$, $\mathcal B^*$ & noise injection operator $\mathcal Bh=(0,\Lambda h)$ and its adjoint \\
    $\mathcal A_{s,t}$ & Malliavin control map $\mathcal A_{s,t}h=\int_s^t\mathcal J_{r,t}\mathcal Bh(r)dr$ \\
    $\mathcal A_{s,t,N}$ & truncated control map $\mathcal A_{s,t}\mathsf P_N$ \\
    $\mathcal M_{s,t,N}$ & truncated Malliavin matrix $\mathcal A_{s,t,N}\mathcal A_{s,t,N}^*$ \\
    $\mathcal J_{s,t,g}$ & deterministic homogeneous linearized coupled flow with reference velocity $g$ \\
    $\mathcal R_{N,g}$ & deterministic low-frequency control map before inserting $\Lambda$ \\
    $\mathcal A_{N,g}$, $\mathcal G_{N,g}$ & deterministic noise-compatible control map and Gramian $\mathcal A_{N,g}\mathcal A_{N,g}^*$ \\
    
    \hline
    Stabilization weights &  \\
    \hline
    $Q$ & space-time cylinder $D\times(0,1)$ \\
    $g$ & reference velocity in the deterministic linearized equations satisfying \eqref{eq g} \\
    $G_0$, $G_1$ & $G_0=\|g\|_{L^\infty(Q)}$ and $G_1=\|\partial_rg\|_{L^2(0,1;L^4(D))}$ \\
    $\alpha$, $\xi$ & Carleman weights used in Section 3.1 \\
    $\alpha_*$, $\alpha^*$, $\xi^*$ & extrema of the Carleman weights over $\overline D$ \\
    $\mathcal I(s,\lambda;\varphi)$ & weighted Carleman functional for the adjoint velocity $\varphi$ \\
    $\lambda_g$, $s_g$ & Carleman parameters fixed by the reference field $g$ \\
    $\Theta_g$ & time weight in observability; see \eqref{eq weight} \\
    $\mathcal C_g$ & constant depending on the reference field through $G_0$ and $G_1$ \\
    \hline
\end{longtable}
}

\vspace{4mm}

\noindent\textbf{Acknowledgement}\; 
We thank Víctor Hernández-Santamaría, Kévin Le Balc'h and  Liliana Peralta for informing us about their ongoing work, for sharing their perspective on the problem, and for stimulating discussions. Ziyu Liu is partially  supported by Fundamental Research Funds for the Central Universities FRF-TP-26-054. Shengquan Xiang is partially  supported by NSFC 12571474. Zhifei Zhang is partially supported by NSFC 12288101.

    \bibliographystyle{plain}
    \bibliography{References}

    \end{document}